\documentclass[article]{siamltex}
\usepackage{amstext,amsmath,amsfonts,amssymb,amsbsy,enumerate}
\usepackage{comment}
\usepackage{ae}
\usepackage{algorithm}
\usepackage{algorithmic}
\usepackage{xspace}
\usepackage{comment}
\usepackage{graphicx}
\usepackage{psfrag}
\usepackage{color}      % Allows you to create colored text.
\usepackage{array}
\usepackage[normalem]{ulem}
\usepackage{todonotes}

\usepackage[nomain]{glossaries}

\newtheorem{remark}[theorem]{Remark}
\newcommand{\manel}[1]{\textcolor{black}{#1}}
\newcommand{\ab}[1]{\textcolor{black}{#1}}

\newcommand{\pedro}[1]{\textcolor{black}{#1}}
\newcommand{\rhn}[1]{\textcolor{black}{#1}}

\newenvironment{algotab}%
{
\newline\begin{minipage}{\textwidth}
\par\begin{samepage}%
\begin{tabbing}\ttfamily%
 \hspace*{5mm}\=\hspace{3ex}\=\hspace{3ex}\=\hspace{3ex}\=\hspace{3ex}%
\=\hspace{3ex}\=\hspace{3ex}\=\hspace{3ex}\=\hspace{3ex}\kill}%
{\end{tabbing}\end{samepage}
\end{minipage}\\
}

\newcommand{\step}[1]{\noindent\raisebox{1.5pt}[10pt][0pt]{\tiny\framebox{$#1$}}\xspace}

\newcommand{\N}{\ensuremath{\mathbb{N}}}
\renewcommand{\P}{\ensuremath{\mathbb{P}}}

\newcommand{\R}{\ensuremath{\mathbb{R}}}

\newcommand{\ie}{i.\,e.,\xspace}

\newcommand{\definedas}{\mathrel{:=}}

\newcommand{\Cleq}{\ensuremath{\preccurlyeq}}

\newcommand{\Ceq}{\ensuremath{\cong}}

\newcommand{\eps}{\ensuremath{\varepsilon}}

\newcommand{\grids}{\ensuremath{\mathbb{T}}} 
\newcommand{\As}[1][s]{\ensuremath{\mathbb{A}_{#1}}}

\newcommand{\T}{\ensuremath{\mathcal{T}}}
\newcommand{\wT}{ {\ensuremath{\widehat{\mathcal{T}}}} } 
\newcommand{\Tk}[1][k]{\ensuremath{\T_{#1}}} 
\newcommand{\M}{\ensuremath{\mathcal{M}}} 

\newcommand{\RT}{\ensuremath{\mathcal{R}_T}}

\newcommand{\E}{\ensuremath{\mathcal{E}}}

\newcommand{\hT}{\ensuremath{h_{T}}}

\newcommand{\V}{\ensuremath{\mathbb{V}}}

\newcommand{\wnabla}{\ensuremath{\widehat \nabla}}

\DeclareMathOperator{\interior}{int}
\DeclareMathOperator{\osc}{osc}

\DeclareMathOperator{\id}{id}

\newcommand{\cT}{{\T}}

\newcommand{\cR}{{\cal R}}

\newcommand{\cP}{{\cal P}}

\newcommand{\cM}{{\cal M}}

\newcommand{\GREEDY}{\textsf{GREEDY}\xspace}
\newcommand{\SOLVE}{\textsf{SOLVE}\xspace}
\newcommand{\MARK}{\textsf{MARK}\xspace}
\newcommand{\REFINE}{\textsf{REFINE}\xspace}
\newcommand{\ESTIMATE}{\textsf{ESTIMATE}\xspace}

\newcommand{\AFEM}{\textsf{AFEM}\xspace}

\newcommand{\ADAPTSURF}{\textsf{ADAPT\_SURFACE}\xspace}
\newcommand{\ADAPTPDE}{\textsf{ADAPT\_PDE}\xspace}

\newcommand{\ga}{\gamma}

\newcommand{\na}{\nabla}

\DeclareMathOperator{\dist}{dist}

\newcommand{\bXi}{ \boldsymbol{\mathcal{\chi}} }

\newcommand{\bx}{\mathbf{x}}
\newcommand{\by}{\mathbf{y}}
\newcommand{\bz}{\mathbf{z}}
\newcommand{\bv}{\mathbf{v}}
\newcommand{\bu}{\mathbf{u}}
\newcommand{\bn}{\mathbf{n}}

\newcommand{\bq}{\mathbf{q}}
\newcommand{\bw}{\mathbf{w}}
\newcommand{\bG}{\mathbf{T}}
\newcommand{\bD}{\mathbf{D}}

\newcommand{\bA}{\mathbf{E}}
\newcommand{\bI}{\mathbf{I}}

\newcommand{\bgG}{\mathbf{G}}
\newcommand{\bg}{\mathbf{G}}
\newcommand{\bnu}{{\boldsymbol \nu}}

\newcommand{\vg}{\mathbf{g}}

\newcommand{\JS}{\mathcal{J}_S}
\newcommand{\JT}{\mathcal{J}_{\partial T}}

\newcommand{\ldk}[1]{\lambda_{#1}}
\newcommand{\norm}[1]{\left\Vert#1\right\Vert}

\newcommand{\normLtT}[1]{\norm{#1}_{L_2(T)}}

\newcommand{\normLtpT}[1]{\norm{#1}_{L_2(\partial T)}}
\newcommand{\normLtk}[2]{\norm{#1}_{L_2({#2})}}

\newcommand{\abs}[1]{\left\vert#1\right\vert}

\newcommand{\br}[1]{\left[#1\right]}

\newcommand{\DivG}{\nabla_{\!\gamma}}

\newcommand{\DivGk}[1]{\nabla_{\!\Gamma_{#1}}}
\newcommand{\LapG}{\Delta_{\gamma}}

\newcommand{\InT}{\int_T}

\newcommand{\InS}{\int_S}

\newcommand{\vphi}{\varphi}

\newcommand{\Pn}[1][n]{\mathbb{P}_{#1}}

\newcommand{\Solve}{\textsf{SOLVE}}
\newcommand{\Estimate}{\textsf{ESTIMATE}}
\newcommand{\Mark}{\textsf{MARK}}
\newcommand{\Refine}{\textsf{REFINE}}
\newcommand{\bX}{\mathbf X}
\newcommand{\bV}{\mathbf V}

\title
{High-Order AFEM for the Laplace-Beltrami Operator:
Convergence Rates}

\author{%
Andrea Bonito
\thanks{Department of Mathematics, Texas A\&M University, 3368 TAMU, 
College Station, TX 77843-3368, USA. (\texttt{bonito@math.tamu.edu}) .
Partially supported by NSF Grant DMS-1254618.}
\and J.\ Manuel Casc\'on
\thanks{Departamento de Econom{\'\i}a e Historia Econ\'omica, Universidad de Salamanca, Salamanca 37008, Spain. (\texttt{casbar@usal.es}).
Partially supported by Secretar\'{\i}a de Estado de Investigaci\'on, Desarrollo e Innovaci\'on and Centro para el Desarrollo Tecnol\'ogico
Industrial of the Ministerio de Econom\'{\i}a y Competitividad (Spain), grant: CGL2011-29396-C03-02 and by Conserjer\'{\i}a de Educaci\'on of the Junta de Castilla y Le\'on, grants: SA266A12-2, SA020U16.}
\and Khamron Mekchay
\thanks{Department of Mathematics, Chulalongkorn University, Thailand.
(\texttt{k.mekchay@gmail.com}).
Partially supported by NSF Grants DMS-0204670, DMS-0505454, and
INT-0126272.
}
\and Pedro~Morin
\thanks{UNL, CONICET. Departamento de Matem\'atica, Facultad de Ingenier\'{i}a Qu\'{i}mica, 
Santiago del Estero 2829, S3000AOM Santa Fe, Argentina (\texttt{pmorin@santafe-conicet.gov.ar}). 
Partially supported by CONICET through grant PIP 112-2011-0100742, by Universidad Nacional del Litoral through grant CAI+D PI 501 201101 00476 LI, and by Agencia Nacional de Promoci\'on Cient\'{\i}fica y Tecnol\'ogica, through grants PICT-2012-2590, PICT-2013-3293 (Argentina).\looseness=-1
}
\and Ricardo H. Nochetto
\thanks{Department of Mathematics and Institute for Physical Science
  and Technology, University of Maryland, College Park, MD 20742, USA
  (\texttt{rhn@math.umd.edu}). Partially
  supported by NSF Grants DMS-1109325, DMS-1411808,
 and a RASA semester research award of the University of Maryland.
}
}

\newglossary{not}{notgls}{notglo}{Notations}
\newglossary{const}{constgls}{constglo}{Constants}
\newglossary{def}{defgls}{defglo}{Definitions}
\newglossary{algo}{algogls}{algoglo}{Algorithms}
\newglossary{indi}{indigls}{indiglo}{Errors and Indicators}

\makeglossaries

\newglossaryentry{not:poly_init}
{
type=not,
name={\ensuremath{\overline{\Gamma}_0}},
description={Initial polyhedral surface},
text={},
sort={G}
}

\newglossaryentry{not:macro_element}
{
type=not,
name={\ensuremath{\overline{\Gamma}_0^i}},
description={Macro-element},
text={},
sort={G}
}

\newglossaryentry{not:P0}
{
type=not, 
name={\ensuremath{P_0:\overline{\Gamma}_0 \rightarrow \gamma}},
description={Lipschitz homeomorphism from the initial polyhedral surface to the exact surface},
text={},
sort={P}
}

\newglossaryentry{not:Omega}
{
type=not, 
name={\ensuremath{\Omega}},
description={Local parametric domain},
text={},
sort={Omega}
}

\newglossaryentry{not:X0i}
{
type=not, 
name={\ensuremath{X_0^i:\Omega \rightarrow \overline{\Gamma}_0^i}},
description={Local parametrization of $\overline{\Gamma}_0$},
text={},
sort={X},
}

\newglossaryentry{not:Chi}
{
type=not, 
name={\ensuremath{\chi^i:\Omega \rightarrow \gamma^i}},
description={Local parametrization of $\gamma^i$},
text={},
sort={Chi}
}

\newglossaryentry{not:HatT}
{
type=not, 
name={\ensuremath{\widehat \T^i}},
description={Conforming parametrization of the parametric domain $\Omega$},
text={},
sort={Triangulation}
}

\newglossaryentry{not:Xi}
{
type=not, 
name={\ensuremath{X^i_{\T^i}:\Omega  \rightarrow \Gamma^i}},
description={Piecewise polynomial interpolant of $\chi^i$},
text={},
sort={X}
}

\newglossaryentry{not:Gammai}
{
type=not, 
name={\ensuremath{\Gamma^i}},
description={Piecewise polynomial approximation of $\gamma^i$},
text={},
sort={Gamma}
}

\newglossaryentry{not:OmegaM}
{
type=not, 
name={\ensuremath{\Omega^M}},
description={Global parametric space},
text={},
sort={OmegaM}
}

\newglossaryentry{not:T}
{
type=not, 
name={\ensuremath{\T^i}},
description={Local subdivision of the approximate surface},
text={},
sort={Triangulation}
}

\newglossaryentry{not:TT}
{
type=not, 
name={\ensuremath{\mathbb{T}}},
description={Forest associated with the initial subdivision $\T_0$},
text={},
sort={Forest}
}

\newglossaryentry{not:FEM}
{
type=not, 
name={\ensuremath{\mathbb{V}(\cT)}},
description={Continuous piecewise polynomial finite element spaces associated with $\cT$},
text={},
sort={FEM space}
}

\newglossaryentry{indi:lambdalocal}
{
type=indi, 
name={\ensuremath{\lambda_{\T^i}(\gamma,T)}},
description={Local geometric indicator},
text={},
sort={Lambda}
}

\newglossaryentry{indi:lambda}
{
type=indi, 
name={\ensuremath{\lambda_{\T^i}(\gamma)}},
description={Global geometric indicator},
text={},
sort={Lambda}
}

\newglossaryentry{not:first_fund}
{
type=not, 
name={\ensuremath{\bg, \bg_\Gamma}},
description={First fundamental forms},
text={},
sort={G},
}

\newglossaryentry{not:surf_grad}
{
type=not, 
name={\ensuremath{\nabla_\gamma, \nabla_\Gamma}},
description={Surface gradients},
text={},
sort={Gradient}
}

\newglossaryentry{not:area}
{
type=not, 
name={\ensuremath{q, q_\Gamma, r_\Gamma}},
description={Area elements},
text={},
sort={area}
}

\newglossaryentry{not:LB}
{
type=not, 
name={\ensuremath{\Delta_\gamma, \Delta_\Gamma}},
description={Parametric representation of the Laplace-Beltrami operators},
text={},
sort={Laplace-Beltrami}
}

\newglossaryentry{not:normal}
{
type=not, 
name={\ensuremath{\bn, \bn_\Gamma}},
description={Co-normals},
text={},
sort={co-normal}
}

\newglossaryentry{not:errormat}
{
type=not, 
name={\ensuremath{\bA_ \Gamma}},
description={Error Matrix},
text={},
sort={Error Matrix}
}

\newglossaryentry{not:meshsize}
{
type=not, 
name={\ensuremath{h_T}},
description={Meshsize},
text={},
sort={h}
}

\newglossaryentry{indi:eta}
{
type=indi, 
name={\ensuremath{\eta_\T(V,F_\Gamma)}},
description={Global PDE error indicator},
text={},
sort={eta}
}

\newglossaryentry{indi:etalocal}
{
type=indi, 
name={\ensuremath{\eta_\T (V,F_\Gamma,T)}},
description={Local PDE error indicator},
text={},
sort={eta}
}

\newglossaryentry{indi:osc}
{
type=indi, 
name={\ensuremath{\osc_{\cT}(V,f)}},
description={Global oscillation},
text={},
sort={osc}
}

\newglossaryentry{indi:osclocal}
{
type=indi, 
name={\ensuremath{\osc_{\cT}(V,f,T)}},
description={Local oscillation},
text={},
sort={osc}
}

\newglossaryentry{indi:pde_error}
{
type=indi, 
name={\ensuremath{\E_\T(U,f)}},
description={PDE error},
text={},
sort={E},
}

\newglossaryentry{indi:total_error}
{
type=indi, 
name={\ensuremath{E_\T(V;v,f,\gamma)}},
description={Total error},
text={},
sort={E}
}

\newglossaryentry{not:apprxclass}
{
type=not, 
name={\ensuremath{\mathbb{A}_s}},
description={Approximation classes},
text={},
sort={A},
}

\newglossaryentry{const:M}
{
type=const, 
name={\ensuremath{M}},
description={Number of macro-elements},
text={}
}

\newglossaryentry{const:L}
{
type=const, 
name={\ensuremath{L}},
description={Bi-Lipschitz constant of $\chi^i$, see \eqref{bi_lipschitz}},
text={}
}

\newglossaryentry{const:Lambda0}
{
type=const, 
name={\ensuremath{\Lambda_0}},
description={Geometric estimator quasi-monotonicity constant, see \eqref{quasi-mono-n}},
text={}
}

\newglossaryentry{const:b}
{
type=const, 
name={\ensuremath{b}},
description={Number of successive bisections performed on each marked element},
text={}
}

\newglossaryentry{const:upper_lower}
{
type=const, 
name={\ensuremath{C_1, C_2, \Lambda_1}},
description={A-posteriori upper and lower bounds constants, see \eqref{upper} and \eqref{lower}},
text={}
}

\newglossaryentry{const:C3}
{
type=const, 
name={\ensuremath{C_3}},
description={Oscillation and PDE error estimator relation},
text={}
}

\newglossaryentry{const:C45}
{
type=const, 
name={\ensuremath{C_4, C_5}},
description={PDE error and estimator equivalence constants},
text={}
}

\newglossaryentry{const:C6}
{
type=const, 
name={\ensuremath{C_6}},
description={Quasi-monotonicity of the data oscillation constant}, 
text={}
}

\newglossaryentry{const:Lambda3}
{
type=const, 
name={\ensuremath{\Lambda_3}},
description={Quasi-monotonicity of the data oscillation and reduction of residual error estimator constant},  
text={}
}

\newglossaryentry{const:Lambda2}
{
type=const, 
name={\ensuremath{\Lambda_2}},
description={Quasi-monotonicity of the data oscillation, reduction of residual error estimator and quasi-orthogonality constant},  
text={}
}

\newglossaryentry{const:w}
{
type=const, 
name={\ensuremath{\omega}},
description={Surface and PDE approximations relative tolerance constant},
text={}
}

\newglossaryentry{const:w1}
{
type=const, 
name={\ensuremath{\omega_1}},
description={Restriction on $\omega$ to guarantee that \ADAPTPDE  reduces the PDE error},
text={}
}

\newglossaryentry{const:theta}
{
type=const, 
name={\ensuremath{\theta}},
description={D\"orfler marking parameter},
text={}
}

\newglossaryentry{const:C7}
{
type=const, 
name={\ensuremath{C_7}},
description={Complexity of \REFINE constant},
text={}
}

\newglossaryentry{const:w2}
{
type=const, 
name={\ensuremath{\omega_2}},
description={Restriction on $\omega$ required to guarantee that \ADAPTPDE  satisfies a contraction property},
text={}
}

\newglossaryentry{const:omega3}
{
type=const, 
name={\ensuremath{\omega_3}},
description={Restriction on $\omega$ for the adaptive finite element to satisfy an optimal D\"orfler marking property},
text={}
}

\newglossaryentry{const:thetas}
{
type=const, 
name={\ensuremath{\theta^*}},
description={Restriction on the D\"orfler marking parameter $\theta$},
text={}
}

\newglossaryentry{def:shapereg}
{
type=def, 
name={Shape-Regularity},
description={Definition~\ref{d:shape_reg}},
text={}
}

\newglossaryentry{def:label}
{
type=def, 
name={Admissible labeling},
description={Remark~\ref{r:labeling}},
text={}
}

\newglossaryentry{def:besov}
{
type=def, 
name={\ensuremath{|.|_{B^s_q(L_p(\Omega))}}},
description={Besov semi-quasi-norms},
text={}
}

\newglossaryentry{algo:AFEM}
{
type=algo, 
name={\ensuremath{\AFEM}},
description={Adaptive finite element method},
text={}
}

\newglossaryentry{algo:SURFACE}
{
type=algo, 
name={\ADAPTSURF},
description={Adaptive finite element method reducing the surface approximation error},
text={}
}

\newglossaryentry{algo:PDE}
{
type=algo, 
name={\ADAPTPDE},
description={Adaptive finite element method reducing the PDE approximation error},
text={}
}

\newglossaryentry{algo:GREEDY}
{
type=algo, 
name={\GREEDY},
description={Greedy algorithm},
text={}
}

\begin{document}

\maketitle

\begin{abstract}
We present a new AFEM for the Laplace-Beltrami operator 
with arbitrary polynomial degree on parametric
surfaces, which are globally $W^1_\infty$ and piecewise in a
suitable Besov class embedded in $C^{1,\alpha}$ with $\alpha \in (0,1]$.
The idea is to have the surface
sufficiently well resolved in $W^1_\infty$ relative to the current
resolution of the PDE in $H^1$. This gives rise to a conditional contraction 
property of the PDE module. We present a suitable approximation class
and discuss its relation to Besov regularity of the surface,
solution, and forcing.
We prove optimal convergence rates for AFEM which are dictated
by the worst decay rate of the surface error
in $W^1_\infty$ and PDE error in $H^1$.
\end{abstract}

%\tableofcontents

%-------------------------------------------------------------------------------%
\section{Introduction}\label{S:introduction}
%-------------------------------------------------------------------------------%

Let $\gamma$ be a $d$ dimensional surface in $\R^{d+1}$ ($d\ge1$) either with or
without boundary, 
which is globally Lipschitz 
and piecewise in a suitable Besov class embedded in $C^{1,\alpha}$
with $\alpha \in (0,1]$.
We design and study a quasi-optimal adaptive finite
element method (AFEM) to approximate the solution of 
\begin{align}
- \LapG u &= f  \quad\text{ on }\quad \gamma, \label{p:PDE_Gm}
\end{align}
where $f\in L_2(\gamma)$ and $-\LapG$ is the Laplace-Beltrami operator
(or surface Laplacian) on $\gamma$.
In addition, we impose that $u=0$ on $\partial \gamma$ 
or require that $\int_\gamma u = 0$
if $\partial \gamma = \emptyset$ (with $\int_\gamma f = 0$ for compatibility).
To represent $\Delta_\gamma$, one needs to describe $\gamma$ mathematically
using, for example, parametric representations on charts, level
sets, distance functions, graphs of functions, etc. Moreover, one
usually obtains approximate solutions (finite element solutions)
by solving the problem on approximate polyhedral surfaces rather
than the surface $\gamma$ itself. 
Exploiting the variational structure of the Laplace-Beltrami
operator, \cite{DziuK:88} gives an a priori error analysis whereas
\cite{DemlowDziuk:07,Demlow:09,MMN:11,BCMN:Magenes} provide 
a posteriori counterparts. Our present
objective is to continue our
research on AFEM for elliptic PDEs on surfaces initiated in
\cite{MMN:11} for graphs and extended in \cite{BCMN:Magenes} to
parametric surfaces, the latter with polynomial degree $n=1$.
We design herein an AFEM for parametric surfaces using $C^0$ finite elements
of degree $n\ge 1$, prove optimal convergence rates and workload
estimates, and study
suitable approximation classes for the triple $(u, f, \gamma)$.

High-order finite elements are superior to linears for geometric
problems: they provide better approximation of important geometric
quantities such as curvature, and they are less sensitive to mesh tangling
due to tangential node motion for time dependent problems; we refer to 
\cite{BCMN:Magenes} for a discussion of several applications. The
advantage of high-order methods is even more pronounced when they are
combined with adaptivity. AFEMs are known to exploit the nonlinear
Besov regularity scale, instead of the linear Sobolev scale, and to deliver optimal convergence rates $N^{-n/d}$ in terms of
degrees of freedom $N$ for singular elliptic problems on flat domains 
with limited Sobolev regularity
\cite{Stevenson:07,CaKrNoSi:08},
\cite{GSTER14,CN2011,Gantumur,Morin11,NoSiVe:09}.
The study of AFEM for the Laplace-Beltrami operator on parametric
surfaces is, however, restricted to $n=1$ because the
first fundamental form, area element, and normal vector
to the discrete surface as well as the surface gradient of discrete functions
are piecewise constant, which greatly simplifies the analysis
\cite{BCMN:Magenes}. This paper bridges this gap and provides a
comprehensive approach to high-order AFEM on parametric surfaces.

It is standard practice to pose the discrete problem on a piecewise
polynomial approximation $\Gamma$ of the exact surface $\gamma$. This
is unavoidable when dealing with evolving surfaces, such as time
dependent free boundary problems, for which $\gamma$ is unknown
\cite{BCMN:Magenes}. This surface discrepancy is responsible for a
geometric consistency error not present in the flat case, which makes
this setting mathematically challenging and intriguing.
In fact, there is a non-linear interplay between the approximate surface
$\Gamma$ and the approximate solution $U$ defined on $\Gamma$.
To elucidate this issue, one might think of the Laplace-Beltrami
operator as a linear elliptic operator with variable coefficients in a flat
parametric domain, except that the approximate coefficients are not
piecewise polynomials as in \cite{Bonito-DeVore-Nochetto:13} but
rather some rational functions when $n>1$. The multiplicative
structure of the solution-coefficient interaction is an essential new
difficulty we must cope with to develop high-order AFEM and study
their performance.

To handle this nonlinear interaction,
we propose an AFEM which successively applies two
different modules: \ADAPTSURF approximates the surface $\gamma$ in 
$W^1_\infty$ and \ADAPTPDE approximates the solution $u$ in $H^1$.
The former is a greedy algorithm which
monitors the geometric estimator whereas the
latter deals with a residual estimator. If $\{\T_k,U_k\}_{k=0}^\infty$
denotes the sequence of meshes
and Galerkin solutions generated by AFEM in step $k$ using a
discrete forcing function $F_k$, the method reads as follows:

\medskip
{\bf AFEM:}
\rhn{Given $\T_0$} and parameters $\varepsilon_0>0$, 
  		$0<\rho<1$, and $\omega>0$,  set $k=0$.
\begin{algotab}
  \> \>1. \rhn{$\T_k^+ = \ADAPTSURF (\T_k, \omega \eps_k)$}
  \\
  \> \>2. $[U_{k+1}, \T_{k+1}] = \ADAPTPDE (\T_k^+,\eps_k)$ \\
  \> \>3. $\eps_{k+1} = \rho \eps_k$; $k = k+1$ \\  
  \> \>4. go to 1.
\end{algotab}

\smallskip\noindent
This strategy bears  similarities with the algorithms proposed in
\cite{Bonito-DeVore-Nochetto:13} targeting diffusion problems with
partial information on the coefficients.
This concept relates directly to surface approximation in the
present context but it is intrinsically different than piecewise
polynomial approximation of coefficients. We develop herein new
techniques to handle such differences upon insisting on the geometric
nature of the approximation.

We now describe AFEM. For the purpose of this introduction, we assume
that $\gamma$ can be parameterized by a single map
$\chi:\Omega \to \gamma$, where $\Omega\subset\R^d$ denotes the corresponding parametric domain and refer to \S~\ref{S:repres-surface} for the more general case.
 If $\wT:=\wT(\Omega)$ is a generic triangulation in $\Omega$, then $\V(\wT)$ denotes the
space of continuous piecewise polynomial functions of degree $\le n$
subordinate to $\wT$. 
\rhn{Let $I_\cT:C(\overline{\Omega})\to \V(\wT)$ be the Lagrange interpolation operator and $X_\cT = I_\cT \chi$ be the interpolant of the parametrization $\chi$. The map $X_\cT$
}
%The Lagrange interpolant of the parameterization
%$\chi$, $X_{\cT}:=I_\wT \chi \in \widehat \V(\wT)$, 
induces the
discrete (piecewise polynomial) surface
$\Gamma := X_{\cT}(\Omega)$ and a curved mesh $\T := \{ T =
X_{\cT}(\widehat T) \mid \widehat T \in \wT\}$;
note the correspondence between a flat element $\widehat{T}\in\wT$ and
a curved element $T\in\T$.
\rhn{This one-to-one correspondence between $\wT$ and $\cT$ justifies the use of subscript $\cT$ instead of $\wT$ in the notation of $I_\cT$ and $X_\cT$; the same slight abuse of notation will be employed throughout the paper with other quantities defined in $\Omega$ or $\wT$. In this vein, we next}
define the
\textit{geometric estimator}
$\lambda_{\cT}(\gamma) := \max_{\manel{T\in \T}} \ldk{\cT}(\gamma,T)$ in terms of the
\textit{geometric element indicator}
\begin{equation}\label{p:geom_osc}
\lambda_{\cT}(\gamma,T) := 
\|\wnabla (\chi - X_{\cT})\|_{ L_\infty(\widehat{T})}
\qquad\forall \, T\in\T;
\end{equation}
we observe that $\ldk{\cT}(\gamma,T)$ is evaluated in the 
flat element $\widehat{T}\in\wT$.
Given a tolerance $\eps>0$ and a mesh $\T$, the procedure
\begin{equation*}
\manel{	\T_* = \ADAPTSURF (\T,\eps),}
\end{equation*}
finds adaptively a refinement $\T_*$ of $\T$, denoted $\T_* \geq \T$,
and its corresponding piecewise polynomial
approximation $\Gamma_*$ of $\gamma$,  such that
\begin{equation} \label{cond:AS}
\lambda_{\T_*}(\gamma)\leq \eps.
\end{equation}
We \rhn{say that this module is $t-optimal$ provided} the number
of marked elements $\#\cM$ to achieve \eqref{cond:AS} satisfies
\begin{equation}\label{optimal:AS}
\#\cM \Cleq C(\gamma) \; \eps^{-1/t}.
\end{equation}
The largest value of $t\le n/d$ depends on the dimension $d$ and the
polynomial degree $n\ge1$.
In \S \ref{S:gamma}, we show that \eqref{optimal:AS} holds if
$\chi$ belongs to a suitable 
Besov space.

\rhn{Throughout this paper we use the notation $A\Cleq B$ to denote $A\leq
CB$ with a constant $C$ independent of $A$ and $B$, and write $A \Ceq B$ to mean $B \Cleq A \Cleq B$. 
We shall indicate
if appropriate on which quantities the constant $C$ depends on.
}

Since the exact and approximate solutions $u$ and $U$ are defined on
different surfaces $\gamma$ and $\Gamma$, we have to
decide how to compare them.
We lift $U$ to $\gamma$ via the map $X_{\cT} \circ \chi^{-1}$,
\rhn{but keep the symbol $U$,} and define
the {\it energy error} to be
\begin{equation}\label{p:energy}
e_\T(U) := \normLtk{\DivG(u-U)}{\gamma},
\end{equation}
where $\DivG$ denotes the \emph{surface gradient} on $\gamma$
defined below in \S\ref{S:diff-geom}. We further denote
by $\eta_\T(U,F)$ the {\it residual estimator} of $e_\T(U)$,
defined later in \S\ref{S:bounds}.
If $\eps$ stands for a tolerance, the procedure
\begin{equation*}
[U_*,\T_*] = \ADAPTPDE (\cT,\eps)
\end{equation*}
finds adaptively a refinement $\T_*$  of $\T$ such that the Galerkin
solution $U_* \in\mathbb V(\T_*)$ on $\Gamma_*$---the approximate surface corresponding to $\T_*$---satisfies the prescribed bound
$\eta_{\T_*}(U_*,F_*) \leq \varepsilon$. This is 
the usual loop for linear elliptic PDE \cite{Doerfler:96,MNS:02},
\cite{BCMN:Magenes,BN:10,GSTER14,CaKrNoSi:08,CN2011,Morin11,MMN:11,NoSiVe:09}
\begin{equation}  \label{p:loop}
\Solve \to \Estimate \to \Mark \to \Refine,
\end{equation}
except that the approximate surface $\Gamma$
is updated after each $\Refine$  call and therefore changes within \ADAPTPDE.

Note that there is a tolerance $\eps_k$ being 
reduced geometrically in every outer loop of AFEM, 
and a small parameter $\omega$ (to be determined explicitly) that relates
the tolerances for both procedures. The role of $\omega$ is critical
to derive convergence rates for AFEM,
and is explored computationally in~\cite[Section 2]{BCMN:Magenes} for dimension $d=2$ and polynomial degree $n=1$.
The presence of $\omega$ is in the spirit of the inner
loop of \cite{Stevenson:07} to handle data $f\in H^{-1}$ and of
\cite{Bonito-DeVore-Nochetto:13} to deal with discontinuous coefficients in
the flat case. It means that the surface must be resolved slightly better
than the solution for $\eta_\T(U,F)$ to provide reliable
information about $e_\T(U)$.

Our first main result is a \textit{conditional contraction property}
of \ADAPTPDE, which reads as follows and is shown in \S \ref{S:contraction}:
\medskip
\begin{quote}
{\em
If the parameter $\omega>0$ is small enough, there exist constants
$0< \alpha < 1 $ and $\beta >0$ such that, 
if the geometric estimator $\lambda_{\T_k}(\gamma) \le \omega \varepsilon_k$
and the error estimator $\eta_k = \eta_{\T_k}(U_k,F_k) \ge \varepsilon_k$,
then the inner iterates $\{\Gamma_j, \T_j, U_j, \eta_j\}_{j = 0}^J$
of $\ADAPTPDE(\T_k^+, \varepsilon_k)$ satisfy
\begin{equation*}
	e_{j+1}^2 + \beta \eta_{j+1}^2 \le \alpha^2 
	\big(e_j^2 + \beta \eta_{j}^2 \big)
	\qquad\forall \, \manel{0\le j < J},
\end{equation*}
where $e_j:= e_{\T_j}(U_j)$ and $J$ is uniformly bounded with
respect to $k$.
}
\end{quote}

\medskip
To derive convergence rates we need to seek a suitable error quantity and
associated approximation class; this is fully discussed in \S
\ref{S:approx-class}.
Since all decisions of the AFEM are
based on the estimators $\{\eta_\T(U,F), \ldk{\cT}(\gamma) \}$, the convergence rate
of AFEM is dictated by these quantities. We will show in Lemma 
\ref{L:equiv}  that for all the inner iterates $(\T, U )$ within \ADAPTPDE
\begin{equation}\label{total-error1}
\normLtk{\DivG(u-U)}{\gamma}^2 + \osc_{\cT} (U,f)^2 \approx \eta_\T (U,F)^2.
\end{equation} 
where the oscillation $\osc_{\cT} (U,f)$, a
quantity evaluated in the parametric domain, can be bounded
\rhn{separately in terms of $U$ and $f$} as follows:
\begin{equation}\label{osc:f+U}
\osc_{\cT}(U,f)^2 \leq \osc_{\cT}(U)^2 + \osc_{\cT}(f)^2.
\end{equation}
The presence of the first term is a feature inherent to polynomial
degree $n>1$ which is absent in \cite{BCMN:Magenes}.
This justifies the following notion of \textit{total error}
\begin{equation}\label{total-error2}
\rhn{
E_\T(U; u, f,\gamma) := \left( \normLtk{\DivG(u-U)}{\gamma}^2
		+ \osc_{\cT}(U,f) ^2 +  \omega^{-1}\lambda_{\cT}(\gamma)^2 \right)^{1/2},
}
\end{equation}
where the scaling $\omega^{-1}$ brings the \rhn{geometric} estimator
$\lambda_{\cT}(\gamma)$ to a size comparable with $\eta_\T(U,F)$. 
Then the quality of the best
approximation of $(u,f,\gamma)$ with $N$ degrees of freedom can
be assessed  in terms
of the following best approximation error:
\begin{equation*}
  {\sigma}(N;u,f,\gamma) \definedas \inf_{\cT \in  \mathbb{T}_N} 
  \inf_{V\in\V(\T)}
      E_{\T}(V;u,f,\gamma),
\end{equation*}
where $\V(\T)$ denotes the approximation space on the discrete surface $\Gamma$, and 
$\mathbb{T}_N$ is the set of conforming
triangulations obtained after $N$ bisections from $\T_0$.
We say that the triple $(u,f,\gamma)$ 
belongs to the approximation class $\As$, with $0<s\le n/d$, 
if 
\begin{equation}\label{class-s}
\sigma (N; u, f,\gamma) \Cleq N^{-s}; 
\end{equation}
equivalently, for any natural number $N \geq 1$, there is a
conforming mesh refinement $\T_N$ of
the initial mesh $\T_0$ satisfying $\#\T_N-\#\T_0 \leq N$ and such that 
$E_{\T_N}(V;u,f,\gamma) \Cleq  N^{-s}$ for some $V\in\V(\T_N)$.
\rhn{We observe that if $(u,f,\gamma) \in \As$ then the module
  $\ADAPTSURF$ is $s$-optimal, namely~\eqref{cond:AS}--\eqref{optimal:AS} are valid with $t=s$.}

The algebraic error decay \eqref{class-s} relates to
Besov regularity for flat domains \cite{BDDP,GM:14,Gantumur}. The situation
for surfaces is much more intricate due to the nonlinear surface-PDE
interaction. We wonder whether regularity of $\gamma$ enabling an
error decay $N^{-s}$ in $W^1_\infty$ is compatible with a similar
decay rate for $e_\T(U)$ and $\osc_{\cT}(U,f)$, which depend on the
approximate surface $\Gamma$. Exploring this question is a fundamental
contribution of this paper and entails the study of Besov regularity
of products and composition of functions, which we carry out in
\S \ref{S:Besov} and is of independent interest.
We apply our findings in \S \ref{S:approx-class}
to quantify the effect of surface approximation in the decay of both
$e_\T(U)$ and $\osc_{\cT}(U,f)$.
This leads to our second main contribution:
\medskip
\begin{quote}
{\em
Let $0<p,q\le\infty$, $0<s\le n/d$ 
such that $s>\frac{1}{p}-\frac{1}{2}$, $s>\frac{1}{q}$.
If the triple $(u, f, \gamma)$ satisfies
\begin{equation*}\label{besov}
u \in B^{1+sd}_p(L_p(\Omega)),
\quad
f \in B^{sd}_p(L_p(\Omega)),
\quad
\chi \in B^{1+sd}_q(L_q(\Omega)),
\end{equation*}
then \eqref{class-s} holds, i.e., $(u,f,\gamma) \in \As$. Moreover,
$\osc_{\cT}(f)$ exhibits a faster decay $s+1/d$.
}
\end{quote}

\medskip\noindent
We observe that $s>\frac{1}{p}-\frac{1}{2}$ and $s>\frac{1}{q}$
guarantee that $B^{1+sd}_p(L_p(\Omega)) \subset H^1(\Omega)$ and
$B^{1+sd}_q(L_q(\Omega)) \subset W^1_\infty(\Omega)$, whence the
additional regularity is just above the nonlinear Sobolev scale for
both $u$ and $\gamma$. This shows that the two scales are indeed
compatible, and that if $s=\frac{n}{d}$, then
$p>\frac{2d}{2n+d}$ and $q>\frac{d}{n}$ may be
smaller than 1 for $n>\frac{d}{2}$. The latter does not happen for $n=1$ and
represents a striking difference with \cite{BCMN:Magenes}.

\medskip
Our third main contribution is
a quasi-optimal decay rate for the AFEM in terms of
degrees of freedom under natural restrictions on the initial
triangulation $\T_0$, marking parameter 
$\theta$ of \MARK and parameter $\omega$ of AFEM. This is developed in 
\S \ref{S:rates} and reads:

\medskip
\begin{quote}
\em
\rhn{Let the initial mesh $\T_0$ have an admissible labeling
for refinement},
and $\theta \in (0, \theta_*)$,  $\omega\in (0,\omega_*)$ for 
$\theta_*, \omega_*$ sufficiently small.
If $(u,f,\gamma) \in \As$ 
%and the module \ADAPTSURF is
%$s$-optimal in the sense of \eqref{optimal:AS}, 
then the sequence
$\{\Gamma_k, \T_k, U_k\}$ generated by AFEM verifies 
\begin{equation*}
%	\normLtk{\DivG(u-U_{k})}{\gamma}
%		+ \osc_{\T_k}(f,U_k) + \omega^{-1} \lambda_{\T_k}
E_{\T_k}(U_k;u,f,\gamma)
				\Cleq  ( \#\T_k - \# \T_0)^{-s}.
\end{equation*}
Moreover, the workload up to step $k$ of $\AFEM$ is proportional to
$\eps_k^{-1/s}$ provided each inner loop of $\ADAPTPDE$ 
has linear complexity.
\end{quote}

\medskip
The rest of the paper is organized as follows.
We discuss the representation and interpolation of $\gamma$
in \S~\ref{S:param-surface}, and basic differential geometry
leading to the Laplace-Beltrami operator in \S~\ref{S:Laplace-Beltrami}.
In \S~\ref{S:aposteriori} we obtain a 
posteriori error estimates for the energy error and derive several
properties of the estimator and oscillation. 
In \S~\ref{S:AFEM}  we examine the various modules of AFEM and in \S~
\ref{S:contraction}
establish the conditional contraction property of \ADAPTPDE  (first main result).
In \S~\ref{S:approx-class} we show that suitable Besov regularity of
the triple $(u,f,\gamma)$ implies $(u,f,\gamma)\in\As$, which is our second main result.
We next prove quasi-optimal convergence rates in \S~\ref{S:rates} 
-- our third main result.
After recalling a definition of Besov spaces, we
establish in \S~\ref{S:Besov} \rhn{scale independent estimates} for the products and compositions of
functions in Besov spaces. \rhn{These results are instrumental in~\S~\ref{S:approx-class} and do not seem to be available in the literature.}
\ab{Glossaries of notations, relevant constants, definitions and algorithms are provided at the end.}
%%%%%%%%%%%%%%%%%%%%%%%%%%%%%%%%%%%%%%%%%%%%%%%%%%%%%%%%%%%%%%
\section{Parametric Surfaces}\label{S:param-surface}
%%%%%%%%%%%%%%%%%%%%%%%%%%%%%%%%%%%%%%%%%%%%%%%%%%%%%%%%%%%%%%
%
In this section we discuss how to represent a
parametric surface \rhn{$\gamma$ by non-overlapping charts and its discretization $\Gamma$ by interpolation of $\gamma$}, which is instrumental for the design, analysis,
and implementation of our AFEM. % on parametric surfaces.
\rhn{Dealing with overlapping charts is not practical computationally.}

%--------------------------------------------------------------------------------
\subsection{Representation of Parametric  Surfaces}\label{S:repres-surface}
%--------------------------------------------------------------------------------
%

We assume that the surface $\gamma$ is described as the deformation of a $d$ 
dimensional polyhedral surface \gls{not:poly_init}$\overline{\Gamma}_0$ by a globally Lipschitz
\emph{homeomorphism}
 \gls{not:P0}$P_0:\overline{\Gamma}_0 \rightarrow \gamma \subset \mathbb
R^{d+1}$. 
The overline notation is to emphasize that $\overline{\Gamma}_0$ is piecewise affine.
Moreover, if $\overline{\Gamma}_0=\bigcup_{i=1}^{M} \overline{\Gamma}_0^i$ is made up of \gls{const:M}$M$ (closed) faces
\gls{not:macro_element}$\overline{\Gamma}_0^i, i=1,\dots,M$, we denote by
$P_0^i:\overline{\Gamma}_0^i \rightarrow \mathbb R^{d+1}$
the restriction of $P_0$ to $\overline{\Gamma}_0^i$.
 We refer to $\overline{\Gamma}_0^i$ as a
\emph{macro-element} which induces the partition
$\{\gamma^i\}_{i=1}^{M}$ of $\gamma$ upon setting
\[
\gamma^i := P_0^i(\overline{\Gamma}_0^i). 
\]
Note that this \emph{non-overlapping} parametrization allows for piecewise smooth
surfaces $\gamma$ with possible kinks matched by the decomposition 
$\{\gamma^i\}_{i=1}^{M}$.

In order to avoid technicalities, we assume that all the macroelements
are simplices, i.e. there is a (closed) reference simplex
\gls{not:Omega}$\Omega \subset\R^{d}$, from now on called the {\it local parametric domain}, and an
affine map \gls{not:X0i}$X_0^i:\R^d\to\R^{d+1}$ such that
$\overline{\Gamma}_0^i = X_0^i(\Omega)$;
Figure~\ref{F:surf_representation} sketches the
situation when $d=2$.
We thus let \gls{not:Chi}$\chi^{i} := P_0^i\circ X_0^i:\Omega \to\gamma^i$ be a local
parametrization of $\gamma$ which is bi-Lipschitz, namely there exists
a universal constant $L\geq1$ such that for all $1\le i \le M$
\begin{equation}\label{bi_lipschitz}
\gls{const:L}L^{-1} |\hat{\bx}-\hat{\by}| 
\leq |\chi^i(\hat{\bx})- \chi^i(\hat{\by})| 
\leq L |\hat{\bx}-\hat{\by}|, 
\qquad  \forall \, \hat{\bx},\hat{\by}\in\Omega.
\end{equation}
%{\bf AB: This is not quite "globally" bi-Lipschitz as we to not have differences across patches, do we want to remove the globally? P: I already removed it}
%
This {\it minimal regularity} of $\gamma$, to be soon
strengthened out locally in each macro-element, implies the more familiar 
condition, valid for a.e. $\hat{\bx}\in\Omega$,
\begin{equation}\label{jacobian-X}
L^{-1} |\bw| \le |\widehat{\nabla} \chi^i(\hat{\bx}) \bw| \le L |\bw|
\qquad\forall ~\bw\in\mathbb R^d.
\end{equation}
%
%{\bf AB: For the same reason, I have removed ``hence $L\ge1$ is the Lipschitz constant of $\bXi$ and so of $\gamma$". P: ok}
\rhn{We use bold notation to denote} the collection of these parametrizations, i.e.
$\bXi := \{\chi^i\}_{i=1}^M$.
We further assume that $P_0(\mathbf v) = \mathbf v$ 
for all vertices $\mathbf v$ of $\overline{\Gamma}_0$, so that $X_0^i$ is the nodal interpolant of $\chi^{i}$ into linears. 
\begin{figure}[ht!]
\centerline{\includegraphics[width=0.45\textwidth]{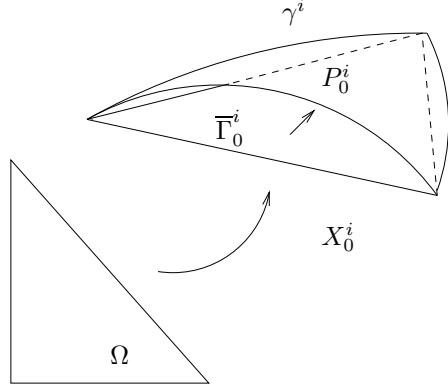}}
\begin{picture}(0,0)(-30,-10)
\put(110,10){$\Omega$}
\put(190,55){$X_0^i$}
\put(150,95){\rhn{$\overline{\Gamma}_0^i$}}
\put(190,115){$P_0^i$}
\put(175,140){$\gamma^i$}
\end{picture}
\caption{\footnotesize{Representation of each component $\gamma^i$ when $d=2$ 
as a parametrization from a
flat triangle $\overline{\Gamma}_0^i\subset\mathbb R^3$ as well as
from the reference simplex $\Omega\subset\mathbb R^2$. The map
$X_0^i:\Omega\to \overline{\Gamma}_0^i$ is affine.}}
\label{F:surf_representation}
\end{figure}

The structure of the map $P_0$ depends on the application. 
For instance, if $\gamma^i$
is described on $\Gamma_0^i$ via the \emph{distance function}
$\dist(\bx)$ to $\gamma$, then
\[
\gamma^i \ni \tilde{\bx} = \bx - \dist(\bx) \nabla\dist(\bx) = P_0(\bx)  \qquad
\forall~\bx\in \overline{\Gamma}_0^i,
\]
provided $\dist(\bx)$ is sufficiently small so that the distance is
uniquely defined.
If, instead, $\gamma^i$ is the \emph{zero level set} $\phi(\bx) = 0$ of
a function $\phi$, then
\[
\overline{\Gamma}_0^i \ni \bx = \tilde{\bx} + \frac{\nabla
\phi(\tilde{\bx})}{\abs{\nabla\phi(\tilde{\bx})}} \abs{\bx-\tilde{\bx}} =
P_0^{-1}(\tilde{\bx}) \qquad \forall~\tilde{\bx}\in\gamma^i
\]
is the inverse map of $P_0$. In both cases, $\dist$ and $\phi$
must be $C^2$ for $P_0^i$ to be $C^1(\overline{\Gamma}_0^i)$.
Yet another option is to view $\gamma^i$ as a graph on $\overline{\Gamma}_0^i$, in
which case $P_0^i$ is a lift in the normal direction to $\overline{\Gamma}_0^i$
and $P_0$ is $C^1(\overline{\Gamma}_0^i)$ if and only if $\gamma^i$ is \rhn{$C^1(\overline\Omega)$};
we refer to \cite{MMN:11}. Notice that the
inverse mapping theorem implies $(P_0^i)^{-1}\in C^1(\gamma^i)$.

The {\it regularity} of $\gamma$ is expressed in terms of the regularity of the maps
$\chi^i$. If $ s > 0$, $0< p, q \leq \infty$,
we say that $\gamma$ is piecewise $B^{1+s}_q(L_p(\Omega))$ whenever
$
\chi^i \in [B^{1+s}_q(L_p(\Omega))]^{d+1}$, $i=1,\dots,M$; or shortly $\bXi \in  [B^{1+s}_q(L_p(\Omega))]^{(d+1)\times M}$.
We refer to \S~\ref{S:Besov} for the definition of Besov norms and spaces.

We observe that a function ${v}^i:\gamma^i\to\R$ defines uniquely two
functions $\widehat{v}^i:\Omega\to\R$ and $\bar{v}^i:\overline{\Gamma}_0^i\to\R$ via the
maps $\chi^{i}$ and $P_0$, namely
\begin{align}  \label{para_lift}
\widehat{v}^i(\hat{\bx}) := v^i(\rhn{\chi}^{i}(\hat{\bx})) \quad \forall\; \hat{\bx}\in
\Omega,
\qquad \bar{v}^i(\bar{\bx}) := v^i(P_0(\bar{\bx}))
\quad\forall\; \bar{\bx}\in\overline{\Gamma}_0^i.
\end{align}
Conversely, a function $\widehat{v}^i:\Omega\to\R$ (respectively,
$\bar{v}^i:\overline{\Gamma}_0^i\to\R$) defines uniquely the two functions
$v^i:\gamma^i\to\R$ and $\bar{v}^i:\overline{\Gamma}_0^i\to\R$ (respectively,
$v^i:\gamma^i\to\R$ and $\widehat{v}^i:\Omega\to\R$).
When no confusion is possible, we will denote by \rhn{$v^i$} the three
functions $v^i,\bar{v}^i$ and \rhn{$\widehat{v}^i$} and set
$\rhn{\tilde{\bx}^i : = \chi^i}(\hat{\bx})$ for all $\hat{\bx}\in\Omega$,
\rhn{$i=1,\ldots, M$}.
Moreover, we will use vector notation
\begin{equation}\label{bf-notation}
\bv := \{ v^i \}_{i=1}^M,
\end{equation}
along with the convention
\begin{equation}\label{global-norm}
\| \bv \|_{B(\Omega)}:= \max_{i=1,...,M} \| v^i \|_{B(\Omega)}, \qquad 
| \bv |_{B(\Omega)}:= \max_{i=1,...,M} | v^i |_{B(\Omega)}.
\end{equation}
for (quasi) norms and semi-norms defined on a quasi-normed linear space $B(\Omega)$;
typically $B(\Omega)$ will be a Lebesgue, Sobolev, or Besov space. Moreover,
we will write
\begin{equation}\label{local-norm}
\|v\|_{B(\widehat{T})}, \qquad |v|_{B(\widehat{T})}
\qquad\forall \, \widehat{T} \in \wT
\end{equation}
to indicate the {\it local} (quasi) norms and semi-norms over a
generic element $\widehat{T}\in\wT^i$ without specifying the
superscript $i$ in either function $v$ or mesh $\wT$.

Before proceeding further, we note that as a general rule, we use hat symbols to denote quantities related to $\Omega$, an overline to refer to quantities on $\overline{\Gamma}_0$, tilde to characterize quantities in $\gamma$ and bold to indicate vector quantities.

%--------------------------------------------------------------------------------
\subsection{Interpolation of Parametric Surfaces and Finite Element Spaces}\label{S:interp--surface}
%--------------------------------------------------------------------------------
%
The partition of the initial polyhedral surface $\overline{\Gamma}_0$ in macro-elements (or faces) induces
a conforming triangulation of $\overline{\Gamma}_0$;
we call this set $\overline{\T}_0$.  
We only discuss the class of conforming meshes $\overline{\mathbb{T}} := \mathbb{T}(\overline{\T}_0)$ 
created by successive bisections of this initial mesh $\overline{\T}_0$.
However, our results remain valid for \rhn{quad-refinements and red-refinements all with hanging nodes; we refer to Remark~\ref{r:Alt_sub_stategies} for the notion of admissible refinements.} 
%any refinement strategy 
%satisfying Conditions 3, 4 and 6 in \cite{BN:10}, \rhn{we refer to Remark \ref{r:Alt_sub_stategies} for more details.}
%\delete{In particular, successive bisections, quad-refinement and red-refinement all with
%hanging nodes are admissible refinement strategies.
%For more details, we refer to %\cite[Section 6]{BN:10}
%}.
A triangulation $\overline{\T} \in \overline{\mathbb T}$
yields triangulations of $M$ copies of $\Omega$ and a piecewise
polynomial approximation $\Gamma$ of $\gamma$ defined below.

%---------------------------------------------------------------------------------
\subsubsection{Finite Element Spaces and Surface Approximations}\label{ss:fem}
%---------------------------------------------------------------------------------

Any number of conforming graded bisections of 
each macro-element  $\overline{\Gamma}_0^i$ generate via $(X_0^i)^{-1}$ a conforming partition of the local parametric domain 
$\Omega\subset\R^{d}$ denoted \gls{not:HatT}$\widehat \T^i(\Omega)$ or simply $\widehat \T^i$.
For $n\geq 1$, let $\widehat \V^i:=\V(\widehat \T^i)$ be the finite element space of 
globally continuous piecewise polynomials of degree $\leq n$  on $\Omega$ subordinate to the partition $\widehat{\T}^i$,
and let $I_{\T^i}: C^0(\overline{\Omega})\to \widehat{\V}^i$
(resp.\ $I_{\T^i}: C^0(\overline{\Omega})^d \to (\widehat{\V}^i)^d$)
be the Lagrange interpolation operator of scalar
functions (resp.\ of vector-valued functions). We next define
\[
\gls{not:Xi}X^i_{\T^i} := I_{\T^i} \chi^i,
\qquad
\gls{not:Gammai}\Gamma^i := X^i_{\T^i}(\Omega),
\qquad
\gls{not:T}\T^i := \big\{T:=X^i_{\T^i}(\widehat T) \, : \, \widehat T\in \wT^i \big\}
\]
to be the piecewise polynomial interpolation of $\chi^i$ and
$\gamma^i$, and their associated mesh.

We now define the corresponding global quantities.
The {\em global parametric space} \gls{not:OmegaM}$\Omega^M$ consists of $M$ identical  copies of the local parametric space $\Omega$.
Its subdivision is denoted $\widehat \T$ and is defined as
$$
\widehat \T:= \cup_{i=1}^{M} \wT^i.
$$
Each triangulation $\overline{\T}\in\overline{\mathbb T}$ uniquely determines $\widehat \T$, so we can define the forest  
$$
\widehat{\mathbb T}:=\mathbb T(\widehat{\T}_0):= \{ \widehat
\T \ : \ \overline{\T} \in \overline{\mathbb T} \}. 
$$
Notice that $\widehat{\mathbb T}$ does not correspond necessarily to
$M$ copies of the same forest, it is rather a set of $M$ different \rhn{but compatible} forests.
Indeed, the bisection rule is governed by the topology of $\overline{\mathcal T}_0$ and dictates which initial bisection of each separate $\Omega$ is performed.
\rhn{Moreover, refinement of a macroelement in $\overline{\mathcal T}_0$ induces a partition of its boundary which must be compatible with refinements of adjacent macroelements.}
Similarly, the global subdivision $\T$ is given by
$$
\T:=\cup_{i=1}^{M} \T^i
$$
and
$$
\gls{not:TT}\mathbb T:=\mathbb T(\T_0):= \{ \T \ : \ \overline{\T} \in \overline{\mathbb T} \}; 
$$
\manel{note that $\T_0 = \overline{\T}_0$ only for polynomial degree $n=1$}.
The global piecewise polynomial surface $\Gamma$
and parametrization $\bX_{\T}$ of $\Gamma$ 
are then given by
$$
\Gamma := \Gamma_{\T} := \cup_{i=1}^{M} \Gamma^i,
\;\qquad
\bX_{\T}:=\big\{ X^i_{\T^i} \big\}_{i=1}^M.
$$
\manel{At this point, we remark that $\bX_0$ and $\bX_{\T_0}$ are, in general, different maps; the first one is the nodal interpolant of $\bXi$ into linears whereas the last one is the parametrization of $\Gamma_{\T_0}$ }.
Moreover, we say that $(\T,\Gamma)$ is a pair of mesh-surface approximation when $\T\in \mathbb T$ and $\Gamma=\Gamma_\T$.
Also, for a subdivision $\T \in \mathbb T$, we denote by $S_\T$ the set of interior faces (edges if $d=2$).
%{\bf AB: I do not think we should define $I_{\T}$ and ${X_{\T}}$ this is confusing when we drop the index $i$ and I do not think we ever use them. P: ok, I removed them.}
%
Finally, we define the finite element space over $\T$
\begin{equation}\label{d:fem_space}
\begin{split}
\gls{not:FEM}\V(\T) :=\{ V \in C^0(\Gamma) \; : \; 
	  & V|_{\Gamma^i} \mbox{ is the lift of some } \widehat V^i \in \widehat \V^i \mbox{ via }    
	    X_{\T}^i,  \\
	  &  
	     \mbox{with } V= 0 \mbox{ on } \partial \Gamma \mbox{ or } \int_\Gamma V = 0 \mbox{ if } \partial \Gamma = \emptyset 
	    \},
\end{split}
\end{equation}
and observe that functions in $\V(\T)$ are not piecewise polynomials.

The refinement procedure consists of bisecting elements in $\overline{\T}_0$ 
and propagating its effects on $\widehat \T$ and $\T$  via the
mappings \manel{ $\bX_0^{-1}$ and
$\bX_{\cT}\circ \bX_0^{-1}$}, respectively.
Figure~\ref{f:surf_refinement} depicts one bisection refinement for
$d=2$.
\begin{figure}[ht!]
\centerline{\includegraphics[width=0.9\textwidth]%
             {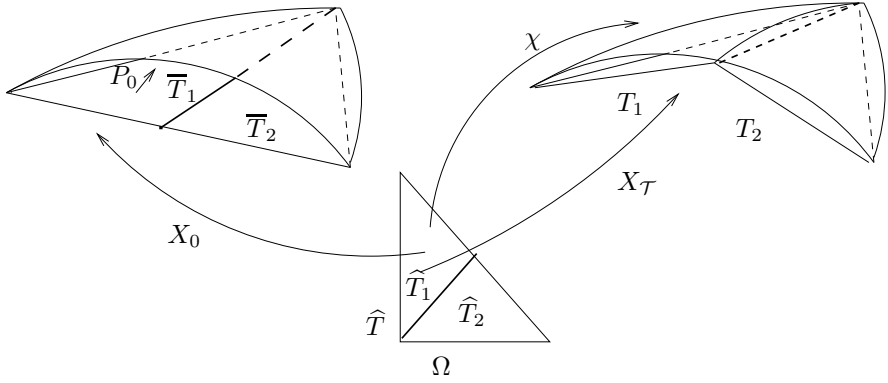}}
\begin{picture}(0,0)
\put(180,0){$\Omega$}
\put(155,15){$\widehat T$}
\put(170,30){$\widehat T_1$}
\put(190,20){$\widehat T_2$}
\put(80,50){$X_0$}
\put(250,70){$ X_{\T}$}
\put(250,100){$T_1$}
\put(295,90){$T_2$}
\put(58,110){$P_0$}
\put(80,105){$\overline{T}_1$}
\put(110,90){$\overline{T}_2$}
\put(215,125){$\chi$}
\end{picture}
\caption{\footnotesize{Effect of one bisection of the macro-element
$X_0(\Omega)$ when $d=2$ and $n=1$; the superscript $i$ is omitted for simplicity. 
(Left) A triangle \rhn{$\overline{T} \in \overline{\T}_0=\T_0$} is split into two
triangles $\overline{T}_1$, $\overline{T_2} \subset \mathbb R^3$. (Bottom) Equivalently, via the affine map $X_0^{-1}$,  the corresponding triangle $\widehat T \in \widehat \T$ is split into two
triangles $\widehat T_1$, $\widehat T_2 \subset \mathbb R^2$,
whereas (Right)
$\gamma$ is interpolated by a new piecewise linear surface
$\Gamma:=X_{\T}(\Omega)$, with $X_{\T} = \mathcal
I_{\T}\chi$ the piecewise linear interpolant of the parametrization $\chi$ defined
in $\Omega$ and subordinate to the new triangulation $\widehat \T$. 
The images via $X_{\T}$ of $\widehat T_1$ and $\widehat T_2$ are denoted $T_1$ and $T_2$ respectively;
they are affine when $n=1$.
}}
\label{f:surf_refinement}
\end{figure}
For $\overline \T$, $\overline \T_* \in \overline{\mathbb{T}}$, we use the notation $\overline{\T}_* \geq \overline{\T}$ to indicate that $\overline{\T}_*$ is a conforming refinement of $\overline{\T}$. 
In addition, slightly abusing the notation, given two subdivisions $\T$, $ \T_* \in \mathbb{T}$, we write
$\T_* \geq \T$ to indicate that $\overline{\T_*} \geq \overline{\T}$.
Notice that given $\T, \T_* \in \mathbb{T}$, with $\T_* \geq \T$, the finite element space $\V(\T)$ is not a subspace of $\V(\T_*)$ since the associated surface approximation $\Gamma$ and $\Gamma_*$ do not match.
This lack of consistency is accounted for in the discussion below
taking advantage of the nested
property \manel{$\V(\widehat \T^i) \subset  \V(\widehat \T_*^i) $} for all $1\le i \le M$.
%over \modif{\sout{$\Omega^M$}}.
  
At this point, the three different subdivisions $\overline \T$, $\widehat \T$ and  $\T$ are defined.
Notice that any of these three subdivisions uniquely determines the other two, which is repeatedly used in this work.
In practice only $\overline \T$ is required and the other two are
recovered using the mappings \manel{$\bX_0^{-1}$} and
\manel{$\bX_{\cT}\circ \bX_0^{-1}$}.
However, they have theoretical different purposes:
the subdivision $\overline{\T}$ is made of flat faces obtained as
refinement of the initial polyhedral surface and drives the refinement
procedure; $\widehat \T$ is the triangulation on the
parametric space and it is used to evaluate some quantities (geometric
estimator, oscillation, etc.) associated to the AFEM in a nested
framework; $\T$ is made of curved faces and is the subdivision defining
$\Gamma=\Gamma_\T$ where the approximate PDE is solved. 

%-------------------------------------------------------------------------------
\subsubsection{Stability of the Lagrange Interpolation Operator}\label{S:Lagrange-stab}
%-------------------------------------------------------------------------------
%
The Lagrange interpolation operator is instrumental to define the
approximate surface and will be central in the definition of the geometric estimator
in \S~\ref{ss:geom_estim}. The following lemma discusses its local
stability in Besov space ${B^s_p(L_q)} $ and Sobolev space $W^1_\infty$. We refer to
\S~\ref{S:Besov} for a definition of the Besov seminorms.

\begin{lemma}[Local stability of Lagrange interpolation]\label{L:interpolation}
Let $\widehat \T$ be any conforming refinement of $\widehat \T_0$. 
If $s>0$ and $0<p,q\le\infty$ satisfy $s>d/q$ and $s\le n+1$, then the
Lagrange interpolation operator $I_{\T}$ (with polynomial degree
$n\ge1$) is stable in $B^s_p(L_q(\widehat T))$, namely there exists a constant 
$C$ depending on $s,d,p,q$ and $n$ such that
\begin{equation}\label{interpolation}
| I_{\T} v|_{B^s_p(L_q(\widehat T))} \le C |v|_{B^s_p(L_q(\widehat T))}
\qquad\forall v\in B^s_p(L_q(\widehat T)),
\end{equation}
for any $\widehat T \in \wT$ and $\wT \in \widehat {\mathbb T}$. The same bound
is valid in $W^1_\infty(\widehat T)$, i.e.
\begin{equation}\label{interpolation-W1}
\|\wnabla I_{\T} v\|_{L_\infty(\widehat{T})} \le
C \| \wnabla v\|_{L_\infty(\widehat{T})}
\qquad\forall v\in W^1_\infty(\widehat{T}).
\end{equation}
\end{lemma}
\begin{proof}
We consider an arbitrary element $\widehat T \in \wT$.
If  $m<s\le m+1$, with $0\le m\le n$, then for any $P \in \P_m(\widehat T)$
(the space of polynomials of degree $\le m$ over $\widehat T$)
we have $P = I_{\T} P$ and the following holds
\[
|I_{\T} v|_{B^s_p(L_q(\widehat{T}))} = |I_{\T} (v - P)|_{B^s_p(L_q(\widehat{T}))}
\Cleq \rhn{|\widehat{T}|^{-s/d}} 
\|I_{\T} (v - P)\|_{L_q(\widehat{T})}, 
\]
where in the last inequality we use an inverse estimate for Besov
semi-norms, which is obtained by scaling.
%\cite[\S 4.4]{GM:14}. 
We now introduce $w := v - P$ and estimate  $\|I_{\T} w
\|_{L_q(\widehat T)}$.
First, scaling to the reference simplex $\widehat T_R$ we get
\[
\|I_\T w \|_{L_q(\widehat{T})} \Cleq \rhn{|\widehat{T}|^{1/q}} 
\| w \|_{L_\infty(\widehat{T}_R)} \Cleq
\rhn{|\widehat{T}|^{1/q}}  \| w \|_{B^s_p(L_q(\widehat {T}_R))}
\]
because $B^s_{\manel{p}}(L_q(\widehat{T}_R))$ is embedded in $L_\infty(\widehat{T}_R)$
in view of $sq>d$. Notice that we do not distinguish between the
function $w$ defined on $\widehat T$ and the corresponding one defined
on the reference element $\widehat T_R$. We recall that 
$
\|w\|_{B^s_p(L_q(\widehat{T}_R))} 
= \|w\|_{L_q(\widehat{T}_R)} + |v|_{B^s_p(L_q(\widehat{T}_R))},
$
according to the definitions of \S~\ref{S:Besov}, and scale back to $\widehat T$:
\[
\rhn{|\widehat{T}|^{1/q}}  \| w\|_{B^s_p(L_q(\widehat T_R))} 
\Cleq \|w\|_{L_q(\widehat{T})} + \rhn{|\widehat{T}|^{s/d}} |w|_{B^s_p(L_q(\widehat{T}))}.
\]
Combining previous estimates with the immediate generalization
of~\cite[Lemma 4.15]{GM:14}
\[
\inf_{P \in \P_m} \| v - P \|_{L_q(\widehat T)} \Cleq
|\widehat T|^{\rhn{s/d}} | v |_{B_p^s(L_q(\widehat T))},
\]
and the property $|P|_{B^s_p(L_q(\widehat{T}))}=0$,
we conclude the desired result \eqref{interpolation}
\begin{align*}
|I_{\T} v|_{B^s_p(L_q(\widehat{T}))} 
\Cleq 
\rhn{|\widehat{T}|^{-s/d}} 
\inf_{P \in \P_m} 
 \|v - P \|_{L_q(\widehat{T})} +  |v|_{B^s_p(L_q(\widehat{T}))} 
\Cleq |v|_{B^s_p(L_q(\widehat{T}))}.
\end{align*}
To prove the stability bound \eqref{interpolation-W1}, we take
advantage of the representation $I_{\T} v = \sum_{j=1}^{d+1}
v(\bz_j) \phi_{\bz_j}$ in terms of the canonical basis functions
$\phi_{\bz_j} \in \mathbb P_n(\widehat T)$.
Then
\[
\wnabla I_{\T} v 
= \sum_{j=1}^{d+1} \big( v(\bz_j) - v(\bz_l)\big) \wnabla
\phi_{\bz_j}
\qquad 1\leq l\leq d+1,
\]
where we exploit that $\{\phi_{\bz_j} \}_j$ is a partition
of unity over $\widehat{T}$, ie. $\sum_{j=1}^{d+1}  \phi_{\bz_j} =1$.
Since any sequence of meshes in the flat parametric domain $\Omega$
obtained by successive bisections is shape regular, using inverse
estimates
in $\widehat T$ and interpolation in $L_\infty(\widehat T)$ yields
\begin{align*}
\|\wnabla I_{\T} v\|_{L_\infty(\widehat T)}
\le \max_{1\le j\le d+1} \big|v(\bz_j) - v(\bz_l)\big|
\sum_{j=1}^{d+1} \|\wnabla \phi_{\bz_j}\|_{L_\infty(\widehat T)}
\le C 
\|\wnabla  v\|_{L_\infty(\widehat T)},
\end{align*}
where $C>0$ is a geometric constant independent of $\gamma$ and
proportional to the sum
$\sum_{j=1}^{d+1} \|\wnabla\phi_{\bz_j}\|_{L_\infty(\widehat T)}$
which depends only on $n$ and $d$. 
This concludes the proof.
\end{proof}

%-------------------------------------------------------------------------------
\subsubsection{Shape Regularity and Geometric Estimators}\label{ss:geom_estim}
%-------------------------------------------------------------------------------
%
The proof of Lemma \ref{L:interpolation} utilizes direct
and inverse estimates that rely on the shape regularity of elements
$\widehat T \in \wT \in \widehat{\mathbb T}$. A discussion about
shape regularity of the forests $\overline{\mathbb T}$,
$\widehat{\mathbb T}$ and $\mathbb T$ is in order.

The forest $\overline{\mathbb T}$ induced by bisection on the flat
faces of the initial subdivision $\overline{\T}_0$  
is shape regular \cite{BiDaDeV:04,NoSiVe:09,Stevenson:08} and so is its counterpart $\widehat{\mathbb T}$ on the 
parametric domain.
Regarding  the forest $\mathbb T$, the question is more subtle and we start with a definition.

\begin{definition}[Shape regularity]\label{d:shape_reg}\gls{def:shapereg}
We say that the class of conforming meshes $ \grids$ is {\it shape regular} if there is a
constant $C_0$ only depending on $\overline{\T_0}$, such that for all $\wT \in \widehat\grids$, and all $i=1,...,M$,
\begin{equation}\label{shape-regular}
C_0^{-1} |\hat{\bx}-\hat{\by}| 
\le |X^i_{\T^i}(\hat{\bx}) - X^i_{\T^i}(\hat{\by})| \le C_0
|\hat{\bx}-\hat{\by}|
\qquad\forall \, \hat{\bx},\hat{\by}\in \widehat T, \quad \forall \,
\widehat T\in \widehat \T^i.
\end{equation}
\end{definition}

We have already noted that $\widehat{\mathbb T}$ is shape regular and
observe that \eqref{shape-regular} states that the deformation of
$\widehat T \in\widehat \T^i$ leading to \rhn{$T= X_{\T^i}^i(\widehat T)\in\T^i$} does not
degenerate.
 We also point out that \eqref{shape-regular} 
implies the usual condition on the Jacobian
$\widehat{\nabla} X^i_{\T^i}$, valid for a.e. $\hat{\bx}\in\Omega$,
\begin{equation}\label{jacobian-F}
C_0^{-1}|\bw| \le |\widehat{\nabla} X_{\T^i}^i(\hat{\bx}) \, \bw| \le C_0 |\bw|
\qquad\forall \, \bw\in\mathbb R^d,
\end{equation}
and that $\widehat{\nabla} X_{\T^i}^i$ happens to be constant  
on $\widehat T$ for an affine map $X_{\T^i}^i$ \cite{Cpg78}.

We stress that a bi-Lipschitz parametrization
satisfying \eqref{bi_lipschitz} does {\it not} guarantee that $\mathbb
T$ is shape regular.
This issue has been tackled in \cite{BP:11}
assuming that the surface $\gamma$ is $W^2_\infty$ and $\overline{\T}_0$ is sufficiently
fine. We present a similar result in Lemma \ref{l:shape_reg},
invoking piecewise $C^1$-regularity of
$\gamma$, which hinges on the quasi-monotonicity of the geometric
estimator $\lambda_{\cT}(\gamma)$, which we prove first in Lemma~\ref{L:quasi-mono-n}.
We start with the definition of $\lambda_{\cT}(\gamma)$. Since there is a 
one-to-one correspondence between subdivisions $\T \in \mathbb T$
defining the surface interpolant $\Gamma=\Gamma_\T$ and
subdivisions $\widehat \T \in \widehat{\mathbb T}$ of $M$ copies of
the parametric domain $\Omega$, we define $\lambda_{\cT}(\gamma)$ \rhn{as follows.}
For $1\leq i \leq M$ and $T\in \T^i$ ($\widehat T\in \widehat \T^i$), let the
\textit{geometric element indicator} be
\begin{equation}%\label{p:geom_osc}
\gls{indi:lambdalocal}\lambda_{\T^i}(\gamma,T) := 
\norm{\smash { \wnabla (\chi^{i}-X_{\T^i}^i)}}_{ L_\infty(\widehat{T})}
= \rhn{
\norm{\smash { \wnabla (\chi^{i}- I_{\T^i} \chi^{i})}}_{ L_\infty(\widehat{T})},
}
\end{equation}
and the corresponding \textit{geometric estimator} be
\begin{equation}\label{geo-estimator}
\gls{indi:lambda}\lambda_{\cT}(\gamma) :=\max_{i=1,...,M} \max_{T\in \T^i}\lambda_{\T^i}(\gamma,T).
\end{equation}
%
%The geometric estimator may not decrease upon refinement
\rhn{It is worth mentioning that this quantity could increase upon refinement}, especially
in the pre-asymptotic regime, but the following 
{\it quasi-monotonicity} property is valid instead.

\begin{lemma}[Quasi-monotonicity of the geometric
    estimator]\label{L:quasi-mono-n}
There exists a constant $\Lambda_0>1$, solely depending on \manel{$\overline{\T}_0$}, 
the polynomial degree $n$, and dimension $d$, such that 
\begin{equation}\label{quasi-mono-n}
\gls{const:Lambda0}\lambda_{\T_*}(\gamma) \le \Lambda_0 \lambda_{\T}(\gamma)
\end{equation}
for any $\T, \T_* \in {\mathbb T}$ with
$\T_* \geq \T$. This bound holds elementwise as well.
\end{lemma}
\begin{proof}
We consider an arbitrary element $\widehat T\in \widehat \T$,
i.e. $\widehat T\in \widehat \T^i$ for some $i$, but we do not write
explicitly the superscript $i$.
We further observe that
the Lagrange interpolation operator $I_{\T_*}$ is invariant on 
polynomials of degree $\le n$ over $\widehat T$, whence
\begin{equation*}
\|\wnabla(\chi-I_{\cT_*}\chi)\|_{L_\infty(\widehat{T})}
= \|\wnabla(\chi-I_{\cT}\chi) -
\wnabla I_{\cT_*}(\chi-I_{\cT}\chi)\|_{L_\infty(\widehat{T})}.
\end{equation*}
From the local $W^{1}_\infty$ stability bound \eqref{interpolation-W1}
of $I_{\T_*}$, we deduce the existence of $C>0$, solely
depending on $\overline {\T}_0$, $d$ and $n$, such that
\[
\|\wnabla I_{\cT_*} \big(\chi-I_{\cT}\chi \big)\|_{L_\infty(\widehat{T})} \le
C \| \wnabla \big( \chi -I_{\cT}\chi\big)\|_{L_\infty(\widehat{T})}.
\]
The desired estimate \eqref{quasi-mono-n} thus follows with $\Lambda_0=1+C$.
\end{proof}

This result turns out to be critical not only for Lemma
\ref{l:shape_reg} below,
which guarantees the shape regularity of $\mathbb T$, but
also to control the possible increase of the geometric estimator due to the 
\ADAPTPDE calls. 
We reiterate that bi-Lipschitz parametrizations
satisfying \eqref{bi_lipschitz} do {\it not} guarantee that
$\mathbb T$ is shape regular.

\begin{lemma}[Shape regularity]\label{l:shape_reg}
The forest $\mathbb T = \mathbb T(\T_0)$ is shape-regular
with constant $C_0=2L$ provided 
\begin{equation}\label{init_sub_cond}
\lambda_{\T_0}(\gamma)\leq \frac 1 {2\Lambda_0 L}, 
\end{equation}
with $L\ge1$ the non-degeneracy constant in \eqref{bi_lipschitz} and
$\Lambda_0>0$ the constant in \eqref{quasi-mono-n}.

\end{lemma}
\begin{proof}
Let $\T\in\grids$ be an arbitrary mesh.
For any $T\in \T$, we recall that $T$ belongs to a mesh patch $\T^i$
for some $i$, which we do not write explicitly.
Let $\widehat T$ be the corresponding element in $\widehat \T$.
Since for $\widehat \bx$, $\widehat \by \in \widehat T$
\[
|(\chi-X_{\T})(\widehat \bx) - (\chi-X_{\T})(\widehat \by)| \le 
|\widehat \bx- \widehat \by| \|\wnabla (\chi-X_{\T})\|_{L_\infty(\widehat T)} = 
|\widehat \bx- \widehat \by| \lambda_{\cT}(\gamma, T),
\]
the shape-regularity assertion is a consequence of \eqref{bi_lipschitz}
and \eqref{quasi-mono-n}.
\end{proof}

\manel{
\ab{Note that once shape-regularity
of the forest $\grids$ is established, the triangulations $\T_0$ and $\overline{\T_0}$ 
are equivalent}. Therefore, with a slight abuse of notation, hereafter we regard
$\T_0$ as the initial triangulation of the AFEM.
}

We refer to \cite[Figure 11]{BCMN:Magenes} for an intermediate degenerate situation
in which $\lambda_{\T_1}(\gamma)>(2\Lambda_0 L)^{-1}$ and
\eqref{init_sub_cond} is violated
for polynomial degree $n=1$.

%%%%%%%%%%%%%%%%%%%%%%%%%%%%%%%%%%%%%%%%%%%%%%%%%%%%%%%%%%%%%%%%%%%%%%%%%%%%%%%
\section{The Laplace-Beltrami Operator}\label{S:Laplace-Beltrami}
%%%%%%%%%%%%%%%%%%%%%%%%%%%%%%%%%%%%%%%%%%%%%%%%%%%%%%%%%%%%%%%%%%%%%%%%%%%%%%%

In this section, we start the discussion with basic differential geometry properties leading to the definition of the Laplace-Beltrami operator $\Delta_\gamma$ together with other relevant geometric operators.
We then derive a weak formulation of $-\Delta_\gamma u = f$ as well as its finite element counterpart.
We assume $\gamma$ to be \rhn{globally $W_\infty^1$ and} piecewise $C^1$, 
i.e., $\chi^i\in C^1(\overline\Omega)^{d+1}$ for all
$1\le i\le M$, and $\Gamma$ denote its piecewise polynomial approximation.
In the discussion below we \rhn{often} remove the superscript $i$, 
because no confusion is possible. \manel{We also note that the hat on a function such as $\hat{v}$
will be omitted when the domain is explicit in the formula or when it appears 
on an operator being applied to the function.}

%--------------------------------------------------------------------------------
\subsection{Basic Differential Geometry}\label{S:diff-geom}
%--------------------------------------------------------------------------------
%
In this subsection we recall a matrix formulation of some basic
differential geometry facts and refer to \cite{BCMN:Magenes} for details. 
Our first task is to relate the gradient $\widehat{\nabla}$ in the 
parametric domain $\Omega$ with the tangential gradient
$\nabla_\gamma$ on $\gamma$. To this end, let $\bG\in \R^{(d+1)\times d}$ be the matrix
$$
\bG := \bG_{\gamma} := [\widehat{\partial}_1\chi, \ldots, 
\widehat{\partial}_d\chi],
$$
whose $j$-th column $\widehat{\partial}_j\chi\in\R^{d+1}$ 
is the vector of partial derivatives of $\chi$
with respect to the $j^{th}$ coordinate of $\Omega$. Since
$\chi$ is a diffeomorphism, the set $\{\widehat{\partial}_j\chi\}_{j=1}^d$ of
tangent vectors to $\gamma$ is well defined, linearly
independent, and expands the tangent hyperplane to each $\gamma^i$ 
at interior points for all $1\le i\le M$. The {\it first fundamental
form} of $\gamma$ is the symmetric and positive definite matrix
$\bg\in\R^{d\times d}$ defined by
\begin{equation}\label{1st-form}
\gls{not:first_fund}\bg %= \big( g_{\gamma,ij} \big)_{1\leq i,j \leq d} 
:= \big(
\widehat{\partial}_i\chi^T\widehat{\partial}_j\chi \big)_{1\leq i,j \leq d} = \bG^T\bG.
\end{equation}
Given $v : \gamma \to \R$, the tangent gradient
$\nabla_\gamma v(\tilde{\bx})
=\sum_{i=1}^d \alpha_i(\hat{\bx})\widehat{\partial}_i\chi(\hat{\bx})^T$ 
\rhn{is a row vector that} satisfies the \rhn{chain rule}
\begin{equation}\label{grad-hat}
\widehat{\nabla} v = \DivG {v} \, \bG .
\end{equation}
To get the reverse relation, we augment $\bG$ to the matrix $\widetilde{\bG} \in
\R^{(d+1)\times (d+1)}$ by adding the (outer) unit normal
$\manel{\bnu}=(\nu_1,\cdots,\nu_{d+1})^T\in\R^{(d+1)}$ 
to the tangent hyperplane
$\text{span}\{\widehat{\partial}_i\chi\}_{i=1}^d$ to $\gamma$
as the last column, namely
\[
\widetilde{\bG} := \br{ \bG, \bnu} = \big[\widehat{\partial}_1\chi, \dots,
\widehat{\partial}_d\chi, \bnu \big].
\]
Since $\widetilde{\bG}$ is invertible, we let $\widetilde{\bD} =
\widetilde{\bG}^{-1}$ and use~\eqref{grad-hat} to realize that
\begin{equation}\label{grad-gamma}
\gls{not:surf_grad}\DivG {v} = \DivG {v} \, \widetilde{\bG} \, \widetilde{\bD} =\big[
\widehat{\nabla} v, 0\big] \widetilde{\bD} = \widehat{\nabla} v\; \bD,
\end{equation}
where $\bD \in \R^{d\times (d+1)}$ results from
$\widetilde{\bD}$ by suppressing its last row. Moreover, 
the first fundamental form $\bg$ has
inverse $\bg^{-1} = \bD\bD^T$. We let 
\begin{equation}\label{area}
\gls{not:area}q := \sqrt{\det \bg}
\end{equation}
be the area element of $\gamma$ and point out the change of variables formula
\begin{equation}\label{change-var}
\int_\omega v q = \int_{\chi(\omega)} v,
\end{equation}
for any $\omega \subset \Omega$ measurable.
When $\chi$ is $C^2$ and $v\in H^2(\gamma)$ ($\widehat v \in H^2(\Omega)$), we have the compact expression for the Laplace-Beltrami operator on $\gamma$
\begin{equation*}
\gls{not:LB}\Delta_\gamma v = \frac{1}{q} \widehat{\text{div}} 
\big(q \widehat{\nabla} v\bg^{-1}\big).
\end{equation*}
The above representation is instrumental to derive the following integration by parts formula on surfaces
\begin{equation}\label{parts-discrete}
\int_{\gamma} \nabla_\gamma w \nabla_\gamma^T v =
\int_{\gamma} -\Delta_\gamma w \, v + \int_{\partial \gamma}
\nabla_\gamma w \, \bn~ v
\qquad\forall\, v,w \in H^2(\gamma),
\end{equation}
where $\bn$ is the unit co-normal on $\partial\gamma$ pointing outside $\gamma$.

The discussion above applies as well to the piecewise polynomial surface
$\Gamma$ (recall that we dropped the index $i$ specifying the underlying patch). 
We denote the corresponding matrices $\bG_\Gamma=\widehat{\nabla}X_{\cT}$ and
$\bD_\Gamma$ associated with $X_{\cT}:\Omega\to\Gamma$, and get
\begin{equation}\label{grad-Gamma}
\gls{not:surf_grad}\nabla_\Gamma v = \widehat{\nabla} v \, \bD_\Gamma.
\end{equation}
The first fundamental form $\bg_\Gamma$ of $\Gamma$ and its elementary area
$q_\Gamma$ are defined by
\begin{equation}\label{1st-form-discrete}
\gls{not:first_fund}\bg_\Gamma := \bG_\Gamma^T \, \bG_\Gamma,
\qquad
\gls{not:area}q_\Gamma := \sqrt{\det \bgG_\Gamma},
\end{equation}
\rhn{and the (outer) unit normal to $\Gamma$ is denoted by $\bnu_\Gamma$}. The corresponding expression of the Laplace-Beltrami operator  is
\begin{equation}\label{lap-bel}
\gls{not:LB}\Delta_\Gamma V = \frac{1}{q_\Gamma} \widehat{\text{div}} 
\big(q_\Gamma \widehat{\nabla} V\bg_\Gamma^{-1}\big),
\end{equation}
and only makes sense elementwise.
In addition, we recall that for $T\in \T$ and $S$ a side of $T$,
the unit co-normal $\bn_\Gamma$ on $S$ pointing outside $T$ satisfies
\begin{equation}\label{eq:hat_n}
\gls{not:normal}  \widehat \bn =  \frac{r_\Gamma}{q_\Gamma} \bG_\Gamma^T \bn_\Gamma,
  \qquad  \bn_\Gamma  = \frac{q_\Gamma}{r_\Gamma} \bD_\Gamma^T \widehat \bn
\end{equation}
where \gls{not:area}$r_\Gamma$ is the area element associated with the subsimplex
$\widehat S:=X_\T^{-1}(S)$ (see \cite{BCMN:Magenes} for a detailed expression).
Hence, \eqref{eq:hat_n} and \eqref{grad-Gamma} give the following local expression for the tangential derivative of $v$ in the direction $\bn_\Gamma$ on $S$ 
\begin{equation}\label{e:normal_grad}
\nabla_\Gamma v \,\bn_\Gamma = \frac{q_\Gamma}{r_\Gamma}
\widehat{\nabla} v \bg_\Gamma^{-1}|_{\widehat S} \, \widehat \bn.
\end{equation}
This is of particular importance when considering residual type
estimators as in the present work; see \S~\ref{S:aposteriori}.

%--------------------------------------------------------------------------------
\subsection{Variational Formulation and Galerkin Method}\label{S:variational}
%--------------------------------------------------------------------------------
%
We start by introducing relevant Lebesgue and Sobolev spaces.
Let
\[
 L_{2,\#}(\gamma) := \Big\{v \in L_2(\gamma) \; : \; \int_\gamma v = 0 \quad
   \text{if} \quad \partial \gamma = \emptyset  \Big\}
\]
be the subspace of $L_2(\gamma)$ of functions with vanishing meanvalue
whenever the surface $\gamma$ is closed, and let $H^1_\#(\gamma)$ be
the subspace of $H^1(\gamma)$ given by
\begin{equation*}
\begin{split}
 H^1_\#(\gamma) := \Big\{v\in L_{2,\#}(\gamma) \; : \; 
      & 
         \DivG v^i\in [L_2(\gamma^i)]^{d+1} ,\\
      &   v^i = v^j \ \text{on} \ \gamma^i\cap\gamma^j
          \ 1 \le i,j\le M, 
         \, v=0 \ \text{on} \ \partial \gamma  \Big\},
\end{split}
\end{equation*}
where $\nabla_\gamma$ and traces of $v^i = v_{|\gamma^i}$ are well defined in each component $\gamma^i$ due
to \eqref{grad-gamma}.
Let the weak form of the
{\it Laplace-Beltrami operator} $\Delta_\gamma v$ for
any function $v\in H^1_\#(\gamma)$ be
\begin{equation}\label{weak-LB}
\langle -\LapG v, \varphi \rangle := 
\sum_{i=1}^{M}\int_{\gamma^i} \DivG v^i  \, \DivG^T \vphi^i
\qquad \forall \varphi \in H^1_\#(\gamma),
\end{equation}
where $\langle \cdot,\cdot  \rangle$ denotes the 
$(H^1_\#(\gamma))^*$-$H^1_\#(\gamma)$ duality product. 

We now build on \eqref{weak-LB} and write the weak formulation of 
$-\Delta_\gamma u = f$ as follows:
given $f\in L_{2,\#}(\gamma)$, we seek $u \in H^1_\#(\gamma)$ satisfying
\begin{equation}\label{p:Weak_PdeGm}
 \int_{\gamma} \nabla_{\gamma}u  \, \nabla_{\gamma}^T \vphi 
=  \int_\gamma f \, \vphi ,
\qquad \forall \; \vphi \in H^1_\#(\gamma),
\end{equation}
where we have written $\int_{\gamma} \nabla_{\gamma}u \, \nabla_{\gamma}^T \vphi $ 
to denote $\sum_{i=1}^{M} \int_{\gamma^i} \nabla_{\gamma}u^i \nabla_{\gamma}^T \vphi^i $.
Existence and uniqueness of a solution $u\in H^1_\#(\gamma)$ is a consequence of
the Lax-Milgram theorem provided $\gamma$ is Lipschitz.

When $\chi^i$ is $C^2$ and $u\in
H^2(\gamma^i)$ for each $1\le i\le M $, we showed in \cite{BCMN:Magenes} that \rhn{in the interior of} each component $\gamma^i$\rhn{, namely $\chi^i(\interior(\Omega))$,} we have
\begin{equation}\label{strong-form}
-\Delta_{\gamma} u^i = f^i, %\quad\text{in $\interior(\gamma^i):= \chi^i(\interior(\Omega))$},
\quad 1\le i\le M,
\end{equation}
together with vanishing jump conditions at the interfaces $\gamma^i\cap\gamma^j$
\ab{
\begin{equation}\label{jumps-ij}
\mathcal{J}(u)|_{\gamma^i\cap\gamma^j}
:=\nabla_{\gamma^i} u \, {\bn}^i
+ \nabla_{\gamma^j} u \, {\bn}^j= 0
\qquad\forall \, 1\le i,j\le M,
\end{equation}
}
where \ab{${\bn}^i$} is the unit outer normal to $\partial\gamma^i$ in the tangent
plane to $\gamma^i$ (see~\eqref{eq:hat_n}).

Given $\T \in \mathbb T$, we next formulate an approximation to the
Laplace-Beltrami operator on
the piecewise polynomial interpolant $\Gamma=\Gamma_\T$ of $\gamma$ as follows. 
If $F_\Gamma \in L_{2,\#}(\Gamma)$ is a suitable approximation of $f$, 
then the finite element solution $U:\Gamma\to\R$ solves
\begin{align}                  \label{FEM:weakform}
U \in\V(\T): \quad \int_{\Gamma} \DivGk{} U 
\DivGk{}^T V = \int_{\Gamma} F_\Gamma \, V \qquad \forall \;
V \in\V(\T),
\end{align}
where again $\int_\Gamma g = \sum_{i=1}^M \int_{\Gamma^i} g^i$.
To this end we choose $F_\Gamma$ to be 
\begin{equation}\label{def:F}
  F_\Gamma := f \frac{q}{q_\Gamma},
\end{equation}
because this specific choice of $F_\Gamma$ satisfies
the compatibility property
\begin{equation}\label{forcing}
\int_\Gamma F_\Gamma = \int_\gamma f = 0,
\end{equation}
whenever $\gamma$ is closed, and allows us to 
handle separately the approximation of surface $\gamma$ and forcing $f$.
In particular, \eqref{FEM:weakform} admits a unique solution $U$ as a
consequence of the Lax-Milgram theorem. 

%%%%%%%%%%%%%%%%%%%%%%%%%%%%%%%%%%%%%%%%%%%%%%%%%%%%%%%%%%%%%%%%%%%%%%%%%%%%%%%%%
\section{A Posteriori Error Analysis}\label{S:aposteriori}
%%%%%%%%%%%%%%%%%%%%%%%%%%%%%%%%%%%%%%%%%%%%%%%%%%%%%%%%%%%%%%%%%%%%%%%%%%%%%%%%%
%
In order to study the discrepancy between $u$ and $U$ we need to agree
on comparing them in a common domain, say $\gamma$.
Our goal is thus to obtain a posteriori error estimates for the energy
error $\|\nabla_\gamma(u-U)\|_{L_2(\gamma)}$. This entails developing
an a priori error analysis for the interpolation error committed in
replacing $\gamma$ by $\Gamma$ in \eqref{FEM:weakform}, which is a 
sort of consistency error, and its impact on the PDE error. We are
concerned with these issues in this section and refer to 
\cite{Demlow:09,MMN:11} where they are addressed \rhn{for different
surface representations as well as
\cite{DemlowDziuk:07,BCMN:Magenes}} that discusses the case
$n=1$. We again drop the superscript $i$ that identifies the
surface patch.

%--------------------------------------------------------------------------------
\subsection{Geometric Error and Estimator}\label{S:geom-error}
%--------------------------------------------------------------------------------
%
We now quantify the error arising from interpolating $\gamma$,
the so-called {\it geometric error}. To this end we resort to the 
matrix formulation of \S~\ref{S:diff-geom} to relate the geometric
error with the geometric estimator $\lambda_{\cT}(\gamma)$ of \eqref{p:geom_osc}.

Given $T \in \T$, we will deal with the regions
$\widehat{T} \in \widehat \T$ and $\widetilde{T} \subset \gamma$ given by
\begin{align}\label{def:regions_T}
\widehat{T}:=X_{\T}^{-1}(T),
\qquad
\widetilde{T} := \chi(\widehat{T}).
\end{align}
On mapping back and forth to $\widehat T$, and using
\eqref{change-var}, we easily see that
\begin{align}   \label{eqi:l2}
\InT v = \int_{\widetilde{T}} {v} \frac{q_\Gamma}{q}.
\end{align}
\pedro{%
The consistency error stems from the different bilinear forms of the
continuous and discrete equations~\cite[Lemma 5.1]{BCMN:Magenes}. From~\eqref{change-var},
\eqref{grad-Gamma}, and \eqref{grad-hat}, we realize that
\begin{equation}\label{diff-int}
\int_{\Gamma} \DivGk{} v \DivGk{}^T w 
= \int_{\gamma}   \DivGk{} v \bG\bD_\Gamma\bD_\Gamma^T\bG^T \DivGk{}^T w \frac{q_\Gamma}{q}.
\end{equation}
Using that $\bG\bD=\bI-\bnu\bnu^T$ is the projection onto the
tangent plane to $\gamma$, we obtain
\begin{equation*}
\int_{\gamma} \DivG v \DivG^T w
= \int_{\gamma}   \DivGk{} v \bG\bD\bD^T\bG^T \DivGk{}^T w.
\end{equation*}
These two expressions, in conjunction with
$\bg^{-1}=\bD\bD^T$ and $\bg_\Gamma^{-1}=\bD_\Gamma\bD_\Gamma^T$,
yield
\begin{equation}\label{eq:cons_error_rep}
\int_{\Gamma} \DivGk{} v \DivGk{}^T w - \int_{\gamma} \DivG v
\DivG^T w 
= \int_{\gamma} \DivG v \bA_\Gamma
\DivG^T w
\qquad \forall \, v,w \in H^1_\#(\gamma),
\end{equation}
where \gls{not:errormat}$\bA_ \Gamma\in\mathbb{R}^{(d+1)\times (d+1)}$ stands for the following error matrix
\begin{equation}\label{eq:cons_error}
\bA_ \Gamma := \frac{1}{q} \bG (q_\Gamma \bgG_\Gamma^{-1}-q\bg^{-1})\bG^T.
\end{equation}
}
%see~\cite[Lemma 5.1]{BCMN:Magenes} for details.}

\rhn{Corollary 5.1} in \cite{BCMN:Magenes} provides the following conditional estimate on the consistency error:
If $\lambda_{\T_0}(\gamma)$ satisfies 
\begin{equation}\label{eq:init_cond}
\lambda_{\T_0}(\gamma) \le \frac{1}{6\Lambda_0L^3},
\end{equation}
then we have, for 
% all $\widehat \T \in \widehat{\mathbb T}$ (corresponding to $\Gamma := \Gamma_\T$ with $\T \in  \mathbb{T}$)
$\T \in \mathbb T$,
\begin{equation}\label{Lp:Ak}
\norm{\bA_\Gamma}_{L_\infty(\widehat{T})} \Cleq  
\lambda_{\cT}(\gamma,T)
%\lambda_{\T^i}(\widehat T)  
\qquad \forall\;
%\widehat T\in \widehat \T^i, \quad 1\leq i \leq M,
T \in \T,
\end{equation}
where the hidden constant depends on \manel{$\T_0$} and the Lipschitz 
constant $L$ of $\gamma$ appearing in \eqref{bi_lipschitz}.
The consistency error estimate \eqref{Lp:Ak} relies on the following
properties for $q_\Gamma$, $r_\Gamma$,  $\bgG_\Gamma$, $\bD_\Gamma$
and $\bnu_\Gamma$ which will be used again later.
Their proofs can be found in \cite[Lemmas 5.2 and 5.4]{BCMN:Magenes}
except that of $r_\Gamma$, 
\pedro{which is analogous and thus omitted.}

\begin{lemma}[Properties of $q_\Gamma$,
$r_\Gamma$, $\bnu_\Gamma$, $\bgG_\Gamma$ and $\bD_\Gamma$]\label{L:geom_estim}
If $\lambda_{\T_0}(\gamma)$ satisfies \eqref{eq:init_cond}, then the matrices $\bg$ and $\bgG_\Gamma$ have eigenvalues in the interval $[L^{-2},L^2]$
and $[\frac12 L^{-2}, \frac32 L^2]$, respectively.
Moreover, the forest $\grids$ is shape regular,
$L^{-d} \Cleq q,q_\Gamma \Cleq L^{d}$, and for $\mathcal T \in \mathbb T$ 
\begin{equation}\label{g-G}
\begin{aligned}
\norm{q -q_\Gamma}_{L_\infty(\widehat T)}
&+ \norm{r - r_\Gamma}_{L_\infty(\partial\widehat T)}
+\norm{\bnu-\bnu_\Gamma}_{L_\infty(\widehat T)}
\\
&+ \norm{\bg-\bgG_\Gamma}_{L_\infty(\widehat T)}
+ \norm{\bD-\bD_\Gamma}_{L_\infty(\widehat T)}
\Cleq \rhn{\lambda_{\cT}(\gamma,T)
\quad\forall \,  T \in \T}
\end{aligned}
\end{equation}
where we recall that $\Gamma=\Gamma_\T$.
\end{lemma}

We stress that if $\T_0$ does not satisfy \eqref{eq:init_cond}, then
the algorithm AFEM of \S \ref{S:AFEM} will first refine $\T_0$ to make
it comply with \eqref{eq:init_cond} without ever solving the discretized PDE. In
this sense, \eqref{eq:init_cond} is not a serious
restriction for AFEM, although necessary for the subsequent theory.
We also note that \eqref{eq:init_cond} implies \eqref{init_sub_cond}
because $L\ge1$ in \eqref{bi_lipschitz}.

We finally point out the equivalence of norms on $\gamma$ and $\Gamma$
provided \eqref{eq:init_cond} is valid.

\begin{lemma}[Equivalence of norms]\label{equi:norm} 
If $\lambda_{\T_0}(\gamma)$ satisfies \eqref{eq:init_cond}, then
%If  $\lambda_\Gamma \le \min\{\frac{1}{2\Lambda L},\frac{1}{6L^3} \}$, then
the following equivalence of norms holds 
for all  $\T\in  \mathbb{T}$  with constants depending on $\T_0$ and $L$
\begin{equation}\label{eq:equiv-norms}
\|v\|_{L_2(\widetilde{T})} \approx \|v\|_{L_2(T)}
\approx \|v\|_{L_2(\widehat T)},
\quad
|v|_{H^1(\widetilde{T})} \approx |v|_{H^1(T)}
\approx |v|_{H^1(\widehat T)}
\quad\forall\, T\in \T,
\end{equation}
where $\widehat T = X_{\cT}^{-1}(T)$ and $\widetilde T = \chi(\widehat T)$.
\end{lemma}
\begin{proof} The first assertion follows directly from
\eqref{eqi:l2} and Lemma \ref{L:geom_estim}, which implies
$L^{-2d} \Cleq\frac{q_\Gamma}{q}\Cleq L^{2d}$.
\rhn{We next rewrite the integrals in \eqref{diff-int} over $T\in\T$
and $\widetilde T$.}
%For the second equivalence, we note that \eqref{change-var}, \eqref{grad-Gamma} and \eqref{grad-hat} readily imply that 
%for $v,w\in H^1(\gamma^i)$ there holds
%\begin{equation*}
%\int_{T}  \DivGk{} v \  \DivGk{}^T w = \int_{\widetilde T} \DivG {v}
% \  \bG \bD_\Gamma  \bD_\Gamma^T \bG^T  \  \DivG^T {w} \frac{q_\Gamma}{q}
%\qquad\forall v,w\in H_\#^1(\gamma).
%\end{equation*}
This, combined with the spectral estimate given in Lemma \ref{L:geom_estim}
for $\bgG_{\Gamma}^{-1}=\bD_\Gamma \bD_\Gamma ^T$
and \eqref{jacobian-X} for $\bG=\widehat{\nabla}\chi$, yields the second equivalence.
Similar reasoning applies to $\widehat T$.
\end{proof}

%--------------------------------------------------------------------------------
\subsection{Inverse Estimates for Discrete Geometric Quantities}
\label{S:inverse-estimates}
%--------------------------------------------------------------------------------

We now establish some inverse estimates for the discrete quantities
$q_\Gamma$ and $\bgG_\Gamma$ that are instrumental to derive
Lemma \ref{L:est-reduction} (reduction of residual estimator)
and Lemma \ref{P:oscU-bound} (local decay of oscillation).
These estimates are only required when the polynomial degree
$n$ is strictly greater than $1$, which is a key distinction between this work and \cite{BCMN:Magenes}.

In the following, for $T \in \T$, we \rhn{set} \gls{not:meshsize}$h_T :=
|\widehat {T}|^{\frac{1}{d}}$ \rhn{where $\widehat T$ is defined in~\eqref{def:regions_T}}.
This choice is motivated by the resulting reduction property
after \gls{const:b}$b\ge1$ bisections of $\widehat {T}$.
\begin{equation}\label{h-reduce}
h_{T'}\le 2^{-b/d}h_T,
\end{equation}
where $T'$ is the curvilinear element corresponding to any element
$\widehat {T}' \subset \widehat{T}$.

\begin{lemma}[Inverse inequalities in $L_p$]\label{L:inverse-inequalities}
If $\lambda_{\T_0}(\gamma)$ satisfies \eqref{eq:init_cond}, then
the following estimates hold for all $1\leq p \leq \infty$, 
%$\Gamma:=\Gamma_\T, \Gamma_*:=\Gamma_{\T_*}$  with 
$\T, \T_*\in  \mathbb{T}$ and $\T\leq \T_*$,
with constants depending on $\T_0$ and $L$
\begin{align}
&\| D q_\Gamma \|_{L_p(\widehat{T})} \Cleq {h_T}^{\frac{d}{p}-1}, 
&& \| D(q_{\Gamma_*} -q_\Gamma) \|_ {L_p(\widehat{T}_*)} 
      \Cleq  {h_T}^{\frac{d}{p}} \, h_{T_*}^{-1} \, \lambda_{\cT}(\gamma,T),  \label{est-inv-Q}
\\
& \| D \bgG^{-1}_\Gamma \|_{L_p(\widehat{T})} \Cleq {h_T}^{\frac{d}{p}-1},
&&
\| D (\bgG^{-1}_{\Gamma_*} - \bgG^{-1}_\Gamma) \|_{L_p(\widehat{T}_*)} 
\Cleq {h_T}^{\frac{d}{p}} \, h_{T_*}^{-1} \, \lambda_{\cT}(\gamma,T) ,\label{est-inv-G}
\end{align}
\rhn{whenever ${T} \in {\T}$, ${T_*} \in {\T_*}$ satisfy
  $\widehat{T}_* \subset \widehat T$.}
\end{lemma}
\begin{proof}
 We start with
$q_\Gamma=\sqrt{\det \bgG_\Gamma}$ and observe that
$\partial_j q_\Gamma = \frac{1}{2 q_\Gamma} \partial_j \det\bgG_\Gamma$
 and $\det\bgG_\Gamma$ is polynomial. Using an inverse inequality for
 $\det\bgG_\Gamma$, along with the fact that $q_\Gamma$ is bounded from above
and below (see Lemma~\ref{L:geom_estim}), we obtain
\begin{align*}
\|\partial_j q_\Gamma\|_{L_p(\widehat T)} 
& \le \frac{1}{2} \|q_\Gamma^{-1}\|_{L_\infty(\widehat T)} 
\|\partial_j \det\bgG_\Gamma\|_{L_p(\widehat T)}
\\ &\Cleq
\frac{1}{h_T} \|\det\bgG_\Gamma\|_{L_p(\widehat T)}
\Cleq  \frac{1}{h_{T}}  \|q_\Gamma^2\|_{L_p(\widehat T)}
\Cleq h_{T}^{\frac{d}{p}-1}.
\end{align*}
We now deal with $q_{\Gamma_*}-q_\Gamma$ as follows. We first write
\[
\partial_j (q_{\Gamma_*} - q_\Gamma) =
\frac{1}{2} \Big( \frac{1}{q_{\Gamma_*}} - \frac{1}{q_\Gamma} \Big)
\partial_j \det\bgG_{\Gamma_*} + \frac{1}{2q_\Gamma} 
\partial_j \big( \det\bgG_{\Gamma_*} - \det\bgG_\Gamma \big),
\]
whence, for \pedro{$T \in \T$, $T_* \in \T_*$} with $\widehat{T}_* \subset \widehat T$,
\begin{equation*}
\begin{aligned}
\| \partial_j (q_{\Gamma_*} - q_\Gamma) \|_{L_p(\widehat {T}_*)} &\Cleq
\|q_\Gamma - q_{\Gamma_*}\|_{L_p(\widehat T)} \|\partial_j\det\bgG_{\Gamma_*}\|_{L_\infty(\widehat {T}_*)}
\\
&\qquad + \|\partial_j \big( \det\bgG_{\Gamma_*} - \det\bgG_\Gamma\big)\|_{L_p(\widehat {T}_*)}.
\end{aligned}
\end{equation*}
Using an inverse inequality for $\det \bgG_{\Gamma_*} -\det\bgG_\Gamma
= q_{\Gamma_*}^2 - q_\Gamma^2$, the
bounds \eqref{g-G} on $q_\Gamma$ and $q_{\Gamma*}$  in terms of
\rhn{$\lambda_{\cT}(\gamma,T)$ and $\lambda_{\T_*}(\gamma,T)$,
and the quasi-monotonicity  \eqref{quasi-mono-n}
of $\lambda_{\cT}(\gamma,T)$}, we get
\begin{equation*}
\| \partial_j (q_{\Gamma_*} - q_\Gamma) \|_{L_p(\widehat {T}_*)}
\le h_{T_*}^{-1} \|q_{\Gamma_*} - q_\Gamma \|_{L_p(\widehat T)}
\Cleq h_{T}^{d/p} h_{T_*}^{-1} \rhn{\lambda_{\cT}(\gamma,T).}
\end{equation*}

To estimate $D\bgG_\Gamma^{-1}$ we see that
$\partial_j(\bgG_\Gamma^{-1}\bgG_\Gamma) = \partial_j\bgG_\Gamma^{-1}\bgG_\Gamma +
\bgG_\Gamma^{-1}\partial_j\bgG_\Gamma = 0$, whence $\partial_j\bgG_\Gamma^{-1} = 
- \bgG_\Gamma^{-1} \partial_j\bgG_\Gamma \bgG_\Gamma^{-1}$. This, an inverse inequality for $\bgG$
and the lower bound of the eigenvalues of $\bgG_\Gamma$ in Lemma \ref{L:geom_estim} imply
\[
\| \partial_j\bgG_\Gamma^{-1} \|_{L_p(\widehat T)} 
\Cleq \| \bgG_\Gamma^{-1} \|_{L_\infty(\widehat T)}^2
\| \partial_j\bgG_\Gamma  \|_{L_p(\widehat T)} \Cleq
h_{T}^{-1} \| \bgG_\Gamma \|_{L_p(\widehat T)} \Cleq h_{T}^{\frac{d}{p} -1}.
\]
Finally, $\bgG_{\Gamma_*}^{-1} - \bgG_\Gamma^{-1}=\bgG_{\Gamma_*}^{-1} \big( \bgG_\Gamma 
- \bgG_{\Gamma_*}\big)\bgG_\Gamma^{-1}$, so that the partial
derivatives can be computed with the product rule and always keeping the 
$L_p$ norm in the middle term and the $L_\infty$ norm in the other
two. Then,  making use of some inverse inequalities together with
\eqref{g-G} and  \eqref{quasi-mono-n}, we arrive at 
\[
\|\partial_j \big(\bgG_{\Gamma_*}^{-1} - \bgG_\Gamma^{-1}\big)\|_{L_p(\widehat {T}_*)}
\Cleq h_{T_*}^{-1} \, \|\bgG_{\Gamma_*} - \bgG_\Gamma\|_{L_p(\widehat T)}
\Cleq h_{T}^{d/p}  \, h_{T_*}^{-1} \, \rhn{\lambda_{\cT}(\gamma,T) ,}
\]
as asserted.
\end{proof}

We now establish an inverse estimate in Besov spaces. We refer to
\S~\ref{S:Besov} for the definition \eqref{besov-seminorm} of the Besov seminorm
$|V|_{B^s_\infty(L_p(\widehat{T}))}$ in terms of the modulus of
smoothness of order $k=\lfloor s \rfloor + 1$
\[
  \omega_k(V,t)_p = \sup_{|h|\le t} \| \Delta_h^k V
  \|_{L_p(\widehat T)} ,
\]
where $\Delta_h^k$ are the $k$-th order differences defined in~\eqref{Delta^k}.
  
\begin{lemma}[Inverse estimate in Besov space]\label{L:inverse-Besov}
 Let $\T \in \mathbb T$ and $s>0$, $0<p\le\infty$. Then,
 the following inequality holds
  \[
  | \partial_i V |_{B^s_\infty(L_p(\widehat T))}
  \Cleq \frac1{h_{T}} | V |_{B^s_\infty(L_p(\widehat T))},\qquad
  \]
  for any $\widehat T \in \wT$, and function
  $V \in \mathbb P_n(\widehat T)$ or $V = q_\Gamma \bgG_{\Gamma}^{-T}$ (with $\Gamma=\Gamma_\T$).
  \end{lemma}

\begin{proof}
We prove the estimate  for $V \in \mathbb P_n(\widehat T)$
because dealing with $ q_\Gamma\bgG_{\Gamma}^{-T}$ reduces to
repeating the steps in the proof of
Lemma~\ref{L:inverse-inequalities} and applying the inverse inequality for
polynomials. Since the $k$-th order differences satisfy
  \[
  \Delta_h^k (\partial_i V) (\widehat \bx) = \partial_i (\Delta_h^k V)(\widehat \bx),
  \qquad\forall \widehat \bx \in \widehat T_{kh},
  \]
  and $\Delta_h^k V \in \mathbb P_n(\widehat T_{kh})$, \rhn{in view of~\eqref{Delta^k}}
  the usual inverse inequality gives
  \[
  \| \Delta_h^k \partial_i V \|_{L_q(\widehat T)}
  =  \| \Delta_h^k \partial_i V \|_{L_q(\widehat T_{kh})} = \| \partial_i\Delta_h^k V \|_{L_q(\widehat T_{kh})}
  \Cleq \frac{1}{h_{T}} \| \Delta_h^k V \|_{L_q(\widehat T)}.
  \]
Invoking the definition \eqref{besov-seminorm} yields the desired estimate.
\end{proof}

%--------------------------------------------------------------------------------
\subsection{Upper and Lower Bounds for the Energy Error}\label{S:bounds}
%--------------------------------------------------------------------------------
%
We now derive an error representation formula leading to
lower and upper \rhn{a posteriori} bounds for the energy error. 
Given $\T \in \mathbb T$, we recall the notation $\Gamma=\Gamma_\T$
and introduce the usual interior 
and jump residuals \rhn{suggested by \eqref{strong-form} and \eqref{jumps-ij}}
for arbitrary $V\in \V(\T)$
\begin{equation*}
\begin{gathered}
\RT(V,F_\Gamma) := F_\Gamma|_T + \Delta_\Gamma V|_T,
\qquad \JT(V) := \{\JS(V)\}_{S\subset\partial T}
\qquad \forall\, T\in \T,
\\
\JS(V) := \DivGk{} V^+|_S \, \bn_S^+ +
\DivGk{} V^-|_S \, \bn_S^-\qquad \, \forall\, S\in  \mathcal{S}_\T,
\end{gathered}
\end{equation*}
where, for each $\bx \in S$, $\bn_S^\pm(\bx)$ denotes the outward unit normal to $S$ and tangent to $T^\pm$ at $\bx$, and $T^+$, $T^-$ are curvilinear elements in
$\T$ that share the side $S\in\mathcal{S}_\T$; recall that  $\mathcal{S}_\T$ denotes the
set of interior faces of $T\in\T$. We emphasize that, in contrast
to flat domains, $\bn_S^+\ne -\bn_S^-$.
% because the vector may have
%different supporting hyperplanes.
Similarly, if $\bD^{\pm}_{\Gamma}$ denote the matrices associated to $T^\pm$, $\DivGk{}V^\pm|_S=\widehat{\nabla}V^\pm\bD^\pm _\Gamma|_{\widehat{S}}$ are tangential gradients of $V$ on $T^\pm$ restricted to $S$.
Moreover, according to \eqref{lap-bel}, see that
$\Delta_\Gamma V|_T=q_\Gamma^{-1}\widehat{\text{div}}
\big(q_\Gamma\widehat{\nabla} V
\bgG_\Gamma^{-1}\big)|_{\widehat{T}} \not = 0$ 
in general for $T\in \T$ when the polynomial degree $n>1$.
This is a major difference relative to \cite{BCMN:Magenes}, which deals with 
$n=1$ and $ V|_{\widehat T} \in \mathbb P_1(\widehat T)$,
$q_\Gamma \in \mathbb P_0(\widehat T)$, $\bg_\Gamma \in \mathbb
P_0(\widehat T)^{d\times d}$ imply $\Delta_\Gamma V|_T = 0$.

Subtracting  the weak formulations \eqref{p:Weak_PdeGm} and \eqref{FEM:weakform}, 
and employing \eqref{parts-discrete} to integrate by parts elementwise, we obtain for all 
\rhn{$v\in H_\#^1(\gamma)$}:
\begin{align}\label{eq:error_form}
\int_\gamma \nabla_\gamma (u - U)\, 
	\nabla_\gamma^T \, v = I_1 + I_2 +I_3,
\end{align}
with
\begin{align*}
I_1 &:= \sum_{T \in \T} \InT \RT(U,F_\Gamma) (v-V)
-\sum_{S \in \mathcal{S}_\T} \InS \JS(U)(v-V), \\
I_2 &:= \int_\Gamma \nabla_\Gamma U  \, \nabla_\Gamma^T v
	- \int_\gamma \nabla_\gamma U \,  \nabla_\gamma^T v 
	= \int_\gamma \nabla_\gamma U \bA_\Gamma \nabla_\gamma^T v, \\
I_3 &:= \int_\gamma f v - \int_\Gamma F_\Gamma v.
\end{align*}
The choice $F_\Gamma = \frac{q}{q_\Gamma}f$ of~\eqref{def:F}  
implies $I_3=0$ so that only $I_1$ and $I_2$ need to be estimated. 
Observe that $I_1$ is the usual residual 
term, whereas $I_2$ is the geometry consistency 
term \eqref{eq:cons_error_rep} and accounts
for the discrepancy between $\gamma$ and $\Gamma$.
An estimate for the error matrix $\bA_ \Gamma$ is given in \eqref{Lp:Ak}.

The PDE error indicator stems from $I_1$ and is defined as follows for any $V\in\V(\T)$ 
\begin{equation*}
\gls{indi:etalocal}\eta_\T (V,F_\Gamma,T)^2 := h_T^2 \normLtT{\RT(V,F_\Gamma)}^2 + 
h_T \normLtpT{\JT(V)}^2 \qquad\forall\, T\in \rhn{\T}.
\end{equation*}
We recall that the definition $h_T = |\widehat {T}|^{\frac{1}{d}}$ with
$\widehat {T} = X_{\cT}^{-1}(T)$ guarantees the strict reduction property \eqref{h-reduce}.

We also introduce the {\it oscillation} for any $V\in \V(\T)$ and \manel{$T \in \T$}
\begin{equation}\label{d:osc-def}
\begin{split}
& \gls{indi:osclocal}\osc_{\cT}(V,f,T)^2  :={} h_T^2 \Big\| 
      (\text{id} - \Pi^2_{2n-2})
      \left(fq + \widehat{\text{div}} 
       \big( q_\Gamma \widehat{\nabla} V \bgG_\Gamma^{-1} \big)\right)
       \Big\|_{L_2(\widehat{T})}^2\\
    & \qquad + h_T
        \Big\| (\text{id} - \Pi^2_{2n-1}) \left(q_\Gamma^+ \widehat{\nabla} V^+ (\bgG_{\Gamma}^+)^{-1}\widehat{\bn}^+
    + q_\Gamma^-\widehat{\nabla} V^- (\bgG_{\Gamma}^-)^{-1}\widehat{\bn}^- \right)
    \Big\|_{L_2(\partial\widehat{T})}^2,
\end{split}
\end{equation}
where $\widehat \bn^{\pm}$ is defined according to \eqref{eq:hat_n},
$\bg_\Gamma^\pm$ and $q_\Gamma^\pm=\sqrt{\det \bg_\Gamma^\pm}$
are the first fundamental form and
area element associated to $T^\pm$, and $\Pi^p_{m}$ denotes the best $L_p$-approximation operator onto the 
space $\Pn[m]$ of polynomials of degree $\leq m$; the domain is implicit from the context.
Notice that we used scaled local versions of the residual
$q(f+\Delta_\Gamma V)$ (see \eqref{lap-bel})  and co-normal derivatives
$r_\Gamma \nabla_\Gamma V \,\bn$ (see \eqref{e:normal_grad}) to
define the oscillation. We refer to Remark \ref{r:osc} for an
alternative definition of oscillation.

Finally, for any subset $\tau \subset \T$ we set
\begin{equation*}
\eta_\T(V,F_\Gamma,\tau)^2:= \sum_{T\in \tau} \eta_\T(V,F_\Gamma,T)^2, 
\qquad  %\text{and} \qquad
\osc_{\cT}(V,f, \tau)^2 
%:= \sum_{i=1}^M \sum_{T \in \tau \cap \T^i} \osc_{\T^i}(V,f,T)^2. 
:= \sum_{T \in \tau} \osc_{\T}(V,f, T)^2,
\end{equation*}
and simply write \gls{indi:eta}$\eta_\T(V,F_\Gamma)$ and \gls{indi:osc}$\osc_{\cT}(V,f)$ whenever $\tau = \T$.

Standard arguments \cite{AO00,Vr96} to derive upper and lower bounds for 
the energy error on flat domains can be extended to this case; see
\cite{DemlowDziuk:07,MMN:11,BCMN:Magenes}. 

\begin{lemma}[A posteriori upper and lower bounds]\label{L:upper-lower}
Let $\lambda_{\T_0}(\gamma)$ satisfy \eqref{eq:init_cond}.
Let $u\in H^1_\#(\gamma)$ be the solution of 
\eqref{p:Weak_PdeGm}, $(\T,\Gamma)$ be a pair of mesh-surface approximations and $U \in \V(\T)$ be the Galerkin solution of 
\eqref{FEM:weakform}.
Then there exist  constants \gls{const:upper_lower}$C_1, C_2$ and $\Lambda_1$ 
depending only on $\T_0$, the Lipschitz constant of $\gamma$, 
and $\|f\|_{L_2(\gamma)}$, such that
\begin{align}\label{upper}
\|\na_\ga(u-U)\|_{L_2(\gamma)}^2 
&\le C_1 \eta_\T(U,F_\Gamma)^2 + \Lambda_1\lambda_{\T}(\gamma)^2,
\\
\label{lower}
C_2 \eta_\T(U,F_\Gamma)^2 
& \le 
 \|\na_\ga(u-U)\|_{L_2(\gamma)}^2+ \osc_{\cT}(U,f)^2 
 + \Lambda_1\lambda_{\T}(\gamma)^2.
\end{align}
\end{lemma} 
\vskip-0.3cm
\begin{proof}
Our departing point is \eqref{eq:error_form} with $v\in H^1_\#(\gamma)$
arbitrary and $V\in\V(\T)$ being its Scott-Zhang interpolant
built over the partition $\overline{\T}$ of $\overline \Gamma$ and lifted to $\Gamma$
using $X_{\cT} \circ X_0^{-1}$.
Using interpolation estimates and \eqref{eq:equiv-norms} yields
\[
|I_1| \Cleq \eta_\T(U,F_\Gamma) \|\nabla_\gamma v\|_{L_2(\gamma)}.
\]
Since $\|\nabla_\Gamma U\|_{L_2(\gamma)} \Cleq
\|f\|_{L_2(\gamma)}$, the estimate
\eqref{Lp:Ak} on the error matrix $\bA_\Gamma$ gives
\[
|I_2| \Cleq \lambda_{\cT}(\gamma)\|\nabla_\gamma v\|_{L_2(\gamma)}.
\]
The upper bound \eqref{upper} follows from $I_3=0$. 
The lower bound \eqref{lower} can be proved locally over an element
$\widehat{T}\in\wT$ in $\Omega$ using standard arguments for flat domains.
\end{proof}

To prove optimality of AFEM we need a localized upper bound for the
distance between two discrete solutions $U$ and $U_*$. This bound measures 
$\|\nabla_{\gamma} (U_* - U)\|_{L_2(\gamma)}$ in terms of the 
PDE estimator restricted to the refined set and geometric estimator;
we refer to \cite[Lemma 4.13]{BCMN:Magenes} for a similar estimate for $n=1$.

\begin{lemma}[Localized upper bound]\label{L:local-upper}
Let $\lambda_{\T_0}(\gamma)$ satisfy \eqref{eq:init_cond}.
For $(\T,\Gamma)$, $(\T_*,\Gamma_*)$ pairs of mesh-surface approximations 
with $\T \leq \T_*$, let 
$\mathcal{R} := \mathcal{R}_{\T \to \T_*} \subset \T$ be the set of 
elements refined in $\T$ to obtain $\T_*$ i.e., $\mathcal{R} = \T \setminus \T_*$. 
Let $U\in \V(\T)$ and $U_*\in \V(\T_*)$ be the
corresponding discrete solutions of \eqref{FEM:weakform} on $\Gamma$ and
$\Gamma_*$, respectively. 
Then the following localized upper bound is valid
\begin{equation}\label{localized}
\|\na_{\gamma} (U_*-U) \|_{L_2(\gamma)}^2 \le C_1 \eta_\T(U,F_\Gamma,\cR)^2 +
\Lambda_1 \lambda_{\T}(\gamma)^2,
\end{equation}
with constants $C_1,\Lambda_1$ as in Lemma \ref{L:upper-lower}.
\end{lemma}
\begin{proof}
We start from the error representation formula \eqref{eq:error_form} 
by replacing $\gamma$ by $\Gamma_*$, $u$ by $U_*$, and taking as a test 
function $v=E_* := U_* - U\in \rhn{\V(\T_*)}$ %H^1_\#(\gamma)$
\begin{equation*}
\int_{\Gamma_*} \nabla_{\Gamma_*} (U_* - U) \,
	\nabla_{\Gamma_*}^T E_* = I_1 + I_2 +I_3.
\end{equation*}
To estimate $I_1$, we proceed as in the flat case 
\cite{CaKrNoSi:08,NoSiVe:09,Stevenson:07}. 
We first construct an approximation $V\in \V(\T)$ of $E_*\in \V(\T_*)$.
Let $\omega$ be the union of elements of $\mathcal{R} = \T \setminus \T_*$
and let $\overline \omega$ be the corresponding union in $\overline{\T}$.
Let $\omega_j$ (resp. $\overline{\omega_j}$),  $1\le j\le J$, denote the connected components of
the interior of $\omega$ (resp. $\overline \omega$).
We stress that $\omega_j$ may intersect several patches $\Gamma^i$ and
likewise $\overline{\omega_j}$ may intersect several copies of $\Omega$.
Let $\overline{\T}_j$ be the subset of elements in $\overline{\T}$
contained in $\overline{\omega}_j$ and let
${\V}(\overline{\T}_j)$  be the restriction of
${\V}(\overline \T)$ to $\overline{\omega}_j$.
We now construct
the Scott-Zhang operator  \cite{MR1011446} on 
$\overline{\omega}_j$ and use the map 
$X_{\T} \circ X_0^{-1}$ to lift it to $\Gamma$.
We denote this lift by $\pi_j:H^1(\overline{\omega}_j)\to \V(\T_j)$,
with 
$$
\T_j:= \Big\{T=X_{\T} \circ X_0^{-1}  (\overline T)
\ : \ \overline T \in \overline{\T}_j\big\} \subset \T.
$$
Let $V\in \V(\T)$ be the following approximation of the error 
$E_*\in \V(\T_*)$:
\begin{equation*}
V:= \pi_j E_* 	\quad \text{in} ~ \omega_j,
\qquad
V:= E_* \quad \text{elsewhere}.
\end{equation*}
By construction, $V$ has conforming boundary values on
$\partial \omega_j$, $V\in \V(\T)$,
and is an $H^1$-stable approximation to $E_*$. Since $V = E_*$ in 
$\Gamma \backslash \omega$, by the same standard argument for flat
domains, we obtain
\begin{equation*}\label{eq:loc-1}
|I_1| \leq C_1 \eta_\T(U,F_\Gamma,\cR)\|\nabla_\Gamma E_*\|_{L_2(\Gamma)}.
\end{equation*}
To estimate $I_2$, we note that $\Gamma$ and $\Gamma_*$ coincide in the unrefined region
$\Gamma \backslash \omega$, so that 
%$I_2|_{\Gamma \backslash \omega} = 0$.
%Adding and substracting 
%$\sum_{i=1}^M\int_{\widetilde{\omega}^i}\nabla_\gamma U\nabla_\gamma E_* $ 
%
\begin{equation*}
I_2 = \sum_{j=1}^J \int_{\widetilde{\omega}_j} \nabla_\gamma U \bA_\Gamma \nabla_\gamma^T E_*
-  \nabla_\gamma U \bA_{\Gamma_*} \nabla_\gamma^T E_*
\end{equation*}
with $\widetilde{\omega}_j := \chi \circ X_{\T}^{-1}(\omega_j)$.
Combining the estimate \eqref{Lp:Ak} on the error matrices
$\bA_\Gamma$ and $\bA_{\Gamma_*}$ with \eqref{eq:equiv-norms} and
\eqref{quasi-mono-n}, in its elementwise form, we obtain
\begin{equation*}\label{eq:loc-2}
|I_2| \Cleq \left(\lambda_{\cT}(\gamma) + \lambda_{\cT_*}(\gamma)
		   \right) \|\nabla_\Gamma E_*\|_{L_2(\gamma)}
	\Cleq (1 + \Lambda_0) \|f\|_{L_2(\gamma)}\lambda_{\T}(\gamma).
\end{equation*}
Since $I_3=0$ in view of the choice \eqref{def:F} of
$F_{\Gamma_*}$ and $F_{\Gamma}$, collecting the preceding estimates
we finally conclude \eqref{localized}.
\end{proof}

%--------------------------------------------------------------------------------
\subsection{Properties of the PDE Estimator and Oscillation}
\label{S:pde-estimator}
%--------------------------------------------------------------------------------
%
As indicated in \eqref{upper}--\eqref{lower}, we have access to the
energy error $\|\nabla_\gamma(u-U)\|_{L_2(\gamma)}$ only through the
PDE estimator $\eta_\T(U,F_\Gamma)$, the geometric estimator
$\lambda_{\T}(\gamma)$, and the oscillation quantity
$\osc_{\cT}(U,f)$.
As is customary for flat domains, the definition \eqref{d:osc-def} of
oscillation guarantees that
$\osc_{\cT}(U,f)$ is dominated by $\eta_\T(U,F_\Gamma)$, namely
\begin{equation}\label{dominance}
\osc_{\cT}(U,f,T)^2 \le C_3 \eta_{\T}(U,F_\Gamma,T)^2
\qquad\forall\, T\in \T,
\end{equation}
where the constant \gls{const:C3}$C_3$ depends on the surface $\gamma$.
\begin{remark}[Definition of oscillation] \label{r:osc}
\rm
The alternative definition to~\eqref{d:osc-def}:
%$\osc_{\cT}(V,f,\widehat T)$  
%
\begin{align*}
\osc_{\cT} &(V,f,T)^{\rhn{2}} = h_T^2 \Big\| 
      (\text{id} - \Pi^2_{2n-2})
      \Big(fqq_\Gamma^{-1/2} - q_\Gamma^{-1/2} \widehat{\text{div}} 
       \big( q_\Gamma \widehat{\nabla} V \bgG_\Gamma^{-1} \big)\Big)
       \Big\|_{L_2(\widehat{T})}^2\\
       & +  h_T
        \Big\| (\text{id} - \Pi^2_{2n-1})\Big( r_\Gamma^{-1/2}\big( q_\Gamma^+  \widehat{\nabla} V^+ (\bgG_{\Gamma}^+)^{-1}\widehat{\bn}^+
+ q_\Gamma^-\widehat{\nabla} V^- (\bgG_{\Gamma}^-)^{-1}\widehat{\bn}^-
				\big)\Big)
       \Big\|_{L_2(\widehat{\partial T})}^2
\end{align*}
would imply \eqref{dominance} with an optimal constant $C_3=1$.
However, this would be at the expense of a more intricate proof of
Proposition~\ref{P:oscU-bound} (local decay of oscillation).
We opted to use definition \eqref{d:osc-def} to 
simplify the presentation.
\end{remark}

The main novelty in \eqref{upper}--\eqref{localized} relative to flat
domains, which is also the chief challenge of
the present analysis, is the presence of $\lambda_{\T}(\gamma)$. 
In this respect, we show now the equivalence of $\eta_\T(U,F_\Gamma)$ 
and the \gls{indi:pde_error}{\it PDE error} 
\begin{equation}\label{total-error}
\E_\T (U,f) := \Big(\|\nabla_\gamma(u-U)\|_{L_2(\gamma)}^2
+\osc_{\cT}(U,f)^2 \Big)^{\frac12}
\end{equation}
provided $\lambda_{\cT}(\gamma)$ is small relative to
$\eta_\T(U,F_\Gamma)$. We refer to \cite{CaKrNoSi:08,NoSiVe:09} for a
similar result for flat domains, and to \cite{BCMN:Magenes} for
parametric surfaces and $n=1$.

\begin{lemma}[Equivalence of error and estimator]\label{L:equiv}
Let $C_1$, $C_2$, $\Lambda_1$ be given in Lemma \ref{L:upper-lower} \rhn{and $C_3$ be as in \eqref{dominance}}. If \looseness=-1
\begin{equation}\label{bound_lambda}
  \lambda_{\cT}(\gamma)^2 \le \frac{C_2}{2\Lambda_1} \eta_\T(U,F_\Gamma)^2,
\end{equation}
then there exist explicit constants \gls{const:C45}$C_4\ge C_5>0$, depending on $C_1$, $C_2$ \rhn{and $C_3$},
such that
\begin{equation}\label{equiv}
C_5\eta_\T(U,F_\Gamma) \le \E_\T(U,f) \le C_4\eta_\T(U,F_\Gamma).
\end{equation}
\end{lemma}
\vskip-0.3cm
\begin{proof}
Combining \eqref{upper} with \eqref{bound_lambda}, we infer that
\begin{equation}\label{simpler-upper}
\|\nabla_\gamma (u-U)\|_{L_2(\gamma)}^2 \le 
\Big(C_1 + \frac{C_2}{2}\Big) \eta_\T(U,F_\Gamma)^2.
\end{equation}
This, together with \eqref{dominance}, gives the upper bound in \eqref{equiv}.
We next resort to \eqref{lower} and \eqref{bound_lambda} to obtain
\begin{equation*}
C_2 \eta_\T(U,F_\Gamma)^2 \le \|\nabla_\gamma (u-U)\|_{L_2(\gamma)}^2 
+ \osc_{\cT}(U,f)^2 + \frac{C_2}{2} \eta_\T(U,F_\Gamma)^2,
\end{equation*}
which implies the lower bound in \eqref{equiv} and concludes the proof.
\end{proof}

It turns out that the usual reduction property of $\eta_\T(U,F_\Gamma)$
\cite[Corollary 3.4]{CaKrNoSi:08}, \cite{NoSiVe:09},
which is instrumental to prove a contraction property of $\AFEM$, is also
compromised by the presence of $\lambda_{\cT}(\gamma)$ as stated below. 
%The following result is proved in \cite[Lemma 4.2]{MMN:11} for any polynomial degree.

\begin{lemma}[Reduction of residual error estimator]\label{L:est-reduction}
Let $\lambda_{\T_0}(\gamma)$ satisfy \eqref{eq:init_cond}.
Given a mesh-surface pair $(\T,\Gamma)$, let $\cM\subset\T$ be the
subset of elements bisected at least $b\ge1$ times in refining $\T$ to
obtain $\T_*\ge\T$. If $\xi := 1 - 2^{-\frac{b}{d}}$, then 
there exist constants \gls{const:Lambda2}$\Lambda_2$ and \gls{const:Lambda3}$\Lambda_3$, solely depending on 
the shape regularity of $\mathbb T$, the Lipschitz constant $L$ of
$\gamma$, and $\|f\|_{L_2(\gamma)}$,
such that for any $\delta>0$
\begin{equation}\label{reduction}
\begin{aligned}
\eta_{\T_*}(U_*,F_{\Gamma^*})^2  \le{}& (1+\delta) 
\big(\eta_\T(U,F_\Gamma )^2 - \xi \eta_\T(U, \rhn{F_\Gamma},\cM)^2  \big)
\\
&+ (1+\delta^{-1}) \big(\Lambda_3 \|\nabla_\gamma (U_*-U)\|_{L_2(\gamma)}^2 
+ \Lambda_2 \lambda_{\cT}(\gamma)^2 \big).
\end{aligned}
\end{equation}
\end{lemma}
\begin{proof}
We first examine the residual $\mathcal{R}_T(U,F_\Gamma)$.
If $T_*\in \T_*$ and $T\in \T$ 
satisfy $\widehat T_* \subset \widehat T$, and
$T':= X_{\cT}\circ X_{\T_*}^{-1}(T_*) \subset T$,  then
the bound on $q_{\Gamma_*}$ given in Lemma~\ref{L:geom_estim} yields
\begin{equation*}
\begin{split}
\|\mathcal{R}_{T_*} &(U_*,F_{\Gamma_*}) \|_{L_2(T_*)} =
\| q_{\Gamma_*} ^{\frac{1}{2}} \mathcal{R}_{T_*}(U_*,F_{\Gamma_*}) \|_{L_2(\widehat T_*)}  \\
& \Cleq 
\left( \|F_{\Gamma_*} - F_\Gamma \|_{L_2(\widehat T_*)} 
 + \| \Delta_{\Gamma_*}(U_* - U) \|_{L_2(\widehat T_*)} 
 + \| (\Delta_{\Gamma_*} - \Delta_{\Gamma}) U \|_{L_2(\widehat T_*)}\right) \\
 &  \quad + \| (q_{\Gamma_*}^{1/2} - q_\Gamma^{1/2}) \mathcal{R}_{T}(U,F_\Gamma) \|_{L_2(\widehat T_*)} +\| q_\Gamma^{1/2} \mathcal{R}_{T}(U,F_\Gamma) \|_{L_2(\widehat T_*)}.
 \end{split}
\end{equation*} 
Now, from \eqref{g-G} and the local form of \eqref{quasi-mono-n} we bound the first term
\begin{equation*}
\|F_{\Gamma_*} - F_\Gamma \|_{L_2(\widehat T_*)} 
\leq \| \left( q_{\Gamma_*}^{-1} - {q_\Gamma}^{-1} \right) qf \|_{L_2(\widehat T_*)}
\Cleq \lambda_{\T}(\gamma,T') \| f \|_{L_2(T')}.
\end{equation*}
Recalling the expression \eqref{lap-bel} for the Laplace-Beltrami operator  and taking
$V = U_* - U$, we can write
\begin{align*}
\Delta_{\Gamma_*} V ={}& q_{\Gamma_*}^{-1} \widehat {\text{div}}\;  (
q_{\Gamma_*}  \widehat \nabla V \bgG^{-1}_{\Gamma_*})
					  \\={}& q_{\Gamma_*}^{-1} \left( 
					  \widehat \nabla q_{\Gamma_*}\cdot \widehat \nabla V \bgG^{-1}_{\Gamma_*}
					  +
					  q_{\Gamma_*} \widehat D^2 V :\bgG^{-1}_{\Gamma_*}
					  +
					  q_{\Gamma_*} \widehat \nabla V \cdot\widehat{\text{div}}\; \bgG^{-1}_{\Gamma_*}
					  \right),
\end{align*}
  and using bounds for $\|q_{\Gamma_*}\|_{L_\infty(\widehat{T}_*)}$
  and $\|\bgG_{\Gamma_*}^{-1}\|_{L_\infty(\widehat{T}_*)}$ from
  Lemma~\ref{L:geom_estim},  the inverse inequalities
  \eqref{est-inv-Q}, \eqref{est-inv-G}  and a third one for
  $\widehat{D}^2V$, we get
\begin{equation*}
\| \Delta_{\Gamma_*}(U_* - U) \|_{L_2(\widehat T_*)} \Cleq \frac{1}{h_{T_*}} 
\|\nabla_\gamma (U_*-U)\|_{L_2(T')}.
\end{equation*}
Again by virtue of \eqref{lap-bel} we rewrite the third term above
\begin{equation*}
\begin{split}
\| (\Delta_{\Gamma_*} &- \Delta_{\Gamma}) U \|_{L_2(\widehat T_*)}
\leq 
\| (q_{\Gamma_*}^{-1}  - q_\Gamma^{-1})  \widehat {\text{div}}\;
\big( q_{\Gamma_*} \widehat \nabla U \bgG^{-1}_{\Gamma_*} \big)\|_{L_2(\widehat T_*)} \\ 
& + \| q_\Gamma^{-1}   \widehat{\text{div}} \big (
\big(q_{\Gamma_*} - q_\Gamma) \widehat \nabla U \bgG_{\Gamma_*}^{-1} \big) \|_{L_2(\widehat T_*)}\\
& +\|  q_\Gamma^{-1} \widehat{\text{div}}\; \big( q_\Gamma \widehat
\nabla U (\bgG_{\Gamma_*}^{-1}  - \bgG_\Gamma^{-1}) \big) \|_{L_2(\widehat T_*)} 
\Cleq \frac{1}{h_{T_*}} \lambda_{\cT} (\gamma,T')
\|\nabla U\|_{L_2(T')}
\end{split}
\end{equation*}
due to an inverse inequality for $\widehat{D}^2U$, \eqref{est-inv-Q}, \eqref{est-inv-G}  and Lemma
\ref{L:geom_estim}. Finally, using the same arguments for the fourth
term we obtain
\begin{equation*}
\begin{split}
\| (q_{\Gamma_*}^{1/2} - q_\Gamma^{1/2}) \mathcal{R}_{T}(U,F_{\Gamma}) \|_{L_2(\widehat T_*)}
&= \| (q_{\Gamma_*} - q_{\Gamma}) ( q_{\Gamma_*}^{1/2} + q_{\Gamma}^{1/2})^{-1} \mathcal{R}_{T}(U,F_\Gamma) \|_{L_2(\widehat T_*)}
\\ & \Cleq \lambda_{\cT}(\gamma,T')
\| \mathcal{R}_{T}(U,F_{\Gamma}) \|_{L_2(\widehat T')}
\\
&\Cleq  \lambda_{\cT} (\gamma,T') \left( \frac{1}{h_{T_*}} \|\nabla U\|_{L_2(T')}
   + \| f \|_{L_2(T')} \right).
\end{split}
\end{equation*}
As a consequence, the interior residuals on $\Gamma_*$ and $\Gamma$
are related through the estimate
\begin{equation}
\begin{aligned}\label{one}
h_{T_*} \| & \mathcal{R}_{T}(U_* ,F_{\Gamma_*})\|_{L_2(T')} \le
h_{T_*} \| \mathcal{R}_{T}(U,F_\Gamma) \|_{L_2(T')} \\
& + C \| \nabla_\gamma (U_* - U) \|_{L_2(T')}   +
C \lambda_{\cT}(\gamma,T) 
\left( \|\nabla U\|_{L_2(T')}
+ h_{T_*} \| f \|_{L_2(T')}
\right),
\end{aligned}
\end{equation}
for some constant $C$ only depending on the shape regularity of
$\mathbb T$ and the Lipschitz constant $L$ of $\gamma$.

We now examine the jump
residual $\mathcal{J}_{\partial T}(U)$.
 Let $S_*\in \mathcal{S}_{\Gamma*}$ and  
$S':= X_{\cT}\circ X_{\T_*}^{-1}(S_*) \subset \Gamma$.
We denote by $T_*^\pm$ the two elements of $\T_*$ sharing $S_*$ 
\manel{(resp. $[T^\pm]' : =X_{\cT}\circ X_{\T_*}^{-1}(T_*^\pm) $) }
and
recall that the corresponding outward pointing co-normals
$\bn_{S_*}^{\pm}$ are not necessarily co-linear;
moreover, $T_*^\pm$ may belong to different surface patches,
i.e. $T_*^+ \in \T_*^i$ and $T_*^- \in \T_*^j$ for some $1\leq i,j \leq M$.
Still, observe that the  jump $\mathcal{J}_{S_*}(U_*)$ can be rewritten as follows
\begin{equation*}
\begin{split}
\mathcal{J}_{S_*}(U_*)  = \mathcal{J}_{S}(U) |_{S_*}  & +
\left(
\nabla_{\Gamma_*} U_*^+|_{S_*} \bn_{S_*}^+ 
		- \nabla_{\Gamma} U ^+|_{S_*}  \bn_S^+|_{S_*}
		\right)
\\
&+ \left(
\nabla_{\Gamma_*} U_*^-|_{S_*} \bn_{S_*}^-  
		- \nabla_{\Gamma} U ^-|_{S_*} \bn_S^-|_{S_*}
		\right),
\end{split}
\end{equation*}
regardless of $\Gamma^i$ and $\Gamma^j$. Therefore,
the last two terms in the right hand side can now be estimated
using the geometric error estimates \eqref{g-G}.
Note that on $S_*$
\begin{equation*}
\begin{split}
\nabla_{\Gamma_*} U_*^\pm \, \bn_{S_*}^\pm
		&- \nabla_{\Gamma} U^\pm \, \bn_S^\pm
		= \; \nabla_{\Gamma_*} (U_*^\pm-U^\pm) \, \bn^\pm_{S_*}			
               \\&+ \left(\nabla_{\Gamma_*}-\nabla_{\Gamma}\right) U ^\pm 
		\, \bn_{S_*}^\pm 
		+
		\nabla_{\Gamma} U^\pm \left( 
		\bn_{S_*}^\pm - \bn_{S}^\pm
		\right) = I + II + III.
\end{split}
\end{equation*}
We bound each term using their parametric representation on
$\widehat S_*:= X_{\T_*}^{-1}(S_*)$.
For the first term, we use the expression \eqref{e:normal_grad}
of the tangential derivative in the co-normal direction,
the spectral bounds on $\bg_{\Gamma_*}$ and $q_{\Gamma_*}$ given in
Lemma \ref{L:geom_estim}, and a scaled trace estimate to deduce
$$
\| \nabla_{\Gamma_*} (U_*^\pm-U^\pm) \bn^\pm_{S_*}\|_{L_2(S_*)}\Cleq |\widehat T_*^\pm|^{-\frac 1 {2d}} \| \widehat \nabla  (\widehat U_*^\pm-\widehat U^\pm) \|_{L_2(\widehat T_*)}.
$$
Recalling that \manel{$h_{T_*^\pm}^d =  |\widehat
T_*^\pm|$}, we see that
$$
\| I \|_{L_2(S_*)}\Cleq h_{T_*^\pm}^{-1/2} \|  \nabla_{\Gamma_*} (U_*^\pm-U^\pm)\|_{L_2(T_*^\pm)}.
$$
Similarly, in view of \eqref{grad-Gamma} and \eqref{e:normal_grad},
we obtain
\begin{equation*}
\begin{split}
\| II \|_{L_2(S_*)}
  +\| III \|_{L_2(S_*)}
  & \Cleq h_{T_*^\pm}^{-1/2} \|\nabla_\Gamma U\|_{L_2(\widehat T_*^\pm)}
  \Big(\| \bD_{\Gamma_*}-\bD_{\Gamma})\|_{L_\infty(\widehat T_*^\pm)}
\\ &+
                        \|q_{\Gamma_*}-q_\Gamma\|_{L_\infty(\widehat{T}_*^\pm)}
                        + \|r_{\Gamma_*}-r_\Gamma\|_{L_\infty(\partial\widehat{T}_*^\pm)}
                        \Big),
\end{split}
\end{equation*}
where $r_{\Gamma_*}$ and $r_\Gamma$ denote the area elements associated with
$S_*$ and $S' := X_{\cT} \circ X_{\T_*}^{-1}(S_*)$ respectively.
Utilizing the geometry error estimate~\eqref{g-G}, we further get
\[
\| II \|_{L_2(S_*)} +\| III \|_{L_2(S_*)}
\Cleq  h_{T_*^+}^{-1/2} \manel{\lambda_{\cT}(\gamma, [T^\pm]') }\|\nabla_\Gamma U\|_{L_2(\widehat T_*^\pm)}.	
\]
Hence, combining the previous two estimates,  we get
\begin{multline*}
\| \nabla_{\Gamma_*} U_*^\pm \, \bn_{S_*}^\pm
		- \nabla_{\Gamma} U ^\pm \, \bn_S^\pm
\|_{L_2(S_*)} \\
\Cleq h^{-1/2}_{T_*^\pm} \left (\| \nabla_\gamma (U_*^\pm - U_*^\pm) \|_{L_2(T_*^\pm)} 
+ \manel{\lambda_{\cT}(\gamma, [T^\pm]') }\|\nabla_\Gamma U\|_{L_2(\widehat T_*^\pm)} \right),
\end{multline*}
whence
\begin{equation*}
\begin{split}
\|\mathcal{J}_{S}(U) \| _{L_2(S_*)}
= \| r_{\Gamma_*}^{1/2} \mathcal{J}_{S} (U) \|_{L_2(\widehat S_*)} 
 \leq  \| (r^{1/2}_{\Gamma_*} - r^{1/2}_\Gamma)  \mathcal{J}_{S}(U) \|_{L_2(\widehat S_*)}
+ \| \mathcal{J}_{S}(U) \|_{L_2(S')}.
\end{split}
\end{equation*}
Invoking again~\eqref{g-G} %Lemma~\ref{L:geom_estim},  
we realize that
$
\| r_{\Gamma_*}- r_\Gamma\|_{L_\infty(\widehat S)} \leq
 \manel{\lambda_{\cT}(\gamma, [T^\pm]')}.
$
Combining this with a scaled trace theorem, we deduce that
\begin{equation*}
\| (r^{1/2}_{\Gamma_*} - r_{\Gamma}^{1/2})  \mathcal{J}_{S}(U) \|_{L_2(\widehat S_*)}
\Cleq {h^{-1/2}_{T_*^\pm}} \manel{\lambda_{\cT}(\gamma, [T^\pm]')} \|\nabla_\Gamma U\|_{L_2(T_*^\pm)}
\end{equation*}
whence
\begin{equation*}
\|\mathcal{J}_{S}(U) \| _{L_2(S_*)}
 \leq   \| \mathcal{J}_{S}(U) \|_{L_2(S')} +
 C {h^{-1/2}_{T_*^\pm}} \manel{\lambda_{\cT}(\gamma, [T^\pm]') } \|\nabla_\Gamma U\|_{L_2(T_*^\pm)}.
\end{equation*}
The above three estimates guarantee the existence of a constant $C$
only depending on the shape regularity of $\mathbb T$ and the Lipschitz constant $L$ of $\gamma$ such that
\begin{equation}\label{two}
\begin{split}
h_{T_*^\pm}^{1/2}
&
\|\mathcal{J}_{S_*}(U_*) \| _{L_2(S_*)}
 \leq h^{1/2}_{T_*^\pm} \| \mathcal{J}_{S'}(U) \|_{L_2(S')}
\\
&+ C \left( \| \nabla_\gamma (U_* - U) \|_{L_2(T^\pm)}
+ \manel{\lambda_{\cT}(\gamma, [T^\pm]')} \|\nabla_\Gamma U\|_{L_2(T_*^\pm)}
 	\right).	
\end{split}
\end{equation}
To conclude the proof we proceed as for graphs \cite[Lemma 4.2]{MMN:11},
basically squaring \eqref{one} and \eqref{two} via
Young's inequality, adding over
all elements $T_*\in\T_*$ and sides $S_*\in\mathcal{S}_*$, and using
the strict reduction
\eqref{h-reduce} of meshsize $h_T$ for all refined elements.
In addition, we employ the global
bound $\| \nabla_\Gamma U \|_{L_2(\Gamma)} \Cleq \| f \|_{L_2(\gamma)}$.
\end{proof}

Another difference with the theory of adaptivity for flat domains is
the behavior of data oscillation under refinement. The usual situation
is that $\osc_{\cT}(U,f)$ does not increase upon refinement from
$\T$ to $\T_*$ \cite{MNS:00}. This is no longer true because
$\osc_{\cT}(U,f)$ and $\osc_{\cT_*}(U_*,f)$ correspond to different
domains $\Gamma$ and $\Gamma_*$. We state a quasi-monotonicity
property in Lemma \ref{L:osc-reduction} but omit its proof because
it is similar and somewhat simpler than that of Lemma~\ref{L:est-reduction}.

\begin{lemma}[Quasi-monotonicity of data oscillation]\label{L:osc-reduction}
Let $\lambda_{\T_0}(\gamma)$ satisfy \eqref{eq:init_cond}. 
Let $(\T,\Gamma)$, $(\T_*,\Gamma_*)$ be mesh-surface pairs
with $\T \leq \T_*$. 
Then, there exist constant \gls{const:C6}$C_6$, \gls{const:Lambda2}$\Lambda_2$ and \gls{const:Lambda3}$\Lambda_3$  
depending only on \manel{$\T_0$}, the Lipschitz constant $L$ of $\gamma$, and 
$\|f\|_{L_2(\gamma)}$,  such that
\begin{equation}\label{osc-mono}
\osc_{\T_*}(V_*,f)^2 \le C_6 \osc_{\cT}(V,f)^2 
			       +\Lambda_3 \|\nabla_\gamma(V_* - V)\|_{L_2(\gamma)}^2  
			       +\Lambda_2 \lambda_{\cT}(\gamma)^2.
\end{equation}
%Moreover,
%\begin{equation}\label{osc-pert}
%\osc_{\Gamma_*}(V_*,\cT_* \cap \cT)^2 \le 2 \osc_{\Gamma}(V,\cT)^2 
%			       +2 \Lambda_3 \|\nabla_\gamma(V_* - V)\|_{L_2(\gamma)}^2  
%			       + 2 \Lambda_2\lambda_\Gamma(\T)^2
%\end{equation}
%
\end{lemma}

\begin{remark}[Local perturbation of data oscillation] 
\rm
The previous result is also valid locally, that is for any subset $ \tau \subset \T_*$. In fact, if $\tau = \T \cap \T_*$
the same proof gives \eqref{osc-mono} with $C_6 = 2$,
\begin{equation}\label{osc-pert}
\rhn{\osc_{\T_*}(V,f, \tau)^2 \le 2 \osc_{\cT}(W,f,	\tau)^2 
			       + \Lambda_3 \|\nabla_\gamma(V - W)\|_{L_2(\gamma)}^2
			       +  \Lambda_2\lambda_{\cT}(\gamma)^2,}
\end{equation}
for any  piecewise polynomials $V,W$  subordinate to $\tau$. 
Although the elements in $\tau$ describe (part of) the
common surface $\Gamma \cap \Gamma_*$, whence there is no geometric
discrepancy, the presence of the geometric
estimator $\lambda_{\cT}(\gamma)$ in \eqref{osc-pert} 
is due to the boundary of this common region. Note that  the 
contribution to the oscillation associated to a side on the boundary of $\tau$
involves  the terms $q_\Gamma^\pm$ according to \eqref{d:osc-def}.
\end{remark}

%%%%%%%%%%%%%%%%%%%%%%%%%%%%%%%%%%%%%%%%%%%%%%%%%%%%%%%%%%%%%%%%%%%%%%%%%%%%%%%%%
\section{AFEM: Design and Properties}\label{S:AFEM}
%%%%%%%%%%%%%%%%%%%%%%%%%%%%%%%%%%%%%%%%%%%%%%%%%%%%%%%%%%%%%%%%%%%%%%%%%%%%%%%%%
%\modif{AB: I AM HERE}

Since $\lambda_{\cT}(\gamma)$ and $\eta_\T(U,F_\Gamma)$ account for quite
different effects, following \cite{Bonito-DeVore-Nochetto:13}, the algorithm AFEM is designed to handle
them separately via the modules \ADAPTSURF and \ADAPTPDE.

\medskip
\gls{algo:AFEM}{\bf AFEM:}
\rhn{Given $\T_0$ and parameters $\varepsilon_0>0$, $0<\rho<1$, and
$\omega>0$, set $k=0$.}
\begin{algotab}
  \> \>1. $\T_k^+ = \ADAPTSURF (\T_k,\omega \eps_k)$\\
  \> \>2. $\rhn{[U_{k+1},\T_{k+1}]} = \ADAPTPDE (\T_k^+,\eps_k)$ \\
  \> \>3. $\eps_{k+1} = \rho \eps_k$; $k = k+1$ \\  
  \> \>4. go to 1.
\end{algotab}
\noindent
We notice the presence of the factor \gls{const:w}$\omega$ , which is employed to
make the geometric error small relative to the current tolerance $\eps_k$, thereby controlling the interactions between the geometry and the PDE.
This turns out to be essential for both contraction and optimality of
AFEM, even for polynomial degree $n = 1$ as discussed in
\cite{BCMN:Magenes}.

%--------------------------------------------------------------------------------
\subsection{Module  \ADAPTSURF}\label{S:adapt-surface}
%--------------------------------------------------------------------------------
%
Given a tolerance $\eps>0$ and an admissible subdivision $\cT$,
$\rhn{\T_*}=\ADAPTSURF(\cT, \eps)$ improves the surface resolution
until the new subdivision $\rhn{\cT_*} \ge \T$ satisfies
\begin{equation}\label{e:target_surf}
\lambda_{\T_*}(\gamma)\leq \eps,
\end{equation}
where $\lambda_{\cT}(\gamma)$ is the \emph{geometric estimator} introduced in \eqref{p:geom_osc}. 
This module is based on a {\it greedy} algorithm and acts on a generic
mesh $\T=\cup_{i=1}^M\T^i \in\grids$:
\smallskip
\begin{algotab}\gls{algo:SURFACE}
\> $\rhn{\cT_*} = \ADAPTSURF (\cT,\eps)$ \\
\> \> 1. if  $\cM :=\{T\in \cT \, : \,\lambda_{\cT}(\gamma,T)>\eps \} =\emptyset$ \\
\>  \>  \> \> return($\cT$) and exit \\
\> \> 2. $\cT = \REFINE(\cT,\cM)$\\
\> \> 3. go to 1.
\end{algotab}
%
%\rhn{
%\smallskip
%\begin{algotab}
%  \> $\cT^+ = \ADAPTSURF (\cT,\tau)$ \\
%\>  \> while $\cM :=\{T\in \cT \, : \,\lambda_{\cT}(\widehat T)>\eps \}\ne\emptyset$\\
%\>  \> \> $\cT = \REFINE(\cT,\cM)$\\
%\>  \> end while\\
%\> return $\c$
%\end{algotab}
%}
%

\smallskip\noindent
where $\REFINE(\cT,\cM)$ refines all elements in the marked set
$\cM$ and keeps conformity; more details are given in \S \ref{S:adapt-pde}.
To derive convergence rates for AFEM, we require that \ADAPTSURF is 
{\it $t$-optimal}, i.e. there exists a constant
\rhn{$C(\gamma)$} such that the set $\M$ of all the elements marked for refinement 
in a call to $\ADAPTSURF(\T,\manel{\eps})$ satisfies
\begin{equation}\label{complex-adaptsurf}
\#\M \le C(\gamma) \; \eps^{-1/t}.
\end{equation}
In \S~\ref{S:gamma} we show that
this assumption is satisfied by a greedy algorithm provided that 
$\chi^i \in B_q^{1+td}(L_q(\Omega))$  with  $tq>1$, $0<q\leq
\infty$ and $td \leq n$ for all $1\le i\le M$.

%--------------------------------------------------------------------------------
\subsection{Module \ADAPTPDE}\label{S:adapt-pde}
%--------------------------------------------------------------------------------
%
Given a tolerance $\varepsilon>0$ and an admissible subdivision 
\rhn{$\cT \in\grids$},
\rhn{$[U_*,\cT_*]=\ADAPTPDE(\cT, \varepsilon)$} outputs a refinement 
\rhn{$\cT_*\geq \cT$} 
and the associated finite element solution \rhn{$U_*\in \mathbb V(\T_*)$}
such that
\begin{equation}\label{afem-enter}
\rhn{
\eta_{\T_*}(U_*,F_{\Gamma_*}) \leq \varepsilon.
}
\end{equation}
This module is based on the standard adaptive sequence:
\smallskip
\begin{algotab}\gls{algo:PDE}
\> $\rhn{[U_*,\cT_*]} = \ADAPTPDE (\cT,\varepsilon)$ \\
\> \> 1. $U = \SOLVE(\cT)$ \\
\> \> 2. $\{ \eta_\T(U,F_\Gamma,T) \}_{T \in \T} = \ESTIMATE(\cT, U)$ \\
\> \> 3. if $\eta_\T(U, F_\Gamma) < \varepsilon$\\
\> \> \> \> return($\cT$, $U$) and exit \\
\> \> 4. $\cM = \MARK(\T,\{ \eta_\T(U,F_\Gamma,T) \}_{T \in \T})$\\
\> \> 5. $\cT = \REFINE(\cT,\cM)$\\
\> \> 6. go to 1
\end{algotab}

\smallskip\noindent
We describe below the modules $\SOLVE, \ESTIMATE, \MARK$ and $\REFINE$ separately.

\medskip\noindent
%--------------------------------------------------------------------------------
{\bf Procedure \SOLVE.}
%--------------------------------------------------------------------------------
%
This procedure solves the SPD linear system resulting for 
\eqref{FEM:weakform} where we recall that $\Gamma = \Gamma_\T$. 
For simplicity we assume that \eqref{FEM:weakform} is solved exactly
with linear complexity.
We refer to \cite{KY:08} for a hierachical basis multigrid preconditioner 
and to \cite{BP:11} for standard variational and non-variational multigrid algorithms.

\medskip\noindent
%--------------------------------------------------------------------------------
{\bf Procedure \ESTIMATE.}
%--------------------------------------------------------------------------------
Given the Galerkin solution  $U \in \V(\T)$ of \eqref{FEM:weakform} 
the procedure 
$\ESTIMATE$ computes the PDE error indicators $\{\eta_\T(U,F_\Gamma,T)\}_{T\in\T}$. We emphasize that this procedure does not compute the oscillation terms, which are only needed to carry out the analysis.

Lemma \ref{L:equiv} (equivalence of error and estimator) is critical to deduce 
that the \ADAPTPDE module, which reduces the error indicators $\eta_\T(U,F_\Gamma)$, is successful in reducing the PDE error
$\E_\T (U,f)$ of \eqref{total-error}
provided the parameter $\omega$ satisfies
\begin{equation}\label{bound_omega}
\gls{const:w1}  \omega \leq \omega_1 \definedas \sqrt{\frac{C_2}{2 \Lambda_0^2 \Lambda_1}}.
\end{equation}
In fact, given a tolerance $\varepsilon>0$ to be reached by $\ADAPTPDE$ starting from the input subdivision \rhn{$\T$} satisfying \rhn{$\lambda_{\cT}(\gamma)\leq \omega \varepsilon$},
we  observe that \eqref{quasi-mono-n} guarantees that \rhn{$\T$} as well
as all subdivisions \rhn{$\T_*\ge \T$} constructed within the inner iterates of $\ADAPTPDE$ satisfy
\looseness=-1
\[
\lambda_{\T_*}(\gamma)^2 \le \Lambda_0^2 \lambda_{\T}(\gamma)^2 \le 
\frac{C_2}{2\Lambda_1} \eps^{2}.
\]
Within the while loop of $\ADAPTPDE$ we have $\eta_\T(U,F_\Gamma)>\eps$, so we deduce the validity of
\eqref{bound_lambda} whence that of \eqref{equiv} within that loop.

\medskip\noindent
%--------------------------------------------------------------------------------
{\bf Procedure \MARK.}
%--------------------------------------------------------------------------------
We rely on an optimal {\it D\"orfler's} marking strategy for the selection of elements.
Given the set 
of indicators $\{\eta_\T(U,F_\Gamma,T)\}_{T\in \T}$ and a marking parameter \gls{const:theta}$\theta \in
(0,1]$, $\MARK$ outputs a subset of marked elements $\cM \subset \T$
such that%
\begin{equation}\label{doerfler}
\eta_\T(U, F_\Gamma, \cM) \ge \theta \eta_\T(U, F_\Gamma).
\end{equation}
In contrast to \cite{MMN:11}, $\MARK$ only employs the error 
indicators and does not use either the oscillation or surface indicators.
We will see that quasi-optimal cardinality requires that $\cM$ is
\emph{minimal} and quasi-optimal workload that the sorting scales linearly.

\medskip\noindent
%--------------------------------------------------------------------------------
{\bf Procedure \REFINE.}
%--------------------------------------------------------------------------------
%
Given a subdivision $\cT$ and a set $\cM \subset \T$ of marked elements, 
the call $\REFINE(\cT,\cM)$ bisects 
all elements in $\cM$ at least 
$b\geq 1$ times and \rhn{performs} additional refinements necessary to maintain conformity.
The resulting subdivision is denoted by $\T_*$. 
Recall that the bisection procedure is first executed on faces of the
corresponding flat subdivision $\overline{\T}$ and its effect is
transferred to the actual subdivision via interpolation maps
$X_{\T^i}^i\circ (X_0^i)^{-1}$ for $i=1,...,M$.

Since the refinement procedure is performed on $\overline{\T}$ or similarly on $\widehat{\T}$, the complexity results of the overall refinement algorithm proved by Binev, Dahmen, and DeVore for $d=2$ \cite{BiDaDeV:04} and Stevenson \cite{Stevenson:08} 
for $d>2$ hold in our setting. 
\rhn{In order to state them precisely, following~\cite{BiDaDeV:04,Stevenson:08,NoSiVe:09}, we need the concept of \gls{def:label}admissible labeling \ab{\cite{BiDaDeV:04,Stevenson:08}}.
\ab{
\begin{remark}[Admissible labeling]\label{r:labeling}
\rm 
For $d=1$, any subdivision is said to have an admissible labeling.
For $d=2$, we say that $\T_0$ has an admissible labeling if each edge of $\T_0$ is labeled either 0 or 1 such that each element of $\T_0$ has exactly two edges with label 1 and one with label 0~\cite{BiDaDeV:04}; refining an element entails connecting the middle of the edge labeled 0 with the opposite angle.
For $d>2$, the corresponding condition (b) of \S4 in~\cite{Stevenson:08} is much more technical and is omitted here. In short, an admissible initial labeling guarantees that the bisection procedure terminates in finite steps with a conforming mesh, and that any uniform refinement of $\T_0$ is conforming.
\end{remark}
}
}
\smallskip
\begin{lemma}[Complexity of \REFINE]\label{L:REFINE:complexity}
Assume that the initial triangulation $\T_0$ has an admissible labeling.
Let $\{\T_k\}_{k\geq 0}$ be a
sequence of triangulations produced by successive calls to $\Tk[k+1] = \REFINE(\Tk,\cM_k)$, where $\cM_k$ is any subset of $\T_k$, $k\ge 0$.
Then, there exists a constant \gls{const:C7}$C_7$ 
solely depending on $\T_0$, its labeling,
and the refinement depth~$b$ such that
\begin{equation}\label{REFINE:complexity}
\# \T_k - \# \T_0 \leq 
	C_7 \sum_{j=0}^{k-1} \#\cM_j, \qquad \forall k\geq 1.
\end{equation} 
\end{lemma}

\vskip-0.2cm
It is worth noticing that the user parameter $b\ge 1$ can be chosen equal to one, which only implies a minimal refinement, and does not force an interior node property \cite{MNS:00,MNS:02}
 or an extra refinement to improve the surface approximation \cite{MMN:11}.

\begin{remark}[Alternative subdivision strategies] \label{r:Alt_sub_stategies}
\rm For simplicity we only discuss the refinement strategy based on 
simplex bisection. However, all the results obtained here can be extended
to quadrilaterals with fixed number of hanging nodes or red refinements. 
We refer to \cite[Section 6]{BN:10} for details.
\end{remark}

%%%%%%%%%%%%%%%%%%%%%%%%%%%%%%%%%%%%%%%%%%%%%%%%%%%%%%%%%%%%%%%%%%%%%%%%%%%%%%%%%
\section{Conditional Contraction Property}\label{S:contraction}
%%%%%%%%%%%%%%%%%%%%%%%%%%%%%%%%%%%%%%%%%%%%%%%%%%%%%%%%%%%%%%%%%%%%%%%%%%%%%%%%%

The procedure $\ADAPTPDE$ is known to yield a contraction property
in the {\it flat} case. 
In the present context, however, the surface approximation is
responsible for 
lack of consistency in that the sequence of finite element spaces
is no  longer nested.
This in turn leads to failure of a key orthogonality property
between discrete solutions, 
the Pythagoras property. We have, instead, a perturbation result 
referred to as \emph{quasi-orthogonality} below. Its proof follows the
steps of that for graphs \cite[Lemma 4.4]{MMN:11}.
In this section, we use the notation
\begin{gather*}
e^j :=\| \nabla_\gamma (u-U^j)\|_{L_2(\gamma)}, \quad
E^j:= \| \nabla_\gamma (U^{j+1}-U^j)\|_{L_2(\gamma)}, \\
\eta^j := \eta_{\T^j}(U^j,F^j), \quad
\eta^j(\cM^j) := \eta_{\T^j}(U^j,F^j,\cM^j), \quad
\lambda^j:=\lambda_{\T^j}(\gamma),
\end{gather*}
where $\cT^j$ are meshes obtained after each inner iteration of \ADAPTPDE,
starting with $\cT^0=\rhn{\cT}$, $\cM^j\subset \cT^j$ are the subsets of elements selected by the marking procedure, 
$F^j$ are the scaled right hand sides defined in \eqref{def:F} with $\Gamma$ replaced by $\Gamma^j$, the surface associated to $\cT^j$,
 and $U^j \in \mathbb V(\cT^j)$ are the corresponding Galerkin solutions. 

\begin{lemma} [Quasi-orthogonality]  \label{l:quasi}
There exists a constant \gls{const:Lambda2}$\Lambda_2>0$ solely depending on the
Lipschitz constant $L$ of $\gamma$ and $\|f\|_{L_2(\gamma)}$ such that for $i=j,j+1$ with $j \ge 0$, we have
\begin{equation}
\label{quasi-ortho}
(e^j)^2 - \frac32 (E^j)^2 - \Lambda_2 (\lambda^i)^2
\le
(e^{j+1})^2 
\le 
(e^j)^2 - \frac12 (E^j)^2 + \Lambda_2 (\lambda^i)^2.
\end{equation}
\end{lemma}
Before proceeding with the proof of the above lemma, we point out that the constant $\Lambda_2$ was already defined in Lemma~\ref{L:est-reduction}.
This is to simplify the notations below and is without loss of generality (upon redefining $\Lambda_2$ as the maximum of the two constants).
\begin{proof} %[of Lemma~\ref{l:quasi}]
Since the symmetry of the Dirichlet form implies
\[
(e^j)^2 = (e^{j+1})^2 + (E^j)^2 + 
2 \int_\gamma\nabla_\gamma(u-U^{j+1})\nabla_\gamma^T(U^{j+1}-U^j),
\]
we just have to examine the last term.
\rhn{Utilizing the error representation~\eqref{eq:error_form} with \manel{$v = U^{j+1}-U^{j}$} and realizing that $I_1=I_3=0$, we readily obtain}
% Combining \eqref{p:Weak_PdeGm},
%\eqref{FEM:weakform}, and \eqref{def:F} with the consistency error
%representation \eqref{eq:cons_error_rep} gives
%
\[
\int_\gamma \nabla_\gamma(u-U^{j+1})\nabla_\gamma^T (U^{j+1}-U^j)
= \int_\gamma \nabla_\gamma U^{j+1} \bA_{\Gamma^{j+1}} \nabla_\gamma^T(U^{j+1}-U^j).
\]
Invoking \eqref{Lp:Ak} yields
\[
\Big|\int_\gamma\nabla_\gamma(u-U^{j+1})\nabla_\gamma^T(U^{j+1}-U^j) \Big|
\Cleq \|f\|_{L_2(\gamma)} 
\lambda^{j+1}  E^j.
\]
This leads to \eqref{quasi-ortho} after applying Young's inequality
and using \eqref{quasi-mono-n}.
\end{proof}

Notice that relation \eqref{quasi-ortho}  also holds when (i)
$\T^{j}$, $\T^{j+1}$ are replaced with $\cT$, $\cT^*$ satisfying $\T^*
\geq \T$; (ii) $U^{j+1}$ is replaced by $U^* \in \mathbb V(\cT^*)$
and (iii) $U^j$ is replaced by any $V\in \mathbb V(\cT)$
because $V$ need not be the Galerkin solution over $\T$.
The parameter
\begin{equation}
\gls{const:w2}\omega_2 \definedas 
\frac{\xi\theta^2}{\Lambda_0 \sqrt{32\Lambda_2(2 \Lambda_3 + 1)}},   \label{bound_omega_2}
\end{equation}
where $\xi:=1-2^{-b/d}$ is defined in Lemma \ref{L:est-reduction}, is used subsequently as a threshold for the  \AFEM parameter $\omega$.
\begin{theorem}[Conditional contraction property]\label{T:conditional}
Let $\theta \in (0,1]$ be the marking parameter of \MARK and let 
$\{\T^j, U^j\}_{j = 0}^J$ be a
sequence of meshes and discrete solutions
produced by one call to procedure $\ADAPTPDE \, (\T^0, \varepsilon)$ inside AFEM, \ie $\lambda^0:=\lambda_{\T^0}(\gamma) \le \omega \varepsilon$. Assume that the AFEM parameter $\omega$ satisfies
\begin{equation*}
 \omega \leq \min \left\{ \omega_1, \omega_2\right\}, 
\end{equation*}
where $\omega_1$ and $\omega_2$ are given in \eqref{bound_omega} and \eqref{bound_omega_2}, respectively.
Then there exist constants $0<\alpha<1$ and $\beta>0$ such that
\begin{equation}\label{contraction}
(e^{j+1})^2 + \beta
(\eta^{j+1})^2 \le \alpha^2 \Big(
(e^j)^2 + \beta (\eta^j)^2\Big) \qquad\forall\, 0\le j<J.
\end{equation}
Moreover, the number of inner iterates $J$ of   $\ADAPTPDE$ is uniformly bounded.
\end{theorem}
\begin{proof}
We proceed in four steps. Note that $\eta^j\ge\eps$ for $0\le j <J$.

\step{1} Let $\beta>0$ be a scaling parameter to be found later. 
We combine \eqref{quasi-ortho} and \eqref{reduction} to write
\begin{equation*}
\begin{aligned}
(e^{j+1})^2 + \beta (\eta^{j+1})^2 & \le
(e^j)^2 + 
\Big(-\frac12 + \beta (1+\delta^{-1}) \Lambda_3 \Big) (E^j)^2
\\
& + \Lambda_2 \Big(1+ \beta (1+\delta^{-1}) \Big) (\lambda^j)^2
+  \beta (1+\delta) \Big((\eta^j)^2 - \xi \eta^j(\cM^j)^2 \Big),
\end{aligned}
\end{equation*}
where $\mathcal M^j$ is the set of elements  marked for refinement at
the $j$-th subiteration.
To remove the factor of $E^j$ we now
choose $\beta$ dependent on  $\delta$, to be
\begin{equation}\label{eq:beta}
\beta(1+\delta^{-1}) \Lambda_3 = \frac12 \quad\Rightarrow\quad
\beta (1+\delta) = \frac{\delta}{2\Lambda_3},
\end{equation}
and thereby obtain
\begin{equation*}
(e^{j+1})^2 + \beta (\eta^{j+1})^2
\le (e^j)^2  
%+ \Lambda_2 \Big(1+ \beta (1+\delta^{-1}) \Big) (\lambda^j)^2
+ \Lambda_2 \Big(1+ \frac{1}{2\Lambda_3} \Big) (\lambda^j)^2
+  \beta (1+\delta) \Big((\eta^j)^2 - \xi \eta^j(\cM^j)^2 \Big) .
\end{equation*}

\step{2} Invoking D\"orfler marking
\eqref{doerfler}, we deduce
\[
(\eta^j)^2 - \xi \eta^j(\cM^j)^2 \le (1-\xi\theta^2) (\eta^j)^2.
\]
Since the initial mesh $\T^0$ comes from $\ADAPTSURF$ we know that
$\lambda^0 \le \omega\eps\le\omega\eta^j$ for $1 \le j \rhn{<} J$.
Using \eqref{quasi-mono-n} yields
$\lambda^j\le \Lambda_0 \omega\eta^j$, whence
\begin{equation*}
\begin{aligned}
(e^{j+1})^2 + \beta (\eta^{j+1})^2 \le &
(e^j)^2 - \beta(1+\delta)\frac{\xi\theta^2}{2} (\eta^j)^2\\
&+ \beta \Big((1+\delta)\Big(1-\frac{\xi\theta^2}{2}\Big) + 
\Lambda_2 \Big(1+\frac{1}{2\Lambda_3}\Big) \frac{\Lambda_0^2\omega^2}{\beta} \Big)(\eta^j)^2.
\end{aligned}
\end{equation*}
Moreover, $\omega\le\omega_1$ implies
$(\lambda^j)^2 \le \frac{C_2}{2\Lambda_1} (\eta^j)^2$ which turns out
to be \eqref{bound_lambda}. Therefore, applying the bound \eqref{simpler-upper}
and replacing $\beta$ according to \eqref{eq:beta}, 
we obtain 
\begin{equation*}
(e^{j+1})^2 + \beta (\eta^{j+1})^2 \le \alpha_1(\delta) (e^j)^2 
+ \alpha_2(\delta) \beta (\eta^{j})^2
\end{equation*}
with
\begin{equation*}
\alpha_1(\delta)^2 \definedas 1 - \delta \frac{\xi \theta^2}{4 \Lambda_3 (C_1+\frac{C_2}2)},
\quad
\alpha_2(\delta)^2 \definedas (1+\delta) \left(1 - \frac{\xi \theta^2}{2} \right) 
   + \Lambda_2 \left(1 + \frac{1}{2 \Lambda_3} \right) \frac{\Lambda_0^2 \omega^2}{\beta}.
\end{equation*}

\step{3}
It remains to prove that $\delta$ can be chosen so that $\alpha_2(\delta)^2<1$. We then
fix the parameter $\delta$ so that 
\[
(1+\delta) \Big(1-\frac{\xi\theta^2}{2}\Big) = 1 -\frac{\xi\theta^2}{4}
\quad
\Rightarrow
\quad
\delta = \frac{\xi \theta^2}{4 - 2 \xi \theta^2},
\]
We now realize that \eqref{eq:beta} gives $\beta = \frac{\xi \theta^2}
{2 \Lambda_3 (4 - \xi \theta^2)} \ge \frac{\xi \theta^2}
{8\Lambda_3}$ and, since $\omega \leq \omega_2$,
we infer that
\[
\Lambda_2 \left( 1 + \frac{1}{2 \Lambda_3} \right) \frac{\Lambda_0^2\omega^2}{\beta}
\leq 
 \frac{4 \Lambda_2 ( 2\Lambda_3 + 1) }{ \xi \theta^2 } \Lambda_0^2 \omega^2
\le 
\frac{\xi \theta^2}{8}.
\]
Hence $\alpha_2^2 \le 1 - \frac{\xi \theta^2}{8} < 1$, and choosing
$\alpha \definedas \max\{\alpha_1, \alpha_2\}<1$ yields~\eqref{contraction}.

\step{4} The contraction property \eqref{contraction} guarantees that
$\ADAPTPDE$ stops in a finite number of iterations $J$. To show that
$J$ is independent of the outer iteration counter $k$, take $k\ge1$ and note
that before the call $\ADAPTPDE(\T_k^+,\eps_k)$ \rhn{in \AFEM of \S~5}, we have
\[
\eta_k:=\eta_{\T_k}(U_k, F_k)\le
  \eps_{k-1}=\frac{\eps_k}{\rho},
\qquad
\lambda_k:=\lambda_{\T_k}(\gamma) \le \Lambda_0\lambda_{\T_{k-1}^+}(\gamma)
\le \frac{\Lambda_0\omega}{\rho}\eps_k.
\]
We next combine \eqref{reduction}, with $\delta = 1$,  
and the estimate 
$\|\nabla_\gamma(U_k^+ - U_k)\|_{L_2(\gamma)}^2 \Cleq \|\nabla_\gamma(u - U_k)\|_{L_2(\gamma)}^2
+ \lambda_k^2$ arising from \eqref{quasi-ortho}, to get
\begin{align*}
\eta_{\T_k^+}(U_k^+,F_k^+)^2 
&\Cleq \eta_k^2 + \lambda_k^2 + \| \nabla_\gamma(U_k^+ - U_k)\|_{L_2(\gamma)}^2 
\Cleq
\eta_k^2 + \lambda_k^2 + \| \nabla_\gamma(u - U_k)\|_{L_2(\gamma)}^2,
\end{align*}
where $F_k^+$ is the right hand side associated to $\Gamma_{\T_k^+}$
defined in \eqref{def:F} and the hidden constants depend on $\Lambda_2,\Lambda_3$.
The bounds on $\eta_k$, $\lambda_k$, together with \eqref{upper}, yield
\[
(\eta^0)^2=\eta_{\T_k^+}(U_k^+,F_k^+)^2 \Cleq \eta_k^2 + \lambda_k^2
\Cleq \eps_k^2.
\]
Since the stopping condition of $\ADAPTPDE$ is $\eta^J\le\eps_k$,
\eqref{contraction} implies that $J$ is bounded independently of $k$,
as asserted.
\end{proof}

The fact that $J$ is uniformly bounded controls the complexity of $\ADAPTPDE$
because the most expensive module $\SOLVE$ is run just $J$
times. However, this property is not required for the study of cardinality
of \S \ref{S:rates}.

%%%%%%%%%%%%%%%%%%%%%%%%%%%%%%%%%%%%%%%%%%%%%%%%%%%%%%%%%%%%%%%%%%%%%%%%%%%%%%%
\section{Approximation Class}\label{S:approx-class}
%%%%%%%%%%%%%%%%%%%%%%%%%%%%%%%%%%%%%%%%%%%%%%%%%%%%%%%%%%%%%%%%%%%%%%%%%%%%%%%

In this section we discuss the approximation classes $\As$ and 
their connection with Besov regularity.
We start with the notion of total error in \S \ref{S:total-error}
leading to the definition of $\As$. 
We then introduce and discuss a {\it greedy} algorithm in
\S \ref{S:greedy_algo}, that we use repeatedly in the rest of the section.
We study the
best approximation error achievable with piecewise polynomials of
degree $n\ge1$
for the surface $\gamma$ in \S \ref{S:gamma} and for the solution $u$ in \S \ref{S:u-f}.
We analyze the decay rate of oscillation in \S \ref{S:decay-osc}.
Finally in \S \ref{S:member-As} we conclude
with our second main result:  the membership  $(u,f,\gamma) \in \As$ 
in terms of Besov regularity of $u,f$ \rhn{and} $\gamma$.

\manel{In the discussion below we also  remove the hat on functions, when 
no confusion is possible.}

%--------------------------------------------------------------------------------
\subsection{The Total Error}\label{S:total-error}
%--------------------------------------------------------------------------------

Let $\mathbb{T}_N\subset\mathbb{T}:=\mathbb{T}(\T_0)$ 
be the set of all possible conforming 
triangulations, generated on $\gamma$ with at most $N$ elements more than
$\cT_0$ by successive bisection of $\cT_0$:
\begin{equation*}
  \mathbb{T}_N \definedas \big\{\cT \in\mathbb{T} \mid \#\cT - \#\cT_0 \leq N
  \big\}.
\end{equation*}
Given $v\in H^1_\#(\gamma), f\in L_2(\gamma)$ and $V\in \mathbb V(\T)$, we recall the notion of {\it total error}
\[
\gls{indi:total_error} E_\T(V;v,f,\gamma)^2 = \normLtk{\DivG(v-V)}{\gamma}^2 
+ \osc_{\cT}(V,f)^2 
+ \rhn{\omega^{-1} }\lambda_{\cT}(\gamma)^2,
\]
or $E_\T(V;v,f,\gamma)^2 = \E_\T(V,f)^2 + \omega^{-1} \lambda_{\cT}(\gamma)^2$.
Owing to the equivalence of norms \eqref{eq:equiv-norms} we rewrite the first term in the parametric
domain $\Omega$, and omit the factor $\omega$, to obtain the following equivalent notion of total
error provided $\lambda_{\T_0}(\gamma)$ satisfies \eqref{eq:init_cond}:
\begin{equation}\label{total-error-3}
\widehat{E}_\T (V;v,f,\gamma)^2 := %\inf_{V\in\V(\T;\Gamma)}
%\left(
\sum_{T \in \T} \normLtk{\smash{\widehat\nabla( v-  V)}}{\widehat T}^2 
+ \osc_{\cT}(V,f)^2 
%\right)
+ \lambda_{\cT}(\gamma)^2,
\end{equation}
Note that the last two terms are already evaluated in $\Omega$ according to definitions \eqref{geo-estimator}, \eqref{d:osc-def}.
Yet, there is a nonlinear interaction between the approximations of $\gamma$ and of $u,f$ defined on $\gamma$.
At this point, we recall the convention of dropping the patch index when no confusion arises, for example $ v|_{\widehat T}$ for $\widehat T \in \widehat{\cT}^i$ in \eqref{total-error-3} stands for $ v^i|_{\widehat T}$.

Assuming that the parameter $\omega$ satisfies $\omega \leq \omega_1$
in \eqref{bound_omega} to guarantee the validity of Lemma \ref{L:equiv}
\rhn{(equivalence of error and estimator),} along with the
fact that AFEM is driven by $\eta_\T(U,F_\Gamma)$ and
$\lambda_{\cT}(\gamma)$, we assess the quality of the best
approximation of $(v,f,\gamma)$ with $N$ degrees of freedom in terms
of the following modulus of smoothness:
\begin{equation*}
  { \sigma}(N;v,f,\gamma) \definedas \inf_{\cT \in  \mathbb{T}_N} 
  \inf_{V\in\V(\T)}
  \widehat{E}_{\T}(V;v,f,\gamma).
\end{equation*}
This is thus consistent with the approach taken for flat domains in 
\cite{CaKrNoSi:08,NoSiVe:09}.
For $s>0$, we define the nonlinear (algebraic) approximation class $\As$ to be
\begin{equation}\label{approx_class}
 \gls{not:apprxclass} \As \definedas \Big\{ (v,f,\gamma)  : \;
    |v,f,\gamma|_{\mathbb A_s} \definedas \sup_{N\ge1}  
    \big(N^s \; \sigma(N;v,f,\gamma) \big) < \infty
  \Big\}.
\end{equation}
%
\begin{comment}
We emphasize that the approximability of the surface $\gamma$ only appears
implicitly by measuring the errors on $\gamma$.
In fact, the definition of data oscillation \eqref{d:osc}, and in particular the
specific choice of $F_\Gamma$, implies that
$\osc_\Gamma^2(f,\cT)$ entails the approximation of $f$ by piecewise constants
on $\gamma$ but does not include the approximation of $\gamma$ by
$\Gamma$. On the other hand,
\end{comment}
The generic range of $s$ is dictated by the polynomial degree, namely
$0<s\le n/d$.

A useful and alternative definition to $(u,f,\gamma)\in \As$
follows: given $\eps>0$, there exists a subdivision $\cT_\eps$ with $\cT_\eps\ge\cT_0$ and a discrete function $V_\eps\in\V(\T_\eps)$ such
that
\begin{equation}\label{class-As}
\widehat E_{\T_\epsilon}(V_\epsilon;u,f,\gamma) \le  \eps,
\qquad \text{and} \qquad \#\cT_\eps-\#\cT_0 \le |u,f,\gamma|_{\mathbb{A}_s}^{\frac{1}{s}} \eps^{-\frac{1}{s}}.
\end{equation}
The characterization of $\As$ in terms of Besov
regularity is an open issue but we give sufficient conditions for
membership in $\As$ in \S~\ref{S:gamma} and \S~\ref{S:u-f}.
Before doing so, we discuss in \S~\ref{S:greedy_algo} \emph{greedy} algorithms suited to our particular framework, where the subdivisions consist of a collection of compatible subdivisions.

%

%--------------------------------------------------------------------------------
\subsection{Greedy Algorithm}\label{S:greedy_algo}
%--------------------------------------------------------------------------------
In this section we present and discuss a {\it greedy} algorithm
to construct a {\it near best} piecewise polynomial approximation
of a vector-valued function
$\mathbf{g}:=\{g^i\}_{i=1}^M:  \Omega \to \mathbb{R}^M$ in a suitable
semi-norm. Given a mesh $\T :=\cup_{i=1}^M \T^i \in \mathbb{T}$, and a corresponding parametric mesh $\wT :=\cup_{i=1}^M \wT^i \in \widehat{\mathbb{T}}$,
the algorithm requires a {\it local} error estimator $\zeta_{\T^i}(g^i,T)$ for $ T \in \T^i$, $1\le i\le M$. To simplify the notations, we set
\[
\zeta_{\T}(\mathbf{g}, T) := \zeta_{\T^i}(g^i, T),
\quad T \in \T^i, ~1\le i\le M.
\]

We emphasize at this point that approximating the functions $g^i$ requires mesh compatibility conditions on $\partial \Omega$
\rhn{to account for adjacent components $\T^i$ of the entire conforming mesh $\T$. This explains why we employ this somewhat cumbersome notation, which however we will simplify as much as possible below}.  
Given a conforming refinement $\T$ of an initial
triangulation $\T_0$ and a prescribed tolerance
$\delta$, the algorithm reads:
\medskip
\begin{algotab}\gls{algo:GREEDY}
  \> $ \cT_* = \GREEDY ( \mathbf g, \cT,\delta)$ \\
  \> \> 1. if  $ \cM :=\{ T \in   \cT \, : \,
  \zeta_{\T}(\mathbf{g},T)>\delta \} = \emptyset$ \\
  \> \> \> \> return \rhn{$( \cT)$} and exit \\
  \> \> 2.  $ \cT = \REFINE(  \cT,  \cM)$ \\
  \> \> 3. go to 1
\end{algotab}

\smallskip\noindent
where the module $\REFINE$ bisects all elements in the marked set
$\M$ and keeps conformity as described in
\S \ref{S:adapt-pde}.
Note that $\ADAPTSURF$ is a particular instance of $\GREEDY$ that 
uses $\zeta_{\T^i}(\chi^i,T):=\lambda_{\T^i}(\gamma,T)$ as
local error estimator to approximate the patch of the surface $\gamma^i$
parameterized by $\chi^i:\Omega \rightarrow \mathbb R^{d+1}$.
We now discuss some properties of the greedy algorithm following
\cite{NoSiVe:09}. Results of this type started with Birman and Solomyak
for Sobolev spaces \cite{BirSol67}, and continued with \cite{BDDP}
for Besov spaces and \cite{CDDD} for wavelet tree approximation. We do
not refer to any specific norm below.

\vskip0.2cm
\begin{proposition}[Performance of $\GREEDY$]\label{P:greedy_algo}
Let $\T :=\cup_{i=1}^M \T^i$ be created by successive bisections of $\T_0$,
which \rhn{has an admissible labeling.
Let $0 < \mathtt{p}\le \infty$ and let $\mathbf
g:=\{g^i\}_{i=1}^M:\Omega \to \mathbb{R}^M$ be a
vector-valued function} and 
$\{\zeta_{\T}(\mathbf{g},T)\}_{T\in\T}$ be corresponding
local error estimators that satisfy
\begin{equation}\label{eq:zeta-bound}
  \rhn{\zeta_{\T}(\mathbf{g},T) \Cleq h_{T}^r
  |g^i|_{T}, \quad r>0, \quad T \in \T^i, \quad 1\le i \le M,}
\end{equation}
where $h_{T}:=|\widehat T|^{1/d}$ and
$\left(\sum_{i=1}^M\sum_{T \in \T^i}
|g^i|_{T}^\mathtt{p}\right)^{\mathtt {1/p}}\le |
\mathbf g|_{\Omega}$ is a given semi-(quasi) norm (with the convention that
$\max_{i=1,...,M} \max_{T \in \T^i} |g^i|_{T}\le | \mathbf g|_{\Omega} $ if $\mathtt {p} = \infty$).

If $|\mathbf g|_{\Omega} < \infty$, then  the module $\GREEDY (\mathbf g,\cT,\delta)$ terminates in a
finite number of steps and the number of elements marked
$\M$ within $\GREEDY$ satisfies
\begin{equation}\label{eq:greedy-bounds}
%\begin{split}
	\# \M \Cleq 
	|\mathbf g|_\Omega^\frac{d\mathtt{p} }{d + r \mathtt{p}} 
	\delta^{-\frac{d\mathtt{p} }{d + r \mathtt{p}}}.
%	,\qquad 
%	\delta \Cleq  |g|_\Omega \|\Omega\|^{\frac{r}{d}}
%	(\#\T - \#\T_0)^{-\frac{d+r \mathtt{p}}{\mathtt{p}d}}	
%\end{split}	
\end{equation}
\end{proposition}
\vskip-0.2cm
\begin{proof}
The algorithm stops in a finite number of steps because 
the local estimator $\zeta_{\T}(\mathbf{g},T)$
is bounded by a positive 
power of $h_{T}$ according to \eqref{eq:zeta-bound} and the \rhn{admissibility of the labeling of $\T_0$}
ensures that a finite number of refinements \rhn{is} required to guarantee the conformity \rhn{of $\T$}.
To prove \eqref{eq:greedy-bounds} we organize the elements in $ \cM$
by size so that it allows for a counting argument. Let 
$\cP_j$ be the set of elements $T$ of $\cM$ with size (in the parametric domain) satisfying $2^{-(j+1)} \le |\widehat T| < 2^{-j}$, so that
\[
 T \in \cP_j \quad\iff\quad	2^{-(j+1)}\le|\widehat T|<2^{-j} \quad\iff\quad 
        2^{-(j+1)/d}\le h_{T}<2^{-j/d}.
\]
\rhn{We assume that $\T = \T_0$ and} proceed in several steps. %, starting with $\T=\T_0$.

\step{1} We first observe that all $T$'s in $\cP_j$ are
\emph{disjoint}. In fact, if $T_1,\,T_2\in \cP_j$ and
they overlap (their interiors have a nonempty intersection), then one of them is
contained in the other, say $T_1\subset T_2$, due to the bisection
procedure, thus
$|\widehat T_1|\le\frac{1}{2}\,|\widehat T_2|$,
contradicting the definition of $\cP_j$. Then, recalling that we have $M$ copies of $\Omega$, we deduce
\begin{equation}\label{small}
 2^{-(j+1)}\,\#\cP_j
 \le M|\Omega|
\quad\implies\quad
 \#\cP_j\le M|\Omega|\, 2^{j+1}.
\end{equation}
\step{2} We note that $\cP_j= \cup_{i=1}^M \cP_j^i$ where $\cP_j^i$ contains the elements of $\cP_j$ which are refinements of elements in $\T_0^i$.
Each element $T\in \cP_j$ belongs to a subdivision $\T$ \rhn{created by $\GREEDY$ 
so that,} in light of \eqref{eq:zeta-bound} and the fact that $\cP_j \subset \cM$, we have
\[
 T \in \cP_j^i \quad\implies\quad
	\delta\le \zeta_{\T}(\mathbf{g},T) \Cleq 2^{-(j/d)r} 
|g^i|_{T}.
\]
Therefore we have
%
%\[
$\displaystyle
\delta^{\mathtt{p}}\,\#\cP_j\Cleq 2^{-(j/d)r\mathtt{p}}
\sum_{i=1}^M\sum_{T\in{\mathcal P}_j^i} |g^i|^\mathtt{p}_{\widehat{T}}
$
%\]
and 
\begin{equation}\label{large}
	\#\cP_j\Cleq
        \delta^{-\mathtt{p}}\,2^{-(j/d)r\mathtt{p}}\,
        |\mathbf g|^\mathtt{p}_{\Omega}.
\end{equation}
\step{3} The two bounds for $\#\cP_j$ in \eqref{small} and
\eqref{large} are complementary. The first one is good for $j$ small
whereas the second one is suitable for $j$ large (think of $\delta\ll
1$). The crossover takes place for $j_0$ such that 
\[
2^{j_0+1}M|\Omega| \approx \delta^{-\mathtt{p}}\,2^{-j_0r\mathtt{p}/d}
|\mathbf g|_\Omega^{\mathtt{p}}
\quad\iff\quad 
2^{j_0} \approx 
\left( 
 M^{-1} |\Omega|^{-1} \delta^{-\mathtt{p}} |\mathbf g|_\Omega^{\mathtt{p}}
 \right)^\frac{d}{d + r \mathtt{p}}.
 \]
\step{4} We now compute 
\[
\# \cM =\sum_j\#\cP_j\Cleq
\sum_{j\le j_0}2^j+\delta^{-\mathtt{p}}\,
|\mathbf g|_\Omega^{\mathtt{p}} 
\sum_{j>j_0}(2^{-r\mathtt{p}/d})^j.
\]
Since 
$\sum_{j\le j_0}2^j\approx 2^{j_0}$ and
$\sum_{j>j_0}(2^{-r\mathtt{p}/d})^j\Cleq 2^{-(r\mathtt{p}/d)j_0}$,
we can write 
\begin{equation*}
\begin{split}
\# \cM &\Cleq
\left(\,|\mathbf g|_\Omega \delta^{-1} \right)^
{\frac{d\mathtt{p}}{d + r \mathtt{p}}},
\end{split}
\end{equation*}
which is the desired estimate.

\step{5} It remains to remove the \rhn{simplifying} assumption $\T=\T_0$. 
Since $\T$ is a conforming refinement of $\T_0$,
\cite[Proposition 2]{Bonito-DeVore-Nochetto:13} shows that the number
of elements marked by $\GREEDY (\mathbf{g},\T,\delta)$ does
not exceed those marked by $\GREEDY (\mathbf{g},\T_0,\delta)$
and estimated in step 4. This concludes the proof.
\end{proof}

We consider the estimator
\[
\zeta_\T (\mathbf{g}) :=
\{\zeta_\T(\mathbf{g},T)\}_{T \in \T}
\]
and its accumulation in $\ell_\mathtt{q}$ with $0<\mathtt{p}<\mathtt{q}\le\infty$. Its decay rate
is assessed next.
\begin{corollary}[Estimate in $\ell_q$]\label{C:est-lq}
  Let $\zeta_\T (\mathbf{g})$ satisfy \eqref{eq:zeta-bound} with
  $r:=d(s-1/\mathtt{p}+1/\mathtt{q})>0$. Let the initial subdivision $\T_0$
  \manel{have} \rhn{ an admissible labeling}. 
Given $\delta>0$ there exists a conforming mesh refinement 
$\T \in \mathbb{T}$ such that
\begin{equation}\label{g-rate}
\|\zeta_{\T}(\mathbf g)\|_{\ell_\mathtt{q}} \Cleq  \delta,
\qquad \#\T - \#\T_0 \Cleq \#\M \Cleq
|\mathbf{g} |_{\Omega}^{1/s}
 \delta^{-1/s}.
\end{equation}
\end{corollary}
\vskip-0.5cm
\begin{proof}
Since $\frac{d \mathtt{p}}{d+r\mathtt{p}} = \frac{\mathtt{q}}{1+\mathtt{q}s}$,
the output of the call $\T= \GREEDY (\mathbf{g},
\T_0,\epsilon)$ satisfies
\[
\zeta_\T(\mathbf{g}, T) \Cleq \epsilon,
\qquad
\# \mathcal{M} \Cleq |\mathbf{g}|_\Omega^{\frac{\mathtt{q}}{1+\mathtt{q} s}}
\epsilon^{-\frac{\mathtt{q}}{1+\mathtt{q}s}},
\qquad\forall \, T \in \T,
\]
for any $\epsilon>0$ according to Proposition
\ref{P:greedy_algo}. Combining this with the complexity estimate
\eqref{REFINE:complexity} readily implies
\[
\|\zeta_\T(\mathbf{g})\|_{\ell_\mathtt{q}}
\Cleq \#\T^{1/\mathtt{q}} \epsilon
\Cleq \#\M^{1/\mathtt{q}} \epsilon
\Cleq \epsilon^{\frac{\mathtt{q} s}{1+\mathtt{q} s}} |\mathbf{g}|_\Omega^{\frac{1}{1+\mathtt{q}s}}.
\]
If $\epsilon$ satisfies
$\delta = \epsilon^{\frac{\mathtt{q} s}{1+\mathtt{q} s}} |\mathbf{g}|_\Omega^{\frac{1}{1+\mathtt{q} s}}$,
then it is easy to see that \eqref{g-rate} is valid.
\end{proof}

\rhn{
\begin{remark}[Scale invariant error estimates]\label{R:scale-invariant}
\rm The interpolation estimate~\eqref{eq:zeta-bound} will be derived and utilized below for functions $g$ which are products or composition of functions to account for geometry. To illustrate the importance of scale invariance, we consider now the product $g=vw$ with $v,w \in W_p^2(\widehat T)$ and $2 > d/p$ ($1\le p \le \infty$); this implies $L_\infty(\widehat T), W_{2p}^1(\widehat T) \subset W_p^2(\widehat T)$ and $g \in W_p^2(\widehat T)$. Therefore $|g|_T := |g|_{W_p^2(\widehat T)}$ satisfies 
\[
|g|_{W_p^2(\widehat T)} 
\Cleq |v|_{W_p^2(\widehat T)} \|w\|_{L_\infty(\widehat T)}
+ |v|_{W_{2p}^1(\widehat T)} |w|_{W_{2p}^1(\widehat T)} + \|v\|_{L_\infty(\widehat T)} |w|_{W_p^2(\widehat T)}.
\]
This scale invariant expression accumulates correctly in $\ell_p$,
\[
\sum_{T \in \T}
|g|_{W_p^2(\widehat T)}^p 
\Cleq |v|_{W_p^2(\Omega)}^p \|w\|_{L_\infty(\Omega)}^p
+ |v|_{W_{2p}^1(\Omega)}^p |w|_{W_{2p}^1(\Omega)}^p + \|v\|_{L_\infty(\Omega)}^p |w|_{W_p^2(\Omega)}^p
\]
as a consequence of applying the Cauchy-Schwarz inequality to the middle term. In \S\ref{S:Besov} we extend estimates for the product and composition of functions to Besov spaces $B_q^s(L_p(\widehat T))$ with any order of differentiability $s$ and integrability $p$, provided $s > d/p$.
\end{remark}
}

%------------------------------------------------------------------------------
\subsection{Constructive Approximation of $\gamma$}
\label{S:gamma}
%------------------------------------------------------------------------------
We now analyze a constructive approximation of $\gamma$ by
piecewise polynomials based on the $\GREEDY$ algorithm. We also show 
that this algorithm, and hence $\ADAPTSURF$, is $t$-optimal, i.e. 
the set of marked elements satisfies \eqref{optimal:AS}, provided that 
$\gamma$ belongs to a suitable Besov space. 
The case of polynomial degree $n=1$ with regularity of $\gamma$ in terms of
Sobolev scales is discussed in~\cite{BCMN:Magenes}. 
We establish here a result for higher order degree $n \ge 1$ for which
the regularity of $\gamma$ must be measured in Besov scales.

We recall the following compact notation \eqref{bf-notation}:
{
\begin{equation*}
\mathbf{\bXi} := \{\chi^i\}_{i=1}^M, 
\quad
| \bXi |_{B_q^{1+td}(L_q(\Omega))}
:= \max_{i=1,...,M} |\chi^i|_{B^{1+td}_q(L_q(\Omega))}.
\end{equation*}
\rhn{
We also note that the superscript $i$ indicating the patch label is dropped when no confusion arises.}
\begin{corollary}[Constructive approximation of $\gamma$]\label{C:t-optimality}
\rhn{Let $\gamma$ be globally of class $W^1_\infty$ and be
parameterized by $\bXi \in [B^{1+td}_q(L_q(\Omega))]^M$}
with $tq>1, 0<q\le\infty$, $td\leq n$.
Let $\T_0$ \manel{have} \rhn{an admissible labeling}.
Then \manel{$ \T_*=\ADAPTSURF(\T,\delta)$}
is $t$-optimal, i.e. \looseness=-1
$$
\lambda_{\T_*}(\gamma) \le \delta, 
\qquad 
\#\M\Cleq  
C_1(\gamma)^{1/t}
\delta^{-1/t},
$$
where $\M$ denotes the number of elements \rhn{marked during the
execution of the procedure} $\ADAPTSURF\pedro{(\T,\delta)}$
and $C_1(\gamma) \leq  | \bXi |_{B^{1+td}_q(L_q(\Omega))}$
 is the constant in \eqref{optimal:AS}.
\end{corollary}
\begin{proof}
Observe that $B^{1+td}_q(L_q(\widehat T))$,
with $tq>1$ and $0<q\le\infty$ is just above the 
nonlinear Sobolev scale of $W^1_\infty$ in dimension $d$
\cite[p.\ 482]{Triebel:02}, \cite[Lemma 4.12]{GM:14}, so that
\[
B^{1+td}_q(L_q(\widehat T)) \subset B^1_\infty(L_\infty(\widehat T))
\subset W^1_\infty(\widehat T).
\]
Therefore, a scaling argument and local interpolation estimates give the
following bound with $r=dt-d/q>0$ and
\manel{$T\in \T$}
%\cite[Lemma 4.15]{GM:14}, %\cite[Thm 3.5]{DVS:84}, 
%
\begin{equation}\label{termination-lambda}
\lambda_{\cT}(\gamma,T) = \| \widehat\nabla(\chi-
  X_{\cT})\|_{L_\infty(\widehat T)}
\Cleq h_{T}^r |\chi|_{B^{1+td}_q (L_q(\widehat T))}.
\end{equation}
We then can apply Proposition \ref{P:greedy_algo} with $\mathtt{p} = q$,
$\vg = \bXi$, $| \bXi |_\Omega = | \bXi |_{B^{1+td}_q (L_q(\Omega))}$, 
$r = dt-d/q$ and $\zeta_{\T}(\vg,T)=\lambda_{\T}(\gamma,T)$.
Upon termination of $\cT_*=\ADAPTSURF(\T,\delta)$, we obtain
$\lambda_{\T_*}(\gamma) \leq \delta$
along with the asserted estimate on $\#\M$ because
$
\frac{d q}{d + r q} = \frac{1}{t}.
$
\end{proof}

%------------------------------------------------------------------------------
\subsection{Constructive Approximation of $u$}\label{S:u-f}
%------------------------------------------------------------------------------
%
We use the vector notation \eqref{bf-notation} 
\[
\mathbf{u} := \{u^i\}_{i=1}^M,
\quad
|\mathbf{u}|_{B^{1+sd}_p(L_p(\Omega))}
:= \max_{i=1,...,M} |u^i|_{B^{1+sd}_p(L_p(\Omega))},
\]
where $u^i:=u|_{\gamma^i} \circ \chi^i$ and $\gamma^i$ is the
$i$-th surface patch, along with 
$\bV = \{V^i\}_{i=1}^M$
 and %\looseness=-1
\manel{
\[
\|\wnabla(\bu-\bV)\|_{L^2(\Omega)}^2 := \sum_{i=1}^M
\|\wnabla(u^i- V^i)\|_{L^2(\Omega)}^2.
\]}
\begin{corollary}[Constructive approximation of $u$]\label{C:approx-u}
Let $u \in \rhn{H^1_{\#}(\gamma) }$ be piecewise of class
$B^{1+sd}_p(L_p(\Omega))$,
\rhn{namely $\bu \in [B^{1+sd}_p(L_p(\Omega))]^M$,} with $s-1/p+1/2 > 0$,
$0<p\le\infty$ and $0 < sd \le n$. 
Let $\T_0$ \manel{have} \rhn{an admissible labeling.}
Then, given $\delta>0$ there exists a triangulation $\T \in \mathbb{T}$ such that
\begin{equation}\label{u-rate}
\inf_{\bV\in\V(\T)} \|\wnabla (\bu-\bV)\|_{L^2(\Omega)} \Cleq \delta,
\qquad 
\#\M \Cleq C(\bu)^{1/s} \delta^{-1/s},
\end{equation}
where $\M$ is the set of marked elements to create $\T$ and
$C(\bu) = |\bu|_{B^{1+sd}_p(L_p(\Omega))}$.
\end{corollary}
\begin{proof}
Taking $\vg :=\wnabla \bu \in B^{sd}_p(L_p(\Omega))$ and applying Corollary
\ref{C:est-lq} with $\mathtt{q}:=2$
we obtain the desired estimate provided we employ
discontinuous piecewise polynomials of degree $\le n$ over $\wT$.
%which are lifts by $\bX_{\cT}$
%of discontinuous elements over $\wT$.
%To conclude
We finally resort to 
\cite{Veeser:12}, which shows that the error decay is in fact the same
regardless of continuity for approximation of \rhn{globally $H^1(\Omega)$-functions.
This takes care also of continuity of traces across patches.}
\end{proof}

%-----------------------------------------------------------------------------
\subsection{Decay Rate of Oscillation}\label{S:decay-osc}
%-----------------------------------------------------------------------------
In order to study the decay rate of the oscillation $\osc_{\cT}(U,f)^2$,
we split it into two terms that we analyze separately, namely
\begin{equation}\label{d:osc-split}
\osc_{\cT}(U,f)^2 \leq \osc_{\cT}(U)^2 + \osc_{\cT}(f)^2,
\end{equation}
where
\begin{equation}\label{label1}
\osc_{\cT}(U)^2:=\sum_{ T \in \rhn{\T }} \osc_{\cT}(U, T)^2,
\qquad
\osc_{\cT}(f)^2:= \sum_{ T \in \rhn{\T}} \osc_{\cT}(f, T)^2,
\end{equation}
and for $T \in \T$ and $V\in\V(\T)$
\begin{align}
   & \osc_{\cT}(V, T)^2  :={} h_T^2 \Big\| 
      (\text{id} - \Pi^2_{2n-2})
      \widehat{\text{div}} 
      \left( \smash{q_\Gamma \widehat{\nabla}  V \bgG_\Gamma^{-1}} \right)
       \Big\|_{L_2(\widehat{T})}^2 \label{d:osc-U} \\
      & \qquad + h_T  
       \Big\| (\text{id} - \Pi^2_{2n-1}) \big(q_\Gamma^+
       \widehat{\nabla}  V^+ (\bgG_{\Gamma}^+)^{-1} \widehat \bn^+
      - q_\Gamma^-\widehat{\nabla}  V^- (\bgG_{\Gamma}^-)^{-1}
       \widehat\bn^- \big)
      \Big\|_{L_2(\partial\widehat{T})}^2  \nonumber \\
   & \osc_{\cT}(f, T)^2  :={}    h_T^2 \Big\| (\text{id} - \Pi^2_{2n-2}) ( f q) \Big\|_{L_2(\widehat{T})}^2. \label{d:osc-f}
\end{align}
To assess their decay rate,
we resort to the following bound \cite[Lemma 3.2]{CaKrNoSi:08}:
\begin{equation}\label{implicit-interpolation}
  \| (\id - \Pi_m^2) (v V) \|_{L_2(\omega)}
  \le \| (\id - \Pi_{m-n}^\infty) v \|_{L_\infty(\omega)}
  \| V \|_{L_2(\omega)},
\end{equation}
which is valid for $0 \le n \le m$, any domain $\omega$ of $\R^d$ or
$\R^{d-1}$, $V \in \Pn(\omega)$ and $v \in L_\infty(\omega)$.
In fact, since $\Pi_m^2$ is invariant over $\mathbb{P}_m$, we see that
$(\id - \Pi_m^2) (\Pi_{m-n}^\infty v V) = 0$ whence
\[
(\id - \Pi_m^2) (v V) = (\id - \Pi_m^2) [(\id - \Pi_{m-n}^\infty)v V].
\]
This yields \eqref{implicit-interpolation} for any interpolant
$\Pi_{m-n}^\infty v$ via the $L^2$-stability of $\Pi_m^2$.

We now embark on the study of the decay rate of oscillation:
we investigate $ \osc_{\cT}(f)$ in \S~\ref{S:osc-f}
and $\osc_{\cT}(U)$ in \S~\ref{S:oscillation}.

%------------------------------------------------------------------------------
\subsubsection{Decay Rate of $\osc_{\cT}(f)$}\label{S:osc-f}
%------------------------------------------------------------------------------
We employ the vector notation \eqref{bf-notation} \rhn{with $\mathbf{f} := \{f^i\}_{i=1}^M$ and $\mathbf{q} = \{q^i\}_{i=1}^M$}.
%
%\[
%\mathbf{f} := \{f^i\}_{i=1}^M,
%\qquad
%|\mathbf{f}|_{B^{sd}_p(L_p(\Omega))}
%:= \max_{i=1,...,M} |f^i|_{B^{sd}_p(L_p(\Omega))},
%\]
%%
%where $f^i:=f|_{\gamma^i} \circ \chi^i$ and $f|_{\gamma^i}$ is the restriction of
%$f$ to the surface patch $\gamma^i$ for $1\le i \le M$.
We also recall that the superscript $i$ indicating the patch label is omitted when no confusion arises.

\begin{corollary}[\rhn{Decay rate of $\osc_{\cT}(f)$}]\label{C:approx-f}
\rhn{Let $\gamma$ be globally of class $W_\infty^1$ and be
parameterized by $\bXi \in [B^{1+td}_q(L_q(\Omega))]^M$ with $tq>1$, 
$0<q\leq\infty$, $td\leq n$, and let \manel{$k = \lfloor td \rfloor + 1$}.
Let $f \in L_2(\gamma)$ be piecewise of class $B^{sd}_p(L_p(\Omega))$,
namely $\mathbf{f} \in [B^{sd}_p(L_p(\Omega))]^M$,
with $s-1/p+1/2>0, 0<p\le\infty$ and $sd \le n$.
Let $\T_0$ \manel{have} an admissible labeling.}
Then, given $\delta > 0$ there exists a triangulation $\T \in \grids$ such that
\begin{equation}\label{eq:oscf_rate}
\begin{split}
\osc_{\T}(f) \Cleq \delta,
\qquad
\#\T - \#\T_0  \Cleq  C(\mathbf f, \gamma)^{\frac1{s\wedge t+1/d}} \,\delta^{-\frac{1}{s\wedge t+1/d}},
\end{split}
\end{equation}
where $s\wedge t := \min\{s,t\}$ and
\[
C(\mathbf f,\gamma) :=  \| \mathbf f
\|_{B^{sd}_p(L_p(\Omega))}
 \Big(
\| \bXi\|_{B_q^{1+td}(L_q(\Omega))} + 
\| \bXi\|_{B_q^{1+td}(L_q(\Omega))}^k \Big) .
\]
\end{corollary}
\begin{proof}
Since $\Pi^2_{2n-2} (f q)\in\mathbb{P}_{2n-2}$ is the best $L_2$-approximation of $fq$,
we see that
\begin{align*}
\| fq - \Pi^2_{2n-2} ( f q) \|_{L_2(\widehat T)}
&\le \| fq - \Pi^2_{2n-2} ( V q) \|_{L_2(\widehat T)} \\
&\le \| (f-V)q \|_{L_2(\widehat T)}  + \| Vq - \Pi^2_{2n-2} ( V q) \|_{L_2(\widehat T)}  \\
&\le \| f-V \|_{L_2(\widehat T)}\| q \|_{L_\infty(\widehat T)}  
   + \| q - \Pi^\infty_{n-1}  q \|_{L_\infty(\widehat T)}  \| V \|_{L_2(\widehat T)},
\end{align*}
for all $V\in\mathbb{P}_{n-1}$
due to~\eqref{implicit-interpolation}.
Taking $V = \Pi_{n-1}^2f$ we have $\| V \|_{L_2(\widehat T)} \le  \| f \|_{L_2(\widehat T)}$ and  
\begin{align*}
\osc_{\cT}(f, T)
\le \hT \| f-\Pi_{n-1}^2 f \|_{L_2(\widehat T)}\| q \|_{L_\infty(\widehat T)}  
   + \hT \| q - \Pi^\infty_{n-1}  q \|_{L_\infty(\widehat T)}  \| f \|_{L_2(\widehat T)}.
\end{align*}
We now introduce for each $T\in\T$
\begin{equation}\label{local-est-f-q}
  E_{2,\T}(\mathbf{f},T):=
  \hT \|f^i-\Pi_{n-1}^2 f^i \|_{L_2(\widehat T)},
 \quad
 E_{\infty,\T}(\bq,T):= \hT \| q^i - \Pi^\infty_{n-1} q^i\|_{L_\infty(\widehat T)}
\end{equation}
and notice that an immediate generalization
of \cite[Lemma 4.15]{GM:14} implies
the local error estimates
\begin{align*}
E_{2,\T}(\mathbf{f},T)
\Cleq \hT^{r_2}   | f |_{B_p^{sd}(L_p(\widehat T))},
\qquad 
E_{\infty,\T}(\bq,T)
\Cleq \hT^{r_\infty}  |q|_{B_q^{td}(L_q(\widehat T))}
\end{align*}
with \rhn{$r_2 = d[(s+1/d)-1/p+1/2]>1$ and $r_\infty = d[(t+1/d)-1/q]>1$}.
Moreover, \rhn{Corollary~\ref{C:composition-norm} yields
$\bq \in [B_q^{td}(L_q(\Omega))]^M$
because $\widehat{\nabla}\bXi\in [B_q^{td}(L_q(\Omega))]^M$ and}
\begin{equation}\label{smoothness of q}
\| \bq \|_{L_\infty(\Omega)} \Cleq\|\bq\|_{B_q^{td}(L_q(\Omega))} \Cleq
\max\left\{\| \bXi \|_{B_q^{1+td}(L_q(\Omega))}, \| \bXi\|_{B_q^{1+td}(L_q(\Omega))}^k
\right\},
\end{equation}
with $k=\lfloor td \rfloor + 1$.
Given a constant $\ab{c_2}>0$ to be determined later, we resort to Corollary \ref{C:est-lq}
with tolerance $\ab{c_2}\delta$, the local indicator $E_{2,\T}(\mathbf{f},T) $,
$\mathtt{q}=2$ and $r=r_2$, to obtain a mesh
$\T_2 \in \mathbb{T}$ that satisfies
\begin{equation}\label{f-rate}
E_{2,\T_2}(\mathbf{f}) \Cleq \ab{c_2}\delta,
\qquad 
 \# \T_2 - \# \T_0 \Cleq
 |\mathbf{f}|^{\frac{1}{s+1/d}}_{B^{sd}_p(L_p(\Omega))} (\ab{c_2} \delta)^{-\frac{1}{s+1/d}}.
\end{equation}
\rhn{Invoking Corollary~\ref{C:est-lq} once again, this time with 
tolerance $\ab{c_\infty}\delta$, constant $c_\infty$ to be chosen later, local indicator 
$E_{\infty, \T}(\bq, T)$, $\mathtt{q}=\infty$ and $r=r_\infty$,}
we find a mesh $\T_\infty \in \mathbb{T}$ such that 
\begin{equation}\label{q-rate}
E_{\infty,\T_\infty}(\mathbf q) \Cleq \ab{c_\infty}\delta,
\qquad 
\# \T_\infty - \# \T_0 \Cleq
|\mathbf{q}|^{\frac{1}{t+1/d}}_{B^{td}_q(L_q(\Omega))}
(\ab{c_\infty} \delta)^{-\frac{1}{t+1/d}}.
\end{equation}
If $\T = \T_2 \oplus \T_\infty$ is the overlay of the meshes
$\T_2$ and $\T_\infty$, then
it remains to show that $\T$ satisfies \eqref{eq:oscf_rate}. Since the
local indicators \eqref{local-est-f-q} are monotone, i.e they do not increase
with refinement, we deduce from \eqref{f-rate} and \eqref{q-rate}
\begin{align*}
\osc_{\cT}(f)^2 &
\Cleq E^2_{2,\T}(\mathbf f) \, \| \bq \|^2_{L_\infty(\Omega)}  
+ E^2_{\infty, \T}(\mathbf q) \, \| \mathbf{f} \|^2_{L_2(\Omega)} 
\\
& \Cleq \delta^2 \left( \ab{c_2}^2 \| \bq \|^2_{L_\infty(\Omega)}
+ \ab{c_\infty}^2 \| \mathbf{f} \|^2_{L_2(\Omega)}  \right).
\end{align*}
We now choose the constants $\ab{c_2}$ and $\ab{c_\infty}$ as follows:
\[
\ab{c_2} = \|\bq\|_{B_q^{td}(L_q(\Omega))}^{-1},
\qquad
\ab{c_\infty} = \| \mathbf f \|_{B^{sd}_p(L_p(\Omega))}^{-1}.
\]
This implies $\osc_{\cT}(f)\Cleq \delta$ in view of \eqref{smoothness of q}
and $\| \mathbf f \|_{L_2(\Omega)} \Cleq \| \mathbf f \|_{B^{sd}_p(L_p(\Omega))}$.
Finally, since 
$
\#\T \le \#\T_2 + \#\T_\infty - \#\T_0
$
according to \cite[Lemma 3.7]{CaKrNoSi:08}, we obtain
\begin{equation*}
\#\T - \# \T_0 \leq  (\#\T_2 - \#\T_0)  + (\#\T_\infty  - \#\T_0 )
\Cleq C(\mathbf f, \gamma) \,
\delta^{-\frac{1}{s\wedge t + 1/d}}.
\end{equation*}
This follows from \eqref{f-rate} and \eqref{q-rate} upon replacing the
exponents $s$ and $t$ by $s\wedge t$ \rhn{because their left-hand sides are
always larger than or equal to 1, and eventually using Corollary~\ref{C:composition-norm} to estimate $\|\bq\|_{B_q^{td}(L_q(\Omega))}$. This concludes the proof.}
\end{proof}
\begin{remark}[Besov regularity of $f$]\label{R:Besov-f}
\rm If $u^i\in B^{1+sd}_p(L_p(\Omega))$ and $\gamma$ is smooth, then
$f^i = - q^{-1} \widehat{\text{div}} \big( q \wnabla u^i \bgG^{-1} \big)\in
B^{sd-1}_p(L_p(\Omega))$ is the natural Besov regularity for $f^i$. 
However, we require that $f^i\in B^{sd}_p(L_p(\Omega))$ because the data oscillation is
evaluated in $L_2(\Omega)$ rather than $H^{-1}(\Omega)$. This
additional degree of regularity of $f$ is responsible for the faster decay of
$\osc_{\cT}(f)$ reported in Corollary \ref{C:approx-f}.
\end{remark}

%------------------------------------------------------------------------------
\subsubsection{Decay Rate of $\osc_{\T}(U)$}\label{S:oscillation}
%------------------------------------------------------------------------------

In this section we study the decay rates of $\osc_{\T}(U)$ defined in~\eqref{d:osc-U}.
 We again use the $\GREEDY$ algorithm where now
the local indicator will be $\osc_{\T}(V,T)$ for $V\in \V(\T)$. 

We start with an estimate for $\osc_{\cT}(V,T)$ in terms
of a positive  power of  $h_T$. 
In view of expression \eqref{d:osc-U} for $\osc_{\cT}(V,T)$,
the major non-standard obstruction is the presence of the surface dependent and non-polynomial term $q_\Gamma \bgG_\Gamma^{-1}$.
This requires two auxiliary
results about Besov spaces, namely Corollary \ref{C:Besov-product-seminorm}
(scale-invariant Besov semi-norm of products of functions) and 
Lemma \ref{L:composition-semi} (scale-invariant Besov norm of composition),
which we prove later in \S~\ref{S:Besov} not to interrupt the flow.

We are now in position to show that $\osc_{\T}(V,T)$ is
bounded by a positive power of $h_{T}$.
The proof of Proposition \ref{P:oscU-bound}
is a consequence of the subsequent three lemmas.\looseness=-1
\medskip
\begin{proposition}[Local decay of oscillation]\label{P:oscU-bound}
\rhn{Let the surface $\gamma$ be globally of class $W_\infty^1$ and
be parameterized by $\bXi \in [B^{1+td}_q(L_q(\Omega))]^M$} with $tq>1$, 
$0<q\leq\infty$, $td\leq n$.  
For all $\T \in \mathbb T$ and all $V \in \mathbb V(\cT)$
\begin{equation}\label{eq:oscUT-bound}
\osc_{\cT}(V,T) \Cleq
h_{T}^r 
\left(
\sum_{j=1, \ell_j = j/k^ {k}}^{k^{k}}
\rhn{| \bXi|^{1/\ell_j}_{B^{1+td \ell_j}_q(L_{q/\ell_j}(N_\T(\widehat T));\T)} } \right)
\| \widehat{\nabla}  V \|_{L_2(N_\T(\widehat T))} ,
\end{equation}
with $r = td - d/q > 0$ and $k = \lfloor td \rfloor + 1$,
$N_\T(\widehat T)$ is the set containing $\widehat T$ and its
adjacent elements, and $B^{1+td \ell_j}_q(L_{q/\ell_j}(N_\T(\widehat T));\T)$
indicates the broken Besov space.
\end{proposition}

\medskip
\begin{lemma}[Element oscillation]\label{L:elem-osc}
	Under the assumptions of Proposition \ref{P:oscU-bound},
	for all $V \in \mathbb V(\T)$ and all $T \in \T$,
        we have
	\begin{align*}
	h_T \left\| 
	(\text{id} - \Pi^2_{2n-2})
	 \widehat{\mathrm{div}} 
	(q_\Gamma \widehat{\nabla}  V \bgG_\Gamma^{-1})
	\right\|_{L_2(\widehat{T})}
	 \Cleq 
	  h_{T}^r 
		|q_\Gamma \bgG_\Gamma^{-1} |_
		{B^{td}_q(L_{q}(\widehat{T}))}
	\| \widehat{\nabla}  V \|_{L_2(\widehat T)},
	\end{align*}
        with $r = td - d/q > 0$.
\end{lemma}
\begin{proof}
	  We first observe that
	  \[
	  \widehat{\text{div}} 
	  (q_\Gamma \widehat{\nabla}  V \bgG_\Gamma^{-1})
	  = \widehat{\text{div}} \left(
	  q_\Gamma \bgG_\Gamma^{-1} \right) \cdot \widehat{\nabla}  V +
	  q_\Gamma \bgG_\Gamma^{-1} :\widehat{D}^2 V,
	  \]
	  and by~\eqref{implicit-interpolation}
	  \begin{equation*}
	  \begin{split}
	  \left\| 
	  (\text{id} - \Pi^2_{2n-2})
	  \widehat{\text{div}} 
	  (q_\Gamma \widehat{\nabla}  V\bgG_\Gamma^{-1})
	  \right\|_{L_2(\widehat{T})}
	  \Cleq{}&
	  \left\| 
	  (\text{id} - \Pi^\infty_{n-1})
	  \widehat{\text{div}} \left(
	  q_\Gamma \bgG_\Gamma^{-1} \right)\right\|_{L_\infty(\widehat T)}
	  \big\| \widehat{\nabla}  V \big\|_{L_2(\widehat T)}
	  \\
	  &+
	  \big\| 
	  (\text{id} - \Pi^\infty_{n-1})
	  (q_\Gamma \bgG_\Gamma^{-1}) \big\|_{L_\infty(\widehat T)}
	  \big\| \widehat{D}^2 V \big\|_{L_2(\widehat T)} .
	  \end{split}
	  \end{equation*}
	  Using interpolation estimates in Besov norms
          of an immediate generalization of  \cite[Lemma 4.15]{GM:14}
	we have
	\[
	\left\|
	(\text{id} - \Pi^\infty_{n-1})
	\widehat{\text{div}} \left(
	q_\Gamma \bgG_\Gamma^{-1} \right)\right\|_{L_\infty(\widehat T)}
	\lesssim
	h_{T}^r \left|\widehat{\text{div}} \left(
	q_\Gamma \bgG_\Gamma^{-1} \right)\right|_{B^{td}_q(L_q(\widehat T))} ,
	\]
	with $0< r \le n$.
	By the inverse inequality of Lemma~\ref{L:inverse-Besov}, we
        readily get
	\[
	\left|\widehat{\text{div}} \left(
	q_\Gamma \bgG_\Gamma^{-1} \right)\right|_{B^{td}_q(L_q(\widehat T))}
	\lesssim \frac{1}{h_{T}} \left|
	q_\Gamma \bgG_\Gamma^{-1} \right|_{B^{td}_q(L_q(\widehat T))}.
	\]
	A similar argument applied to the second term gives
	\[
	\big\| 
	(\text{id} - \Pi^\infty_{n-1})
	(q_\Gamma \bgG_\Gamma^{-1}) \big\|_{L_\infty(\widehat T)}
	\big\| \widehat{D}^2  V\big\|_{L_2(\widehat T)} 
	\lesssim
	h_{T}^r
	\left|q_\Gamma \bgG_\Gamma^{-1}\right|_{B^{td}_q(L_q(\widehat T))}
	\frac{ \big\| \widehat{\nabla}  V \big\|_{L_2(\widehat T)}}{h_{T}} .
	\]
	This proves the asserted estimate.
\end{proof}

\begin{lemma}[Jump oscillation]\label{L:jump-osc}
Let the assumptions of Proposition \ref{P:oscU-bound} be valid.
For all $T \in \T$ and all $V \in \V(\T)$, there holds
\begin{align*}\label{e:estim_jump_osc}
h_T^{1/2}
\Big\| (\text{id} - \Pi^2_{2n-1}) \Big( q_\Gamma^+ \widehat{\nabla}  V^+
(\bgG_{\Gamma}^+)^{-1} \widehat{\bn}^+
     & - q_\Gamma^- \widehat{\nabla} V^- (\bgG_{\Gamma}^-)^{-1}
      \widehat{\bn}^- \Big) \Big\|_{L_2(\partial \widehat T)}
\\
 & \Cleq  h_{T}^r 
 |q_\Gamma \bgG_\Gamma^{-1} |_{B^{td}_q(L_{q}(N_T(\widehat{T}));\T)}
 \| \widehat{\nabla}  V\|_{L_2(N_T(\widehat{T}))},
\end{align*}
\vskip-0.05cm
\noindent
with $r = td - d/q > 0$.

\end{lemma}
\begin{proof}
Let $\widehat S = \widehat{T}^+\cap\widehat{T}^-$ be any side of
$\widehat{T}^+ :=\widehat{T}$. Since 
\begin{equation*}
\begin{split}
h_{T}^{1/2} \Big\|
(\text{id} - \Pi^2_{2n-1})
& \Big(q_\Gamma^+\widehat{\nabla}  V^+(\bgG_\Gamma^+)^{-1} \widehat{\bn}^+
- q_\Gamma^-\widehat{\nabla}  V^-(\bgG_\Gamma^-)^{-1} \widehat{\bn}^-
\Big) \Big\|_{L_2(\widehat{S})} \\
\Cleq {}&
h_{T}^{1/2} 
\Big \| (\text{id} - \Pi^2_{2n-1})
\Big(q_\Gamma^+\widehat{\nabla}  V^+(\bgG_\Gamma^+)^{-1}\widehat{\bn}^+\Big)\Big\|_{L_2(\widehat{S})}\\
&+ h_{T}^{1/2} 
\Big\| (\text{id} - \Pi^2_{2n-1})
\Big(q_\Gamma^-\widehat{\nabla}  V^-(\bgG_\Gamma^-)^{-1}\widehat{\bn}^-\Big)\Big\|_{L_2(\widehat{S})}, 
\end{split}
\end{equation*}
we estimate each term separately, dropping the $\pm$ superscript. We  invoke~\eqref{implicit-interpolation} to deduce that
\begin{equation*}
\Big\| (\text{id} - \Pi^{2}_{2n-1})
\big(q_\Gamma\widehat{\nabla}  V\bgG_\Gamma^{-1}
\widehat{\bn} \big) \Big\|_{L_2(\widehat{S})}
\leq \big\| (\text{id}-\Pi^{\infty}_{n-1})\big(q_\Gamma \bgG_\Gamma^{-1} \widehat{\bn}\big) \big\|_{L_\infty(\widehat{S})} \big\| \widehat{\nabla}  V\big\|_{L_2(\widehat{S})},
\end{equation*}
where \rhn{$\Pi^{\infty}_{n-1}$ is the Lagrange interpolation operator}
%$L_\infty(\widehat S)$ projector 
onto $\P_{n-1}(\widehat S)$. 
Since the unit normal $\widehat{\bn}$ is constant on $\widehat{S}$,
the interpolation estimate  from an immediate generalization
of~\cite[Lemma 4.15]{GM:14}
reveals that
$$
	\| (\text{id}-\Pi^{\infty}_{n-1})\big(q_\Gamma
        \bgG_\Gamma^{-1} \widehat{\bn} \big) \|_{L_\infty(\widehat{S})} \Cleq 
	h_{T}^r \big|q_\Gamma \bgG_\Gamma^{-1} \big|_{B^{td}_q(L_q(\widehat{T}^\pm))},
$$ 
where we have used the assumption $r \leq n$.
This together with a scaled trace estimate
$\|\widehat{\nabla}  V\|_{L_2(\widehat{S})} 
\Cleq h_{T}^{-1/2} \| \widehat{\nabla}  V \|_{L_2(\widehat{T}^\pm)}$
yields the desired estimate.
\end{proof}

We see from Lemmas \ref{L:elem-osc} and \ref{L:jump-osc} that the
  discrete surface $\Gamma$ enters the estimates via
  $|q_\Gamma\bgG_\Gamma^{-1}|_{B^{td}_q(L_q(\widehat T))}$. 
 The next
  lemma provides control of this term.

\medskip
\begin{lemma}[Besov semi-norm of $q_\Gamma\bgG_\Gamma^{-1}$]\label{L:Besov-qG}
Let the assumptions of Proposition~\ref{P:oscU-bound} hold
and $k = \lfloor td \rfloor + 1$.
We then have for $T \in \T^i$, $1\le i \le M$,
\begin{align*}
|q_\Gamma \bgG_\Gamma^{-1} |_{B^{td}_q(L_{q}(\widehat T))}
\Cleq
\sum_{j=1, \ell_j = j/k^ {k}}^{k^{k}}
\rhn{| \pedro{\chi^i} |^{1/\ell_j}_{B^{1+td\ell_j}_q(L_{q/\ell_j}(\widehat T))}.}
\end{align*}
\end{lemma}
\begin{proof}
\pedro{We fix $i$ and drop it from the notation throughout this proof.}
We invoke Lemma \ref{L:Besov-product} (scale-invariant Besov semi-norm
of the product of two functions) along with the H\"older inequality
to write
\begin{equation*}\label{eq:qG-bound1}
\begin{aligned}
  |q_\Gamma \bgG_\Gamma^{-1} |_
  {B^{td}_q(L_{q}(\widehat T))}
  \Cleq{}&
  \big\|q_\Gamma \big\|_{L_\infty(\widehat T)}
  \big|\bgG_\Gamma^{-1} \big|_{B^{td}_q(L_{q}(\widehat T))} \\
 &+  
  \sum_{
  	\begin{smallmatrix} m=1 \\ \ell_m := m/k
  	\end{smallmatrix}
  	}^{k-1}
  \big|q_\Gamma \big|^{1/\ell_m}_
  {B^{td \ell_m}_q(L_{q/\ell_m}(\widehat T))}
  +
  \big|\bgG_\Gamma^{-1} \big|^{1/(1-\ell_m)}_
  {B^{td (1-\ell_m)}_q(L_{q/(1-\ell_m)}(\widehat T))}\\
&+
  \big|q_\Gamma \big|_{B^{td }_q(L_{q}(\widehat T))}
  \big\|\bgG_\Gamma^{-1} \big\|_{L_\infty(\widehat T)} 
  \\
  \Cleq{}& \sum_{
    	\begin{smallmatrix} m=1 \\ \ell_m := m/k
    	\end{smallmatrix}
    	}^{k}
    \big|q_\Gamma \big|^{1/\ell_m}_
    {B^{td \ell_m}_q(L_{q/\ell_m}(\widehat T))}
    +
    \big|\bgG_\Gamma^{-1} \big|^{1/\ell_m}_
        {B^{td \ell_m}_q(L_{q/\ell_m}(\widehat T))},
  \end{aligned}
\end{equation*}
because $\big\|q_\Gamma \big\|_{L_\infty(\widehat T)}, 
\big\|\bgG_\Gamma^{-1} \big\|_{L_\infty(\widehat T)} \Cleq 1$
for $\gamma$ being globally Lipschitz.
We denote
\[
s^* = td \ell_m, \quad q^* = q/\ell_m, \quad 1\le m \le k,
\]
and bound $|q_\Gamma |_{B^{s^*}_q(L_{q^*}(\widehat T))}$ using
Lemma \ref{L:composition-semi} for $q_\Gamma = \big(\det ( \rhn{\widehat{\nabla}
X_{\cT}^T} \widehat{\nabla} X_{\cT})\big)^{1/2}$:
\begin{equation*}
|q_\Gamma|_{B^{s^*}_q(L_{q^*}(\widehat{T}))}
\Cleq 
\sum_{\ell=1}^{k} \sum_{i=1}^\ell \| \widehat{\nabla} X_{\cT} \|_{L_\infty(\widehat{T})}^{\ell-i}
\sum_{\sum_{j=1}^i\ell^*_j=1} \prod_{j=1}^i | \widehat{\nabla} X_{\cT} |_{B^{s^*\ell^*_j}_q(L_{q^*/\ell^*_j}(\widehat{T}))};
\end{equation*}
a similar bound is valid for $\bgG_\Gamma^{-1}$.
To get a simpler expression, we observe that $\sum_{j=1}^i \ell_j^*=1$ with
\rhn{$0 \le \ell_j^*=i_j/k^{i-1} \le 1$} and $i_j\in\N_0$, whence there cannot
be more than $i-1$ vanishing $\ell_j^*$'s in each product
$\prod_{j=1}^i |\widehat{\nabla} X_{\cT}|_{B^{s^*\ell^*_j}_q(L_{q^*/\ell^*_j}(\widehat{T}))}$. Therefore
\begin{equation*}
\begin{split}
\prod_{j=1}^i | \widehat{\nabla} X_{\cT} |_{B^{s^*\ell^*_j}_q(L_{q^*/\ell^*_j}(\widehat{T}))}
\leq  \; 
\mbox{max}\left(1, \| \widehat{\nabla} X_{\cT} \|_{L_\infty(\widehat T)}^{i-1} \right)
\prod_{
	\begin{smallmatrix} j=1 \\ \ell_j^* \neq 0 \end{smallmatrix}}^i
|\widehat{\nabla} X_{\cT} |_{B^{s^* \ell_j^*}_q(L_{q^*/\ell_j^*}(\widehat T))},
\end{split}
\end{equation*}
because $|\widehat{\nabla} X_{\cT}|_{B^0_\infty(L_\infty(\widehat T))} = \|\widehat{\nabla} X_{\cT}\|_{L_\infty(\widehat T)}$.
Now, using $ i \leq \ell \leq k$, we obtain 
 \begin{equation*}
   |q_\Gamma |_{B^{s^*}_q(L_{q^*}(\widehat{T}))} \Cleq
\mbox{max}\left(1, \| \widehat{\nabla} X_{\cT} \|_{L_\infty(\widehat T)}^{k-1} \right)
\sum_{\ell = 1}^{k} \sum_{i=1}^\ell \sum_{\sum_{j=1}^i \ell_j^* = 1} 
\prod_{
	\begin{smallmatrix} j=1 \\ \ell_j^* \neq 0 \end{smallmatrix}}^i
|\widehat{\nabla} X_{\cT} |_{B^{s^* \ell_j^*}_q(L_{q^*/\ell_j^*}(\widehat T))},
 \end{equation*}
 and remove the first factor in light of \eqref{interpolation-W1}
 and $\chi$ being globally Lipschitz.
 Since $\sum_{j=1}^i \ell_j^* = 1$ and $\ell_j^*>0$, we employ H\"older's inequality to 
 estimate each product as
 \begin{equation*}
 \prod_{
 	\begin{smallmatrix} j=1 \\ \ell_j^* \neq 0 \end{smallmatrix}}^i
 |\widehat{\nabla} X_{\cT} |_{B^{s^* \ell_j^*}_q(L_{q^*/\ell_j^*}(\widehat T))}
 \Cleq 
 \sum_{\begin{smallmatrix}
	 	 j=1 \\ \ell_j^* \neq 0
 	   \end{smallmatrix}
 	  }^i
 	  |\widehat{\nabla} X_{\cT}|_{B^{s^* \ell_j^*}_q(L_{q^*/\ell_j^*}(\widehat T))}
 	  ^{1/\ell_j^*}.
 \end{equation*}
 Combining this with the preceding expression, and taking into account
   that the numbers of appearances of each
   $|\widehat{\nabla} X_{\cT} |_{B^{s^* \ell_j^*}_q(L_{q^*/\ell_j^*}(\widehat T))}$ only depends of $k$,
  we have
  \begin{equation*}
  |q_\Gamma |_{B^{s^*}_q(L_{q^*}(\widehat T))}^{1/\ell_m}
  \Cleq
  \sum_{\begin{smallmatrix}
  	j=1 \\ \ell_j^* = j/k^{k-1}
  	\end{smallmatrix}
  }^{k^{k-1}}
  |\widehat{\nabla} X_{\cT} |_{B^{s^* \ell_j^*}_q(L_{q^*/\ell_j^*}(\widehat T))}
  ^{1/\ell_j^*\ell_m},
  \end{equation*}
where $\ell_j^* = j / k^{k-1}$ has been redefined to fit all possible
cases. It suffices now to realize that $\ell_j^*\ell_m$ can be written as
$\ell_n = \ell_j^*\ell_m = n/k^k$ for $1\le n \le k^k$ and
\[
s^*\ell_j^* = t d \ell_n, \quad q^*/\ell_j^* = q /\ell_n.
\]
Finally, applying Lemma \ref{L:interpolation} (local stability
of Lagrange interpolation) \rhn{with $s=1+t d \ell_n$, so that $d\ell_n/q < s \le n+1$}, 
we may replace $X_{\cT}$ by $\chi$ in $\widehat T$
and thereby obtain the asserted estimate.
\end{proof}

\medskip
\begin{proposition} [Uniform decay rate of $\osc_{\cT}(V)$] \label{P:oscU}
\rhn{Let $\gamma$ be globally of class $W^1_\infty$ and be
parameterized by $\bXi \in [B^{1+td}_q(L_q(\Omega))]^M$}
with $tq>1$, $td\leq n$.
Let \rhn{$\T_0$ \manel{have} an admissible labeling} and let $\T\ge\T_0$
be a refinement of $\T_0$.
Then, for any tolerance $\delta>0$
there exists a subdivision $\T_\delta \in \grids$ such that
$\T_\delta\ge\T$ and
\[
\max_{V\in \V(\T_\delta)}
\frac{\osc_{\T_\delta} (V)}{\|\widehat \nabla \bV
  \|_{L_2(\Omega)}} \Cleq \delta ,  \qquad 
\# \M \Cleq C_2(\gamma)^{1/t} \,  \delta^{-1/t},
\]
where $\M$ is the set of elements marked to create
$\T_\delta$ from $\T$ and
the constant $C_2(\gamma)$ depends on $\gamma$ and is given explicitly by
\[
C_2(\gamma) := \max \left(  
  	\| \bXi \|_{B^{1+td}_q(L_{q}(\Omega))},
  	\| \bXi \|_{B^{1+td}_q(L_{q}(\Omega))}^{k^k}
  	\right) ,
\]
\manel{with $k= \lfloor td \rfloor +1$.}
\end{proposition}
\begin{proof}
We make use of the $\GREEDY$ algorithm upon taking
  $\mathtt{p} = q$, $\vg = \bXi$, and 
\begin{align}\label{zeta-oscV}
  \zeta_{\T}(\bXi, T) = h_{T}^r |\chi|_{\widehat T},
  \qquad
  | \chi |_{\widehat T} =
  \sum_{j=1, \ell_j = j/k^ {k}}^{k^{k}} | \chi|^{1/\ell_j}_{B^{1+td \ell_j}_q(L_{q/\ell_j}(\widehat T))},
\end{align}	
  with $r =  d(t - \frac{1}{q}) >0$.
\rhn{Since} %The assumptions of Proposition~\ref{P:greedy_algo} hold because 
\begin{align*}
 |\bXi |_\Omega^q &= \sum_{\widehat T \in \wT} |\chi|_{\widehat T}^q = 
\sum_{\widehat T \in \wT} 
\bigg(
\sum_{\begin{smallmatrix}
	j=1 \\ \ell_j = j/k^ {k}
	\end{smallmatrix}
	}^{k^{k}}
	| \chi |^{1/\ell_j}_{B^{1+td \ell_j}_q(L_{q/\ell_j}(\widehat T))}
\bigg)^{q} 
\\ 
& \Cleq \sum_{\widehat T \in \wT} 
\sum_{\begin{smallmatrix}
	j=1 \\ \ell_j = j/k^ {k}
	\end{smallmatrix}
}^{k^{k}}
| \chi |^{q/\ell_j}_{B^{1+td \ell_j}_q(L_{q/\ell_j}(\widehat T))}
 \Cleq
 \sum_{\begin{smallmatrix}
 	j=1 \\ \ell_j = j/k^ {k}
 	\end{smallmatrix}
 }^{k^{k}}
 | \bXi |^{q/\ell_j}_{B^{1+td \ell_j}_q(L_{q/\ell_j}(\Omega))},
\end{align*}
and % together with
\[
B^{1+td}_q(L_q(\Omega)) \subset
B^{1+td \ell_j}_q(L_{q/\ell_j}(\Omega))
\quad\text{for all $0  < \ell_j \leq 1$},
\]
\rhn{we deduce} $| \bXi |_{B^{1+td \ell_j}_q(L_{q/\ell_j}(\Omega))}\Cleq \| \bXi \|_{B^{1+td}_q(L_{q}(\Omega))}$
\rhn{along with}
\[
 |\bXi |_\Omega  \Cleq  C_2(\gamma).
\]
Then, Proposition \ref{P:greedy_algo} guarantees that the call \rhn{$\T_\delta=\GREEDY( \bXi,\T,\delta)$} stops in a finite number of steps and the resulting subdivision \rhn{$\T_\delta$} satisfies
\[
\zeta_{\T_\delta}(\bXi,  T) \le \delta
\qquad\forall~T\in\T_\delta.
\]
This, in conjunction with Proposition \ref{P:oscU-bound}
and the finite overlapping property of the sets $N_\T(T)$,
implies that \rhn{$\T_\delta$} satisfies
\begin{equation*}\label{eq:oscU-bound}
\osc_{\T_\delta}(V)^2 \Cleq \sum_{ T\in\T_{\delta}} \zeta_{\T}(\bXi,N_{\T_{\delta}}(T))^2
\| \widehat{\nabla}  V \|_{L_2(N_{\T_\delta}(\widehat{T}))}^2
\le \delta^2  \|  \widehat \nabla  \bV \|_{L_2(\Omega)}^2,
\quad\forall~V\in \mathbb V(\T_\delta).
\end{equation*}
This proves the first assertion.
In order to bound the cardinality of $\M$ 
we rely on the estimate \eqref{eq:greedy-bounds} on the elements marked by $\GREEDY$
\begin{equation*}
  \# \M \Cleq
  |\bXi |_\Omega^{\frac{d q}{d + r q}}   \delta^{-\frac{d q}{d + r q}}.
\end{equation*}
The proof concludes upon realizing that
$\frac{d q}{d + r q}= \frac{1}{t}$.
\end{proof}
\begin{remark}[Approximation of $\gamma$]\label{R:approx-gamma}
\rm
\rhn{Since $\lambda_{\T}(\gamma,T)$ and $\zeta_{\T}(\bXi,T)$ satisfy
\[
\lambda_{\T}(\gamma,T) = \| \nabla(\chi - X_\T) \|_{L_\infty(\widehat T)}
\Cleq \hT^r |\chi|_{B_q^{1+td}(L_q(\widehat T))} \le \hT^r |\chi|_{\widehat T} 
= \zeta_{\T}(\bXi,T) ,
\]
we deduce {$\lambda_{\T_\delta}(\gamma)\Cleq\delta$} for the mesh $\T_\delta$ of Proposition \eqref{P:oscU}.
}
%Since the local error estimator in \eqref{zeta-oscV} satisfies
%\rhn{
%$\zeta_{\T_\delta}(\bXi,T) \Cgeq \lambda_{\T_\delta}(\gamma,T)$},
% we deduce \rhn{$\lambda_{\T_\delta}(\gamma)\Cleq\delta$} for the mesh $\T_\delta$ of Proposition \eqref{P:oscU}.
\end{remark}

%--------------------------------------------------------------------------------
\subsection{Membership in $\As$}\label{S:member-As}
%--------------------------------------------------------------------------------
We now collect the estimates derived earlier in this section and
prove our second main result. 

\medskip
\begin{theorem}[Membership in $\As$]\label{T:member-As}
\rhn{Let $\gamma$ be globally of class $W^1_\infty$ and be
parameterized by $\bXi \in [B^{1+td}_q(L_q(\Omega))]^M$
with $tq>1, 0<q\le\infty$ and $td\le n$, and let
\manel{$k := \lfloor td\rfloor +1$}. Let $u\in
H^1_\#(\gamma)$ and $f \in L_2(\gamma)$ be piecewise of class
$B^{1+sd}_p(L_p(\Omega))$ and $B^{sd}_p(L_p(\Omega))$, respectively, namely
$\bu \in [B^{1+sd}_p(L_p(\Omega))]^M$ and $\mathbf{f}\in [B^{sd}_p(L_p(\Omega))]^M$,
with $s-1/p+1/2>0, 0<p\le\infty$ and $0<sd\le n$.
Let $\T_0$ \manel{have} an admissible labeling} and $\lambda_{\T_0}(\gamma)$ satisfy \eqref{eq:init_cond}.
Then,  
$$
	(u,f,\gamma) \in \mathbb{A}_{s \wedge t},
$$
i.e, given $\delta>0$ there exists a conforming refinement $\T$
such that 
\begin{equation}\label{class-st}
\inf_{V \in \V(\T) }
E_{\T} (V; u,f,\gamma)  \Cleq \delta,
\qquad
\#\T-\#\T_0 \Cleq |u,f,\gamma|_{\mathbb{A}_{s\wedge t}}^{\frac{1}{s\wedge t}}
\delta^{-{\frac{1}{s\wedge t}}},
\end{equation}
and
\begin{equation}\label{ufg}
\begin{aligned}
|u,f,\gamma|_{\mathbb{A}_{s\wedge t}} \le&{}
|\bu|_{B^{1+sd}_p(L_p(\Omega))}
\\
& +
\Big(\|\bXi\|_{B^{1+td}_q(L_q(\Omega))} + \|\bXi\|_{B^{1+td}_q(L_q(\Omega))}^{k^k}  \Big)
\Big( 1 + \|{\mathbf f}\|_{B^{sd}_p(L_p(\Omega))} \Big).
\end{aligned}
\end{equation}

\end{theorem}
\vskip-0.3cm
\begin{proof}
Since $\lambda_{\T_0}(\gamma)$ satisfies \eqref{eq:init_cond}, instead of
dealing with $E_{\T} (V; u,f,\gamma)$, we argue with the equivalent
quantity $\widehat{E}_{\T}(V;u, f, \gamma)$ from
\eqref{total-error-3}, which is evaluated in $\Omega$.
Given $\delta > 0$, we invoke Corollaries \ref{C:approx-u} and \ref{C:approx-f} to obtain meshes
$\T_u, \T_f\in\grids$ satisfying
\begin{align*}
\inf_{V\in\V(\T_u)} \| \rhn{\widehat{\nabla}}(\bu - \bV)\|_{L^2(\Omega)} \Cleq \delta&, 
\qquad
\#\M_u \Cleq C(\bu)^{1/s} \, \delta^{-1/s},
\\
\osc_{\T_f}(f) \Cleq \delta&, 
\qquad
\#\T_f - \#\T_0 \Cleq C(\mathbf{f},\gamma)^{1/(s\wedge t+1/d)} \,
\delta^{-1/(s\wedge t+1/d)}.
\end{align*}
We next apply Proposition \ref{P:oscU} and Remark \ref{R:approx-gamma},
starting from $\T_u$, to obtain a refinement $\T_\gamma\in\grids$ such that $\T_\gamma \ge \T_u$,
\[
\lambda_{\T_\gamma}(\gamma) \Cleq \delta,
\qquad
\max_{V\in\V(\T_\gamma)} \frac{\osc_{\T_\gamma} (V)}{\|\wnabla \bV\|_{L_2(\Omega)}}
\Cleq \delta,
\qquad
\#\M_\gamma \Cleq C_2(\gamma)^{1/t} \, \delta^{-1/t}.
\]
The cardinality of $\T_\gamma$ can be estimated via Lemma \ref{L:REFINE:complexity}
(complexity of $\REFINE$)
\[
\#\T_\gamma - \#\T_0 \Cleq \# \M_u + \#\M_\gamma \Cleq C(\bu)^{1/s}
\delta^{-1/s} + C_2(\gamma)^{1/t} \delta^{-1/t}.
\]
Since the cardinalities can be assumed to be at least $1$, we can
replace the exponents
\rhn{$1/s$ and $1/t$ of $C(\bu)\delta^{-1}$ and $C_2(\gamma) \delta^{-1}$
by $1 / s \wedge t$.}
Let $\T = \T_\gamma \oplus \T_f$ be the overlay of the
two meshes $\T_\gamma$ and $\T_f$. According to \cite[Lemma 3.7]{CaKrNoSi:08} the cardinality of 
$\#\T - \#\T_0$ is bounded by
$\# \T_{\gamma} + \# \T_f - 2 \#\T_0$,
whence
\[
\#\T - \#\T_0 \Cleq |u,\gamma,f|_{\mathbb{A}_{s\wedge t}}^{1/s\wedge t}
\, \delta^{-1/s\wedge t},
\]
with the nonlinear quantity $|u,\gamma,f|_{\mathbb{A}_{s\wedge t}}$
satisfying \eqref{ufg}.

It remains to show the first estimate in \eqref{class-st}. We first
observe that
\[
\inf_{V\in\V(\T)} \| \rhn{\widehat{\nabla}}(\bu -  \bV)\|_{L^2(\Omega)}
\le
\inf_{V\in\V(\T_u)} \| \rhn{\widehat{\nabla}}(\bu -  \bV)\|_{L^2(\Omega)}
\]
because $\T\ge\T_u$. We choose $V_u\in\V(\T_u)$ to be the function
that realizes the minimum\rhn{, whence $\| \wnabla(\bu - \bV_u)\|_{L_2(\Omega)}\Cleq \delta$.}
Since the definition \eqref{d:osc-U} of $\osc_{\cT}(V)$ involves the
best $L_2$-approximation,
we can argue as in Lemma \ref{L:est-reduction} to deduce for
$\T\ge\T_\gamma$ \rhn{(see Lemma~\ref{L:osc-reduction})}
\[
\osc_{\T}(V_u) \Cleq \osc_{\T_\gamma}(V_u) + \lambda_{\T_\gamma}(\gamma)
\le \delta \|\wnabla \bV_u\|_{L_2(\Omega)} + \delta,
\]
because $ V_u\in\V(\wT_u)\subset\V(\wT_\gamma)$. Upon adding and
subtracting $u$, we readily see that
\[
\osc_{\T}(V_u) \Cleq \delta \|\wnabla\bu\|_{L_2(\Omega)} +
\delta \|\wnabla(\bu- \bV_u)\|_{L_2(\Omega)} + \delta
\Cleq \delta.
\]
Since the definition \eqref{d:osc-f} utilizes the $L_2$-projection and
$\T\ge\T_f$ we infer that
\[
\osc_{\cT}(f) \le \osc_{\T_f}(f) \Cleq \delta.
\]
Collecting the preceding estimates and using the definition
\eqref{total-error-3} of total error gives the
desired estimate $\rhn{\widehat{E}}_{\T} (V_u; u,f,\gamma) \Cleq \delta$, and
finishes the proof.
\end{proof}

%%%%%%%%%%%%%%%%%%%%%%%%%%%%%%%%%%%%%%%%%%%%%%%%%%%%%%%%%%%%%%%%%%%%%%%%%%%%%%%%%
\section{Convergence rates}\label{S:rates}
%%%%%%%%%%%%%%%%%%%%%%%%%%%%%%%%%%%%%%%%%%%%%%%%%%%%%%%%%%%%%%%%%%%%%%%%%%%%%%%%%

In this section we study the cardinality of AFEM, which is dictated by
the regularity of $u,f$ and $\gamma$.
We now prove that AFEM achieves the asymptotic decay rate $s$ dictated by the class $\As$. We establish the link between the
performance of AFEM and the best possible error by adapting a clever
idea of Stevenson \cite{Stevenson:07} for the Laplace operator, further
extended by Casc\'on et al \cite{CaKrNoSi:08} to general elliptic PDE, in flat
domains. We refer to the survey \cite{NoSiVe:09} for a thorough
discussion and to \cite{GSTER14}. The insight \rhn{of \cite{Stevenson:07}} is the following
\begin{equation}\label{stevenson}
\begin{minipage}{0.9\linewidth}
\emph{
Any marking strategy that reduces the total error \rhn{to a fraction of} its current
value must contain a substantial portion of the error estimator, and so it can
be related to D\"orfler Marking.
}
\end{minipage}
\end{equation}
Exploiting next the minimality of D\"orfler marking we can
compare meshes generated by AFEM with the best meshes within $\grids$.
The approach of \cite{CaKrNoSi:08,NoSiVe:09,Stevenson:07} does not
apply directly to the present context because of the consistency error
due to surface interpolation. We account for this discrepancy below
upon making the parameter $\omega$ of $\ADAPTSURF$ sufficiently small.
Let
\begin{equation}\label{omega34}
\gls{const:omega3}\omega_3 := \frac{C_5}{\Lambda_0\sqrt{3 \Lambda_1 + 4 \Lambda_2 + 2 \Lambda_1 \Lambda_3}},
%\qquad
%\omega_4 := \frac{C_5}{2\Lambda_0}\sqrt{\Big( 1 - \frac{\theta^2}{\theta_*^2} \Big) 
%				    \frac{1}{\Lambda_2}}
\end{equation}
be a threshold for $\omega$ to be used next and let \gls{const:thetas}$\theta_*$ be a threshold
for D\"orfler parameter $\theta$
\begin{equation}\label{theta*}
\theta_* := \frac{C_5}{\sqrt{2C_3 + C_1(3 + 2 \Lambda_3)}};
\end{equation}
since $C_5=\sqrt{C_2/2}$ and $C_2\le C_1$, we see that $\theta_*<1$.
\begin{lemma}[D\"orfler marking]\label{L:dorfler-marking}
Let $\lambda_{\T_0}(\gamma)$ satisfy \eqref{eq:init_cond}, 
and the parameters $\theta$ and $\omega$ satisfy
\begin{equation}\label{eq:theta-omega-def}
0<\theta<\theta_*,\qquad
0<\omega \le \min\{\omega_1,\omega_3\},
\end{equation}
where $\theta_*,\omega_3$ are defined in
\eqref{omega34}, \eqref{theta*}, and $\omega_1$ in \eqref{bound_omega}.
Let $\mu := \frac12  \sqrt{1 - \frac{\theta^2}{\theta_*^2}}$
and $(\Gamma,\cT, U)$ 
be the approximate surface, mesh and discrete solution produced by an inner
iterate of $\ADAPTPDE$. 
If $(\Gamma_*, \cT_*,U_*)$ is a surface-mesh-solution triple with 
$\cT_* \ge \cT$, such that the PDE error satisfies
\begin{equation}\label{mu-condition}
\E_{\T_*}(U_*,f) \le \mu\, \E_\T(U,f),
\end{equation}
then the refined set $\cR \definedas \ab{\T \setminus \T_*}$ satisfies D\"orfler 
property with parameter $\theta$, namely
\begin{equation}\label{dorfler-cond}
\eta_\T(U,F_\Gamma,\cR) \ge \theta \eta_\T(U,F_\Gamma).
\end{equation}
\end{lemma}
\begin{proof}
We proceed as in \cite[Lemma 5.9]{CaKrNoSi:08} using the notation
$e(U):=\|\nabla_\gamma(u-U)\|_{L_2(\gamma)}$. 
Since $\omega \leq \omega_1$, we combine the lower bound of
\eqref{equiv} with \eqref{mu-condition} to write
\begin{align*}
(1- 2 \mu^2) C_5^2 \eta_\T (U,F_\Gamma)^2 &\le
(1- 2 \mu^2) \big(e(U)^2 + \osc_{\cT}(U,f)^2 \big) 
\\
& \le
e(U)^2 - e(U_*)^2 + \osc_{\T}(U, f)^2 - 2\osc_{\T_*}(U_*,f)^2.
\end{align*}
We now estimate separately error and oscillation terms. According to 
\eqref{quasi-ortho} and \eqref{localized}, we obtain
\begin{equation*}
\begin{split}
e(U)^2 - e(U_*)^2 &\le \frac32 \|\nabla_\gamma(U_*-U)\|_{L_2(\gamma)}^2 
+ \Lambda_2 \lambda_{\cT}(\gamma)^2 
\\
& \le \frac32 C_1\eta_\cT(U,F_\Gamma,\cR)^2 + \Big(\frac{3}{2} \Lambda_1 +
\Lambda_2\Big) \lambda_{\cT}(\gamma)^2.
\end{split}
\end{equation*}
For the oscillation terms we argue according to whether an element 
$T \in \T$ belongs to the set of refined elements $\cR$ or not. 
We use the dominance bound \eqref{dominance} to arrive at
\[
 \osc_{\T} (U,f,\cR)^2 \le C_3\eta_\cT(U,F_\Gamma,\cR)^2.
\]
On the other hand, using \eqref{osc-pert} for $\T_*\cap \T$ 
with $V = U$ and $W = U_*$ yields
\[
\osc_{\cT}(U,f, \T_*\cap \T)^2 - 2
\osc_{\T_*}(U_*,f)^2
\leq  \Lambda_3 \|\nabla_\gamma(U_*-U)\|_{L_2(\gamma)}^2
+ \Lambda_2 \lambda_{\cT}(\gamma)^2. 
\]
By combining these two estimates with \eqref{localized} 
we infer that 
\[
\osc_{\cT}(U,f)^2 - 2 \osc_{\T_*}(U_*,f)^2 
   \le 
(C_3+ C_1\Lambda_3) \eta_\T(U,F_\Gamma,\cR)^2 
+ ( \Lambda_1 \Lambda_3  + \Lambda_2) \lambda_{\cT}(\gamma)^2. 
\]
Since $\cT$ is produced within
$\ADAPTPDE$, \rhn{which is initialized with $\T^+$ in \AFEM of \S~5}, we have $\eta_\T(U,F_\Gamma) \ge \eps$ , $\lambda_{\T^+}(\gamma) \le\omega\eps$,  \rhn{and $\T \geq \T^+$}, whence
$$
\lambda_{\cT}(\gamma) \leq \Lambda_0 \lambda_{\T^+}(\gamma)\leq
\Lambda_0 \omega \eps \leq \Lambda_0 \omega\eta_\T(U,F_\Gamma).
$$ 
Collecting these three estimates, we deduce
\begin{equation*}
\begin{aligned}
  (1-2\mu^2) C_5^2 \eta_\T(U,F_\Gamma)^2 &\le
  \frac12 \Big(2C_3+ C_1(3 +2 \Lambda_3) \Big) \eta_\T(U,F_\Gamma,\cR)^2
  \\
  & + \frac12 \big( 3 \Lambda_1+ 4\Lambda_2 + 2 \Lambda_1 \Lambda_3 \big)
  \Lambda_0^2\omega^2 \eta_\T(U,F_\Gamma)^2.
\end{aligned}
\end{equation*}
Finally, using that $\omega\le\omega_3$, along with \eqref{omega34} and
\eqref{theta*}, we infer that
\[
(1-4\mu^2) \rhn{\theta_*^2 }\eta_\T(U,F_\Gamma)^2 \le \eta_\T(U,F_\Gamma,\cR)^2.
\]
The choice of $\mu$ implies the asserted estimate \eqref{dorfler-cond}.
\end{proof}

\begin{lemma}[Cardinality of $\cM$]\label{L:cardinality}
Let $\lambda_{\T_0}(\gamma)$ satisfy \eqref{eq:init_cond} and 
the procedure $\MARK$ select a set $\cM$ with minimal cardinality. Let  
the parameters $\theta$ and $\omega$ satisfy
\begin{equation}\label{theta-omega4}
0<\theta<\theta_*,\qquad
0<\omega\le \min\{\omega_1, \omega_3 \}
\end{equation}
with $\theta_*,\omega_1$ and $\omega_3$  given in \eqref{theta*}, \eqref{bound_omega}, and \eqref{omega34},
 respectively. Let $u$ be the solution  
of \eqref{p:Weak_PdeGm}, and let $(\Gamma,\cT, U)$ be produced within $\ADAPTPDE$.
If $(u,f,\gamma)\in\As$, then
\[
\#\cM \Cleq |u,f,\gamma|_{\As}^{\frac{1}{s}} \, \E_\T(U,f)^{-\frac{1}{s}}.
\]
\end{lemma}
\begin{proof}
We set 
\[
\delta^2 = \hat\mu ^2 \, \E_\T(U,f)^2 =
\hat\mu^2 \Big(e(U)^2 + \osc_{\cT}(U,f)^2 \Big),
\]
for $0<\hat\mu<\mu = \frac12\sqrt{1 - \frac{\theta^2}{\theta_*^2}}<1$
sufficiently small to be determined later. 
Since $(u,f,\gamma)\in\As$, there exists a subdivision $\cT_\delta \in \mathbb T$
and 
$V_\delta\in\V(\T_\delta)$ such that
\begin{equation}\label{u-As}
\#\cT_\delta - \#\cT_0 \Cleq |u,f,\gamma|_{\As}^{\frac{1}{s}} \delta^{-\frac{1}{s}},
\qquad
e(V_\delta)^2 + \osc_{\T_\delta}(V_\delta, f)^2 
+ \lambda_{\T_\delta}(\gamma) ^2 \le \delta^2.
\end{equation}
Let $\cT_* = \cT \oplus \cT_\delta$ be the overlay of $\cT$ and $\cT_\delta$, which
satisfies \cite[Lemma 3.7]{CaKrNoSi:08}, \cite{NoSiVe:09},
\begin{equation}\label{overlay}
\#\cT_* \le \#\cT + \#\cT_\delta - \#\cT_0.
\end{equation}
Let $U_*\in \V(\T_*)$ be the corresponding Galerkin solution.  
We observe that $\cT_*\ge\cT_\delta, \cT$, and invoke the upper bound
of \eqref{quasi-ortho} in conjunction with \eqref{osc-mono} to write
\begin{equation*}
\begin{split}
e(U_*)^2 + \osc_{\T_*}(U_*, f)^2 \le{}&{}
  e(V_\delta)^2 + \Lambda_2\lambda_{\T_\delta}(\gamma)^2 
\\
& + C_6 \osc_{\T_{\delta}}(V_\delta,f)^2
+ \Lambda_3 \|\nabla_\gamma(U_* - V_\delta) \|^2_{L_2(\gamma)}
+ \Lambda_2 \lambda_{\T_\delta}(\gamma)^2.
\end{split}
\end{equation*}
Applying \eqref{quasi-ortho} again gives
$\|\nabla_\gamma(U_* - V_\delta) \|^2_{L_2(\gamma)} \le \rhn{2} e(V_\delta)^2
+ \rhn{2} \Lambda_2 \lambda_{\T_{\delta}}(\gamma)^2$, whence
\begin{equation*}
e(U_*)^2 + \osc_{\T_*}(U_*, f)^2 
\leq (1 + 2\Lambda_3)e(V_\delta)^2  
+ C_6 \osc_{\T_{\delta}}(V_\delta,f)^2
+ 2\Lambda_2(1+\Lambda_3) \lambda_{\T_{\delta}}(\gamma)^2.
\end{equation*}

We  now choose 
$\hat\mu  = \frac{\mu}{ \sqrt{\max\{C_6, 1+2\Lambda_3, 2\Lambda_2(1+\Lambda_3)\}}}$ to end up with
\[
e(U_*)^2 + \osc_{\T_*}(U_*,f)^2
\leq \max\{C_6, 1+2\Lambda_3, 2\Lambda_2(1+\Lambda_3)\} \delta^2
= \mu^2 \Big( e(U)^2 + \osc_{\cT}(U,f)^2 \Big).
\]
We thus deduce from Lemma \ref{L:dorfler-marking} that the 
subset $\cR = \cR_{\T\to\T_*} \subset \T$ satisfies D\"orfler 
property \eqref{dorfler-cond}. Since the set $\cM \subset \cT$ 
also satisfies this property, but with minimal cardinality, we infer from~\eqref{u-As}--\eqref{overlay}
\[
\#\cM \le \#\cR \le \#\cT_* - \#\cT \le \#\cT_\delta-\#\cT_0 \Cleq
|u,f,\gamma|_{\As}^{\frac{1}{s}} \delta^{-\frac{1}{s}},
\]
The asserted estimate finally follows upon using the definition of $\delta$.
\end{proof}

The quasi-optimal cardinality of AFEM is a direct consequence of Lemma
\ref{L:cardinality} and Theorem \ref{T:conditional}.
This is our third main result and we prove it next.

\begin{theorem}[Convergence rate of $\AFEM$]\label{T:rate}
%Let $\gamma \in \mathbb B_{t}$ 
Let $\eps_0\le (6\omega\Lambda_0 L^3)^{-1}$ 
be the initial tolerance, and the parameters $\theta, \omega, \rho$ of
AFEM satisfy
\begin{equation}\label{restrict-omega}
0<\theta \leq \theta_*,
\qquad
0< \omega \leq \omega_* \definedas \min \{\omega_1, \omega_2,\omega_3\},
\qquad
\rhn{0<\rho<1,}
\end{equation}
where $\theta_*,\omega_1,\omega_2,\omega_3$ are given in \eqref{theta*},
\eqref{bound_omega}, \eqref{bound_omega_2}, and \eqref{omega34}, respectively.
\rhn{Let $\T_0$ \manel{have} an admissible labeling,} and let the procedure $\MARK$ select sets with minimal cardinality.
%and $\ADAPTSURF$ be $s$-optimal on the surface $\gamma$. 
Let $u$ be the solution of \eqref{p:Weak_PdeGm}
and $\{\Gamma_k, \cT_k, U_k\}_{k \geq 0}$  be \pedro{the} sequence of approximate 
surfaces, meshes and discrete solutions generated by AFEM.

\rhn{If $(u,f,\gamma)\in \mathbb A_s$ 
for some $0<s\leq n/d$,} then there exists a constant $C$, depending on the Lipschitz constant
$L$ of $\gamma$, $\|f\|_{L_2(\gamma)}$, the refinement depth $b$, the
initial triangulation 
$\T_0$, and AFEM parameters $\theta,\omega,\rho$
such that
\begin{equation}\label{rate}
e(U_k) + \osc_{\T_k}(U_k,f) + \omega^{-1}\lambda_{\T_k}(\gamma) \leq C
|u,f,\gamma|_{\mathbb A_s} 
\big(\#\cT_k-\#\cT_0 \big)^{-s},
\end{equation}
where $|u,f,\gamma|_{\mathbb A_s}$ is defined in \eqref{approx_class}.
\end{theorem}
\begin{proof}
We start by noting that since $\omega\eps_0 \leq \frac{1}{6 \Lambda_0 L^3}$ 
the first output of the procedure $\ADAPTSURF$ fulfills 
$\rhn{\lambda_{\T_0^+}(\gamma)} \le \frac{1}{6 \Lambda_0L^3}$ which is
  \eqref{eq:init_cond} and implies that $\grids(\T_0^+)$ is shape regular.

There are two instances where elements are added, inside $\ADAPTSURF$
and $\ADAPTPDE$.
\rhn{In light of~\eqref{optimal:AS} and~\eqref{class-As}
we observe that \ADAPTSURF is $s$-optimal with
$ C(\gamma) \Cleq |u,f,\gamma|_{\mathbb A_s}^{1/s}$, whence}
the set of all the elements marked for refinement in the $k$-th call to
$\ADAPTSURF$  satisfies
\begin{equation*}\label{eq:card-M+}
\#\cM_k \Cleq C(\gamma) \; \omega^{-\frac{1}{s}} 
\; \eps_k^{-\frac{1}{s}}
\Cleq |u,f,\gamma|^{\frac{1}{s}}_{\mathbb{A}_s} \; \eps_k^{-\frac{1}{s}} .
\end{equation*}
For $\ADAPTPDE$, Lemma \ref{L:cardinality} (cardinality of $\M$) yields
\[
\#\cM_k^j \Cleq |u,f,\gamma|_{\mathbb A_s}^{\frac{1}{s}} \Big(e(U_k^j) +
\osc_{\T_k^j}(U_k^j, f) \Big)^{-\frac{1}{s}}
\qquad
0\le j < J.
\]
where $\cM_k^j$ denotes the subset of elements selected by
the marking procedure at the j-th subiteration of the $k$-th step of $\ADAPTPDE$.
Since the inner iterates of $\ADAPTPDE$ satisfy 
Theorem \ref{T:conditional} (conditional contraction property) and
\[
e(U_k^j) + \osc_{\T_k^j}(U_k^j,f) \approx e(U_k^j) + \eta_{\T_k^j}(U_k^j,F_k^j),
\]
according to~\eqref{equiv}, we deduce that
\begin{align*}
\Big(e(U_k^j) + \osc_{\T_k^j}(U_k^j,f)\Big)^{-\frac{1}{s}} &\Cleq
\alpha^{\frac{J-j-1}{s}} \Big(e(U_k^{J-1}) +
\eta_{\T_k^{J-1}} (U_k^{J-1}, F_k^{J-1}) \Big)^{-\frac{1}{s}} 
\\
&\le \alpha^{\frac{J-j-1}{s}} \eps_k^{-\frac{1}{s}}
\end{align*}
\rhn{by virtue of $\eta_{\T_k^{J-1}} (U_k^{J-1}, F_k^{J-1})>\eps_k$.} This implies
\begin{equation*}\label{eq:card-Mj}
\sum_{j=0}^{J-1} \#\cM_k^j \Cleq |u,f,\gamma|_{\mathbb A_s}^{\frac{1}{s}} \eps_k^{-\frac1s}
\sum_{j=0}^{J-1} \alpha^{\frac{J-j-1}{s}}
\Cleq |u,f,\gamma|_{\mathbb A_s}^{\frac{1}{s}} \eps_k^{-\frac1s}.
\end{equation*}

To do a full counting argument, we resort to the crucial estimate  
\eqref{REFINE:complexity} which, combined with the estimates above, 
gives 
\begin{equation*}
\begin{split}
\#\cT_k - \#\cT_0 \le C_7 \sum_{i=0}^{k-1}
\Big(\# \cM_i + \sum_{j=0}^{J-1} \#\cM_i^j \Big)
\Cleq 
\rhn{|u,f,\gamma|_{\mathbb A_s}^{\frac{1}{s}} \sum_{i=0}^{k-1} \eps_i^{-\frac{1}{s}}.}
\end{split}
\end{equation*}
We now use the relation $\eps_{k+1}=\rho\eps_k$ of step $3$ of AFEM,
\rhn{together with $\rho <1$,} to obtain
$\sum_{i=0}^{k-1} \eps_i^{-\frac{1}{s}} = \eps_{k-1}^{-\frac{1}{s}} 
\sum_{i=0}^{k-1} \rho^{\frac{i}{s}} \Cleq \eps_k^{-\frac{1}{s}}$,
whence
\begin{equation}\label{dof-eps-k}
\#\cT_k - \#\cT_0
\Cleq |u,f,\gamma|_{\mathbb A_s}^{\frac{1}{s}} \, \eps_k^{-\frac{1}{s}}.
\end{equation}
Moreover, the stopping criteria
\eqref{e:target_surf} and \eqref{afem-enter}  guarantee that
\begin{equation}\label{e:estim_no_T}
e(U_k) + \osc_{\T_k}(U_k,f) + \omega^{-1}\lambda_{\cT_k}(\gamma) \Cleq \eps_{k},
\end{equation}
which implies the desired estimate \eqref{rate}.
\end{proof}

The precise constant on the right-hand side of \eqref{rate} is
$\omega^{-1} C(\gamma)^s + |u,f,\gamma|_{\mathbb A_s}$.
This and the condition $\omega\le\omega_*$ in
\eqref{restrict-omega} suggest that $\omega$ should not 
be too small to optimize \eqref{rate}. An
optimal choice of $\omega$,  %for the case $s=t$, 
which unfortunately is not computable, appears to be
\[
\omega = \min\big\{\omega_*, |u,f,\gamma|_{\As}^{-1} C(\gamma)^{s} \big\}.
\]

We also provide an estimate on the workload in the following corollary.
We assume that  the adaptive loop~\eqref{p:loop} on a subdivision $\T
\in \mathbb T$ requires $O(\#\T)$ computations and in particular (i)
the linear algebra solver scales like $\#\T$~\cite{BP:11,KY:08} and (ii) an approximate
sort requiring $O(\#\T)$ arithmetic operations is used to select the
local estimators $\eta_\T(U,F_\Gamma,T)$
for all $T\in \T$ (see e.g. \cite[Remark~5.3]{BN:10}).
\begin{corollary}[Workload estimate]
In addition to the assumptions of Theorem~\ref{T:rate}, suppose that
each inner loop of $\ADAPTPDE$ on a subdivision $\T \in \mathbb T$
requires $O(\#\T)$ arithmetic operations. 
If $\varepsilon \leq \varepsilon_0$, then the number of arithmetic
operations $W$ for \AFEM to construct a triple $(\Gamma,\T,U)$ such that
\begin{equation}\label{e:final_tol}
e(U) + \osc_{\cT}(U,f) + \omega^{-1}\lambda_{\cT}(\gamma) \leq \varepsilon
\end{equation}
satisfies
$$
W \Cleq \varepsilon^{-1/s}.
$$
\end{corollary}
\begin{proof}
Let $C \geq 1$ be the hidden constant in~\eqref{e:estim_no_T} and set
\ab{$K$ to be the integer such that $\varepsilon_{K+1}:=\rho^{K+1}
\varepsilon_0 \leq \varepsilon/C \leq \varepsilon_{K}$}.
Moreover, we define $W_j^+$ to be the number of
arithmetic operations performed within the call
\rhn{$\T_j^+=\ADAPTSURF(\T_j,\omega \varepsilon_j)$}
 and 
$W_{j+1}$ those within the call
$[U_{j+1},\T_{j+1}]=\ADAPTPDE(\T_j^+,\varepsilon_j)$.
With these notations, the total number of operations to
achieve~\eqref{e:final_tol} satisfies
$$
W \Cleq \sum_{j=0}^{\ab{K}} (W_j^+ +W_{j+1}).
$$
We now bound each term separately starting with $W_j^+$.
The computation of each local geometric estimator requires $O(1)$ arithmetic operations and is performed 
$$
\#\T_j + \#\mathcal M_j \leq \#\T_{j}^+ \leq \#\T_{j+1}
$$
times.
Since $\ADAPTSURF$ does not involve sorting the local geometric estimators, we readily deduce that
$$
W_j^+ \Cleq \#\T_{j+1}.
$$
Regarding $W_{j+1}$, we recall that Theorem~\ref{T:conditional} guarantees that the number of inner iterations $J$ within $\ADAPTPDE$ is uniformly bounded.
This, together with the complexity assumption on the
inner loops of $\ADAPTPDE$, yields
$$
W_{j+1} \Cleq \# \T_{j+1}.
$$
Now combining the above two estimates and invoking~\eqref{dof-eps-k}, we deduce that
$$
W_j^+ + W_{j+1} \Cleq \varepsilon_{j+1}^{-1/s}.
$$
Going back  to the total number of operations $W$, we find 
$$
W \Cleq \rho^{-1/s} \varepsilon_{\ab{K}}^{-1/s} \sum_{j=0}^{\ab{K}} \rho^{j/s} \Cleq \varepsilon_{\ab{K}}^{-1/s},
$$
where we used the relations \ab{$\varepsilon_{j+1} = \rho^{-(K-j-1)} \varepsilon_K$ for $j=0,...,K$}.
The desired estimate follows from the definition of $\varepsilon_{\ab{K}}$.
\end{proof}

%%%%%%%%%%%%%%%%%%%%%%%%%%%%%%%%%%%%%%%%%%%%%%%%%%%%%%%%%%%%%%%%%%%%%%%%%%%%%%%
\section{Products and Compositions in Besov Spaces}\label{S:Besov}
%%%%%%%%%%%%%%%%%%%%%%%%%%%%%%%%%%%%%%%%%%%%%%%%%%%%%%%%%%%%%%%%%%%%%%%%%%%%%%%

\rhn{In this section we derive scale-invariant estimates for Besov
(quasi-)semi-norms for products and compositions of functions. This extends estimates for Sobolev norms of products~\cite{Adams-Fournier:2003}, as alluded to in Remark~\ref{R:scale-invariant}, and for Besov norms of compositions~\cite{Bourdaud-Sickel:2011}. Our estimates, however, possess a structure that does not seem to be available in the literature.}

We recall the definition of Besov spaces via modulus of smoothness; a thorough discussion can be found in~\cite{GM:14}.
Let $\Omega$ be a Lipschitz domain in $\R^d$, $0<p\le \infty$ and $u \in L_p(\Omega)$, we define the differences as follow, for $h\in\R^d$, $k \in \N$:
$\Delta_h u : \Omega \to \R$, with 
\[
\Delta_h u (x) = 
\begin{cases} u(x+h) - u(x),  \qquad &\text{if } 
x \in \Omega_h = \{ x \in \Omega : x + th \in \Omega, \forall t \in [0,1]\}, \\
0, \qquad&\text{otherwise,}
\end{cases}
\]
and $\Delta_h^{k+1}u : \Omega \to \R$ as $\Delta_h^{k+1} u (x) = \Delta_h \Delta_h^{k} u(x)$ for $k \in \N$ and $x \in \Omega_{(k+1)h}$ and 0 otherwise.
Therefore,
% $\Delta $\Delta^k_h u:\Omega \to \R$ given by
\begin{equation}\label{Delta^k}
 \Delta^{k}_{h} u(x) =
 \begin{cases}  
 \sum_{j=0}^{k} (-1)^{k+j} \binom{k}{j} u(x+jh) , 
\qquad &\text{if } x \in \Omega_{kh}, \\
0, \qquad&\text{otherwise}.
\end{cases}
\end{equation}
Using these difference operators we define the \emph{modulus of smoothness} of order $k$ in $L_p(\Omega)$ as:
\[
\omega_k (u,t)_p = \sup_{|h| \leq t} \|\Delta^{k}_{h} u\|_{L_p(\Omega)} ,\qquad t > 0,
\]

Given $s > 0$ and $0<q,p\leq \infty$, the Besov space $B^s_{q}(L_p(\Omega))$, is the set of all functions $f\in L_p(\Omega)$ such that the semi-(quasi)norm \gls{def:besov}$|f|_{B^s_{q}(L_p(\Omega))}$ is finite, with
\begin{equation}\label{besov-seminorm}
|f|_{B^s_{q}(L_p(\Omega))} :=
\left \{
\begin{aligned}
 \displaystyle &\left( \int_{0}^{\infty} [t^{-s} \omega_k(u,t)_p]^q \frac{dt}{t} \right)^{\frac{1}{q}}, \qquad & &\textrm{ if } 0<q<\infty 
     \\
 \displaystyle &\sup_{t>0} t^{-s} \omega_k(f,t)_p , \qquad & &\textrm{ if } q=\infty,
\end{aligned}
  \right.
\end{equation}
where $k \in \N$ is such that 
$s < k$. % - 1 + \max\{1, \frac{1}{p}\}$. 
\rhn{Although this definition only requires \rhn{$k>s$,
it turns out that the space $B_q^s(L_p(\Omega))$ is independent of $k$~\cite[Theorem 10.1]{DL}.} % 
From now on, we fix the value of $ k := \lfloor s \rfloor +1$ in most of our results for simplicity. 
%We assert that
%all results  obtained  are also valid for any $k\in \mathbb{N}$ such that 
%$s < k$. % - 1 + \max\{1, \frac{1}{p}\}$, because their proofs only require $k>s$.
}
The (quasi) norm of $B^s_{q}(L_p(\Omega))$ is defined by:
\begin{equation}\label{norma-besov}
 \|f\|_{B^s_{q}(L_p(\Omega))} = \|f\|_{L_p(\Omega)} + |f|_{B^s_{q}(L_p(\Omega))}.
\end{equation}

\rhn{Two important properties that we will exploit in what follows are
  the embeddings of $B^s_{q}(L_p(\Omega))$ in $L_\infty(\Omega)$,
  whenever $sp > d$, and of $B^{s_1}_{q_1}(L_{p_1}(\Omega))$ in
  $B^{s_2}_{q_2}(L_{p_2}(\Omega))$, whenever $s_1 - d/p_1 > s_2 -
  d/p_2$~\cite{GM:14}.}
  
%---------------------------------------------------------------------------------
\subsection{Product of Functions}
%---------------------------------------------------------------------------------

The following result, essential for our discussion, is analogous to~\cite[Theorem~4.39]{Adams-Fournier:2003} for Besov spaces
and is scale-invariant.

\begin{lemma}[Scale-invariant Besov semi-norm of the product of two functions]\label{L:Besov-product}
\rhn{Let $u,v \in B_q^s(L_p(\Omega))$ with }
  $s>0$ and $0 < p ,q\le \infty$ satisfying $ s > d/p$ (i.e. 
  $B^s_q(L_p(\Omega)) \subset L_\infty(\Omega)$)
  and $k = \lfloor s \rfloor + 1$. Then \rhn{$uv \in B_q^s(L_p(\Omega))$ and}
  \begin{equation}\label{product}
  | u v |_{B^s_q(L_p(\Omega))}
  \Cleq
  \sum_{j=0}^{k}
    | u |_{B^{sj/k}_q(L_{pk/j}(\Omega))} 
    | v |_{B^{s(k-j)/k}_q(L_{pk/(k-j)}(\Omega))},
  \end{equation}
  with the convention that $B^0_q(L_{pk/0}(\Omega)) = L_\infty(\Omega)$ and $|\cdot|_{B^0_q(L_{pk/0}(\Omega))} = \|\cdot\|_{L_\infty(\Omega)}$.
\end{lemma}
\begin{proof}
Recall that 
$\omega_k(uv,t)_p = \sup_{|h|\le t} \|\Delta^k_h (uv)
  \|_{L_p(\Omega)}$ and that the $k$-th differences obey a rule similar to Leibniz
rule. This translates into the expression 
\begin{align*}
\|\Delta^k_h (uv) \|_{L_p(\Omega)} \Cleq{}& \sum_{j=0}^k 
\|\Delta^j_h u \Delta^{k-j}_h v \|_{L_p(\Omega)} 
\le  \sum_{j=0}^{k} \|\Delta^j_h u \|_{L_{pk/j}(\Omega)} \|\Delta^{k-j}_h v \|_{L_{pk/(k-j)}(\Omega)}, 
%\Cleq{}& \|u \|_{L_\infty(\Omega)} \|\Delta^k_h v \|_{L_p(\Omega)} 
%+ \sum_{j=1}^{k-1} \|\Delta^j_h u \|_{L_{pk/j}(\Omega)} \|\Delta^{k-j}_h v \|_{L_{pk/(k-j)}(\Omega)} \\
%&+\|\Delta^k_h u \|_{L_p(\Omega)} \| v \|_{L_\infty(\Omega)} ,
\end{align*}
where we have used H\"older inequality in the last step and the conventions $\Delta_h^0=Id$ and $L_{pk/0}(\Omega) = L_\infty(\Omega)$. Therefore,
\[
\omega_k(uv,t)_p \Cleq \sum_{j=0}^k 
\omega_j(u,t)_{pk/j} \, \omega_{k-j}(v,t)_{pk/(k-j)},
\]
where we use the convention
$\omega_0(v,t)_{pk/0}=\|v\|_{L_\infty(\Omega)}$.

We now consider the cases $q=\infty$ and $q<\infty$ separately. Observe that
\begin{align*}
|uv|_{B^s_\infty(L_p(\Omega))} &= \sup_{t>0} t^{-s} \omega_k(uv,t)_p
\\
& \Cleq \sum_{j=0}^k \sup_{t>0} t^{-js/k} \omega_j(u,t)_{pk/j} 
\sup_{t>0} t^{-(k-j)s/k}\omega_{k-j}(v,t)_{pk/(k-j)}.
\end{align*}
Utilizing that 
$|u|_{B^{js/k}_\infty(L_{pk/j}(\Omega))} \simeq \sup_{t>0} t^{-js/k} \omega_j(u,t)_{pk/j}$,
because $k>s$ so that $js/k < j \le j-1+\max\{1,j/(pk)\}$ for $0\le j \le k$, we have
\begin{equation*}\label{Leibniz-q=infty}
|uv|_{B^s_\infty(L_p(\Omega))} \Cleq \sum_{j=0}^k |u|_{B^{js/k}_\infty(L_{pk/j}(\Omega))}
|v|_{B^{(k-j)s/k}_\infty(L_{pk/(k-j)}(\Omega))}.
\end{equation*}

If $0<q<\infty$, we define $q^* = \max\{1,q\}$ and notice that by the triangle and H\"older inequalities
\begin{align*}
|uv|_{B^s_q(L_p(\Omega))}^{q/q^*} &= \left( \int_0^\infty t^{-sq} \omega_k(uv,t)_p^q \frac{dt}t \right)^{1/q^*}
\\
\le{}& \sum_{j=0}^k \left( \int_0^\infty t^{-sq} \omega_j(u,t)_{pk/j}^q \omega_{k-j}(v,t)_{pk/(k-j)}^q  \frac{dt}t  \right)^{1/q^*}
\\
\le{}& \sum_{j=0}^k \left( \int_0^\infty t^{-sq} \omega_j\left(v,t\right)_{pk/j}^{qk/j} \frac{dt}t \right)^{\frac{j}{kq^*}} 
                  \left( \int_0^\infty t^{-sq} \omega_{k-j}\left(v,t\right)_{pk/(k-j)}^{qk/(k-j)}  \frac{dt}t  \right)^{\frac{k-j}{kq^*}}
\\
\Cleq{}& \sum_{j=0}^k \left| u \right|_{B_{qk/j}^{sj/k}(L_{pk/j}(\Omega))}^{q/q^*}
                       \left| v \right|_{B_{qk/(k-j)}^{s(k-j)/k}(L_{pk/(k-j)}(\Omega))}^{q/q^*}  .
\end{align*}
Upon raising both sides to the power $q^*/q \ge 1$, this
  inequality implies \eqref{product}.
\end{proof}

We make the important observation that \eqref{product} is
scale-invariant: simply scale $\Omega$ by a constant $h$ and realize
that both sides of \eqref{product} scale by the same factor $h^{s-d/p}$. This implies
that  \eqref{product} can be used at the element level.
Upon iterating \eqref{product} we obtain the following simple, but
technical, generalization of \eqref{product}.

\begin{corollary}[Scale-invariant Besov semi-norm of products of functions]\label{C:Besov-product-seminorm}
  \rhn{Let $s>0$ and $0 < p,q \le \infty$ satisfy $ s > d/p$ (i.e. 
  $B^s_q(L_p(\Omega)) \subset L_\infty(\Omega)$), and let
  $k = \lfloor s \rfloor + 1$. If $u_i\in B^s_q(L_p(\Omega))$ for $1\le i\le m$,
  then $\prod_{i=1}^m u_i \in B^s_q(L_p(\Omega))$ and
\begin{equation}\label{product-semi}
\Big | \prod_{i=1}^m u_i \Big |_{B_q^s(L_p(\Omega))} 
\Cleq \sum_{\sum_{i=1}^m \ell_i=1} ~
\prod_{i=1}^m  \left | u_i \right |_{B_q^{s\ell_i}(L_{p/\ell_i}(\Omega))},
\end{equation}
where $0 \le \ell_i = m_i/k^{m-1} \le 1$ and
the sum ranges over all choices of $m_i \in \N_0$ such that
$\sum_{i=1}^m m_i = k^{m-1}$.}
\end{corollary}

Using embedding theorems for Besov spaces, \eqref{product-semi} can be further
simplified by replacing the semi-norms by norms.
However, this is at the expense of having a constant depending on $|\Omega|$.

\begin{corollary}[Besov norm of products of functions]\label{C:Besov-product}
  \rhn{Let $s>0$ and $0 < p,q \le \infty$ satisfy $ s > d/p$ (i.e. 
  $B^s_q(L_p(\Omega)) \subset L_\infty(\Omega)$).
  If $u_i\in B^s_q(L_p(\Omega))$ for $1\le i\le m$,
  then $\prod_{i=1}^m u_i \in B^s_q(L_p(\Omega))$ and
  then there exists a constant $C(\Omega)$ such that}
\[
\rhn{
\Big\| \prod_{i=1}^m u_i \Big\|_{B_q^s(L_p(\Omega))} 
\leq C(\Omega)
\prod_{i=1}^m  \left\| u_i \right\|_{B_q^s(L_p(\Omega))} 
}
\]
\end{corollary}
\begin{proof}
Since $s \ell_i - d/(p/\ell_i) = \ell_i (s-d/p) < s-d/p $ for $0 < \ell_i < 1$, we deduce
$B^s_q(L_p(\Omega)) \subset B^{s\ell_i}_{q'}(L_{p/\ell_i}(\Omega))$
for any $0<q,q'\le \infty$. This, together with
$B^s_q(L_p(\Omega)) \subset L_\infty(\Omega)$, enables us to replace
the semi-norms in \eqref{product} by the full Besov norms absorbing the scaling into the constant $C(\Omega)$.
\end{proof}  

\subsection{Composition of Functions}

The following result is a scale-invariant
generalization of a result in~\cite{Bourdaud-Sickel:2011} related to the Besov regularity of the composition of functions.

\begin{comment}
\begin{lemma}[Besov norm of composition]\label{L:composition}
Let $f:\R\to\R$ be $C^k$ and $u:\Omega\to\R$ be of class
$B^s_p(L_p(\Omega)$ for $s>0, 0<p\le\infty$ with $s>d/p$. 
Let $R$ be an interval in $\R$ that contains the range of $u$.
Then the
composite function $f \circ u$ satisfies
%
\[
f \circ u \in B^s_p(L_p(\Omega)
\]
%
and there exists a constant $C(f)$ depending on
$\|f^{j}\|_{L_\infty(R)}$ for $1\le j\le \lfloor u \rfloor +1$
such that
%
\begin{equation}\label{composition}
\|f \circ u \|_{B^s_p(L_p(\Omega)} \le C(f) 
\Big( \| u \|_{B^s_p(L_p(\Omega)} 
+ \| u \|_{B^s_p(L_p(\Omega)}^{\lfloor u \rfloor +1} \Big).
\end{equation}
%
\end{lemma}
\end{comment}

\begin{lemma}[Scale-invariant Besov semi-norm of composition]\label{L:composition-semi}
Let $u:\Omega\to\R$ be of class
$B^s_q(L_p(\Omega))$ with $0<p,q\le\infty$ and $s>d/p$, and let
$R$ be a closed interval in $\R$ that contains the range of $u$. If $f \in C^k(R)$, with $k > s$,
then the composite function $f \circ u \in B^s_q(L_p(\Omega)) $
and there exists a constant $C(f)$ depending on
$\max_{1\le j \le k}\|f^{(j)}\|_{L_\infty(R)}$ % for $1\le j\le \lfloor s \rfloor +1$
such that
\begin{equation}\label{composition-semi}
|f\circ u|_{B_q^s(L_p(\Omega))} \Cleq
C(f) \sum_{\ell=1}^{k} \sum_{i=1}^\ell \| u \|_{L_\infty(\Omega)}^{\ell-i}
\sum_{\sum_{j=1}^i\ell_j=1} \prod_{j=1}^i | u |_{B^{s\ell_j}_q(L_{p/\ell_j}(\Omega))},
\end{equation}
where the inner sum ranges over all choices of \rhn{fractions of the form} $0\le \ell_j = m_j/k^{\rhn{i-1}} \le 1$
with $m_j \in \N_0$ such that $\sum_{j=1}^i m_j = k^{\rhn{i-1}}$.
\end{lemma}
\begin{proof}
Recall formula~\eqref{Delta^k} and notice that $\Delta_h^k 1 = 0$. Then for $x \in \Omega_{kh}$,
\[
\Delta_h^k \big(f\circ u\big)(x) =
\sum_{j=1}^k \binom k j (-1)^{k+j} \big[ f(u(x+jh)) - f(u(x)) \big].
\]
Using Taylor's formula
\begin{align*}
f(u(x+jh))-f(u(x)) ={}& \sum_{\ell=1}^{k-1} \frac{f^{(\ell)}(u(x))}{\ell !} \left(\Delta_{jh}^1 u(x) \right)^\ell \\
&+ \int_0^1 \frac{f^{(k)} \left(u(x)+t \Delta_{jh}^1 u(x)\right)}{(k-1)!} (1-t)^{k-1} dt \left(\Delta_{jh}^1 u(x))\right)^k.
%\\
%=:{}& \sum_{\ell=1}^k I_\ell^j.
\end{align*}
Therefore, $\|\Delta_h^k \big(f\circ u\big)\|_{L_p(\Omega)}
    \le \left\|\sum_{\ell=1}^k I_\ell\right\|_{L_p(\Omega)}$ where
\begin{align*}
I_\ell(x) &:= \sum_{j=1}^{k} \binom k j (-1)^{k+j} 
\frac{f^{(\ell)}(u(x))}{\ell !} \left(\Delta_{jh}^1 u(x) \right)^\ell
\quad 1\le \ell <k,
\\
I_k (x) & := \sum_{j=1}^k \binom k j (-1)^{k+j} 
\int_0^1 \frac{f^{(k)} \left(u(x)+t \Delta_{jh}^1 u(x)\right)}{(k-1)!} (1-t)^{k-1} dt \left(\Delta_{jh}^1 u(x))\right)^k.
\end{align*}
In order to bound the terms corresponding to $\ell < k$ we use Newton's binomial formula:
\begin{align*}
 I_\ell &= \sum_{j=1}^{k} \binom k j (-1)^{k+j} 
 \frac{f^{(\ell)}(u(x))}{\ell !} 
 \sum_{i=0}^\ell \binom \ell i (-1)^{\ell-i} u(x+jh)^i u(x)^{\ell-i}
 \\
 &= \sum_{i=0}^\ell \binom \ell i (-1)^{\ell-i}\frac{f^{(\ell)}(u(x))}{\ell !} \sum_{j=1}^{k} \binom k j (-1)^{k+j} u(x+jh)^i u(x)^{\ell-i}
 \\
 &=  \sum_{i=1}^\ell \binom \ell i (-1)^{\ell-i}\frac{f^{(\ell)}(u(x))}{\ell !} \Delta_h^k u^i(x) u(x)^{\ell-i},
\end{align*}
because $\Delta^k_h u^0 = 0$. Consequently,
\[
\left\| \sum_{\ell=1}^{k-1} I_\ell \right\|_{L_p(\Omega)}
\le C(f) \sum_{\ell=1}^{k-1} \sum_{i=1}^\ell \| \Delta_h^k u^i \|_{L_p(\Omega)} \| u \|_{L_\infty(\Omega)}^{\ell-i}.
\]
A similar formula is valid for $I_k$, whence

\[
\| \Delta_h^k (f\circ u) \|_{L_p(\Omega)}
\le  C(f) \sum_{\ell=1}^{k} \sum_{i=1}^\ell \| \Delta_h^k u^i\|_{L_p(\Omega)}
\| u \|_{L_\infty(\Omega)}^{\ell-i}.
\]
The modulus of smoothness $\omega_k(f\circ u,t)_p$ in turn  satisfies
\begin{align*}
\omega_k(f\circ u,t)_p 
= \sup_{|h|\le t} \| \Delta_h^k f\circ u \|_{L_p(\Omega)} 
\le C(f) \sum_{\ell=1}^{k} \sum_{i=1}^\ell \omega_k(u^i,t)_p
\| u \|_{L_\infty(\Omega)}^{\ell-i}.
\end{align*}
Consequently, the Besov seminorm satisfies
\begin{align*}
|f\circ u|_{B_q^s(L_p(\Omega))}
& := \left( \int_0^\infty t^{-sq} \omega_k(f\circ u,t)_p^q  \frac{dt}t \right)^{1/q} 
\\
& \Cleq C(f) \sum_{\ell=1}^{k} \sum_{i=1}^\ell \left( \int_0^\infty t^{-sp} \omega_k(u^i, t)_p^q \frac{dt}t \right)^{1/q} \| u \|_{L_\infty(\Omega)}^{\ell-i}
\\
& \Cleq C(f) \sum_{\ell=1}^{k} \sum_{i=1}^\ell |u^i|_{B_q^s(L_p(\Omega))}
\| u \|_{L_\infty(\Omega)}^{\ell-i}.
\end{align*}
Employing Corollary \ref{C:Besov-product-seminorm}, we eventually
infer the desired bound \eqref{composition-semi}.
\end{proof}

The inequality \eqref{composition-semi} is scale-invariant and, as such,
can be used at the element level. However, it can be further
simplified   at the expense of having a constant depending on $|\Omega|$. 
\begin{corollary}[Besov norm of composition] \label{C:composition-norm}
Under the assumptions of Lemma \ref{L:composition-semi}, there exists a constant $C(f,\Omega)$ such that
\begin{equation}\label{composition-norm}
\begin{aligned}
|f\circ u|_{B_q^s(L_p(\Omega))} &\Cleq
C(f,\Omega) \sum_{\ell=1}^{k} \| u \|_{B^{s}_q(L_{p}(\Omega))}^\ell
\\
&\le 
C(f,\Omega)\, k \, \max \left\{\| u \|_{B^{s}_q(L_{p}(\Omega))}, \| u \|_{B^{s}_q(L_{p}(\Omega))}^k\right\}.
\end{aligned}\end{equation}
\end{corollary}
\begin{proof}
It suffices to use the embeddings
$B^s_q(L_p(\Omega)) \subset B^{s\ell_j}_q(L_{p/\ell_j}(\Omega)) $ for $0<\ell_j<1$
as well as $B^{s\ell_j}_q(L_{p/\ell_j}(\Omega)) \subset
L_\infty(\Omega)$, which are valid because $s>d/p$, to convert
\eqref{composition-semi} into \eqref{composition-norm}.
\end{proof}

\section*{Glossary}

No glossary available in the arxiv version. 

%\printglossary[type=not]
%\printglossary[type=const]
%\printglossary[type=def]
%\printglossary[type=algo]

%%%%%%%%%%%%%%%%%%%%%%%%%%%%%%%%%%%%%%%%%%%%%%%%%%%%%%%%%%%%%%%%%%%%%%%%%%%%%%
\bibliographystyle{plain}

%%%%%%%%%%%%%%%%%%%%%%%%%%%%%%%%%%%%%%%%%%%%%%%%%%%%%%%%%%%%%%%%%%%%%%%%%%%%%%%
\end{document}